\documentclass[12pt]{amsart}
\usepackage{
   a4wide, 
  amsmath,
  amsfonts,
  amssymb,
  enumerate, 
  graphicx, 
  times,
  color,
  hyphenat, 
  mathabx, 
  pinlabel,
  stmaryrd, datetime, bbm, empheq, thmtools, bbm 
}

\usepackage{tikz,etex,bbm,xspace}
\usepackage{tikz-cd}
\usetikzlibrary{decorations.markings}
\usetikzlibrary{decorations.pathreplacing}
\usetikzlibrary{arrows,shapes,positioning}
\usetikzlibrary{tikzmark}
\tikzstyle directed=[postaction={decorate,decoration={markings,
    mark=at position #1 with {\arrow{>}}}}]
\tikzstyle rdirected=[postaction={decorate,decoration={markings,
    mark=at position #1 with {\arrow{<}}}}]

\usepackage[final]{hyperref, bookmark}
\usepackage{cleveref}
\hypersetup{colorlinks=true, pdfstartview=FitV, linkcolor=blue, citecolor=blue, urlcolor=blue}

\input xy
\xyoption{all}
\definecolor{myblue}{rgb}{0,.5,1}
\definecolor{myred}{rgb}{0.9,0,0}
\definecolor{mygreen}{rgb}{0,0.7,0}

\newcommand{\bV}{\raisebox{0.03cm}{\mbox{\footnotesize$\textstyle{\bigwedge}$}}}

\newcommand{\gln}{{\mathfrak{gl}_{n}}}
\newcommand{\glk}{{\mathfrak{gl}_{k}}}

\newcommand{\sln}{{\mathfrak{sl}_{n}}}
\newcommand{\slk}{{\mathfrak{sl}_{k}}}

\newcommand{\slt}{{\mathfrak{sl}_{2}}}
\newcommand{\glt}{{\mathfrak{gl}_{2}}}

\newcommand{\brak}[1]{\langle #1\rangle}

\newcommand{\n}{\noindent}

\newcommand{\bi}{{\boldsymbol{i}}}
\newcommand{\bj}{{\boldsymbol{j}}}

\DeclareMathOperator{\smod}{\mathrm{-}smod}

\DeclareMathOperator{\Kom}{Kom}

\DeclareMathOperator{\seq}{Seq}
\DeclareMathOperator{\cseq}{CSeq}

\theoremstyle{plain}
\newtheorem{thm}{Theorem}[section]

\newtheorem{lem}[thm]{Lemma}
\newtheorem{prop}[thm]{Proposition}

\theoremstyle{definition}
\newtheorem{rem}[thm]{Remark}

\newtheorem{ex}[thm]{Example}

\theoremstyle{definition}
\newtheorem{defn}[thm]{Definition}

\newcommand{\bN}{\mathbb{N}}
\newcommand{\bZ}{\mathbb{Z}}

\newcommand{\fF}{\mathfrak{F}}

\newcommand{\fR}{\mathfrak{R}}
\newcommand{\fS}{\mathfrak{S}}

\newcommand{\fW}{\mathfrak{W}}

\newcommand{\cC}{\mathcal{C}}

\newcommand{\cH}{\mathcal{H}}

\newcommand{\cR}{\mathcal{R}}

\newcommand{\cU}{\mathcal{U}}

\DeclareMathOperator{\Hom}{Hom}
\DeclareMathOperator{\HOM}{HOM}

\DeclareMathOperator{\id}{Id}

\newcommand{\xra}[1]{\xrightarrow{#1}}

\catcode`\@=11
\long\def\@makecaption#1#2{%
    \vskip 10pt
    \setbox\@tempboxa\hbox{%
\small{#1: }\ignorespaces #2}%
    \ifdim \wd\@tempboxa >\captionwidth {%
        \rightskip=\@captionmargin\leftskip=\@captionmargin
        \unhbox\@tempboxa\par}%
      \else
        \hbox to\hsize{\hfil\box\@tempboxa\hfil}%
    \fi}
\newdimen\@captionmargin\@captionmargin=2\parindent
\newdimen\captionwidth\captionwidth=\hsize
\catcode`\@=12

\let\fullref\autoref
\def\makeautorefname#1#2{\expandafter\def\csname#1autorefname\endcsname{#2}}
\makeautorefname{lem}{Lemma}%
\makeautorefname{prop}{Proposition}%
\makeautorefname{rem}{Remark}%
\makeautorefname{section}{Section}%
\makeautorefname{subsection}{Subsection}%
\makeautorefname{subsubsection}{Subsubsection}%

\title{Not even Khovanov homology}
\author{Pedro Vaz}
\address{Institut de Recherche en Math\'ematique et Physique\\
Universit\'e Catholique de Louvain\\ 
Chemin du Cyclotron 2\\ 
1348 Louvain-la-Neuve\\ 
Belgium}
\email{pedro.vaz@uclouvain.be}
\begin{document}
 \usetikzlibrary{decorations.pathreplacing,backgrounds,decorations.markings}
\tikzset{doubleblack/.style={draw=black,double=black!40!white,double distance=1.5pt,thin}}
\tikzset{doublered/.style={draw=myred,double=myred!40!white,double distance=1.5pt,thin}}
\tikzset{doubleblue/.style={draw=myblue,double=myblue!40!white,double distance=1.5pt,thin}}
\tikzset{doublegreen/.style={draw=mygreen,double=mygreen!40!white,double distance=1.5pt,thin}}
\tikzset{bdot/.style={fill,circle,color=blue,inner sep=3pt,outer sep=0}}
%
\newdimen\captionwidth\captionwidth=\hsize
%
%
\begin{abstract}
  We construct a supercategory that can be seen as a skew version of (thickened) KLR algebras for the type $A$ quiver.  
We use our supercategory to construct homological invariants of tangles and show that 
for every link our invariant gives a link homology theory 
supercategorifying the Jones polynomial. 
Our homology is distinct from even Khovanov homology and
we present evidence supporting the  
conjecture that it is isomorphic to odd Khovanov homology. 
We also show that cyclotomic quotients of our supercategory give supercategorifications 
of irreducible finite-dimensional representations of $\gln$ of level 2. 
\end{abstract}
\maketitle

%
%
\pagestyle{myheadings}
\markboth{\em\small Pedro Vaz}{\em\small Not even Khovanov homology}
%
%
%
\section{Introduction}\label{sec:intro}

After the appearance of odd Khovanov homology in~\cite{ORS} there has been a certain interest in odd
categorified structures and supercategorification (see for
example~\cite{Egilmez-Lauda,EKL,EllisLauda,EllisQi,KKO,KKO2,LaudaRussell,NV-OHn}). 
In contrast to (even) Khovanov homology, odd Khovanov homology has an anticommutative feature. 
Both theories categorify the Jones polynomial and both agree modulo 2, but they are intrinsically distinct   
(see~\cite{shumakovitch} for a study of the properties of odd Khovanov homology and a comparison with even
Khovanov homology).

A construction of odd Khovanov homology using higher representation theory is still missing. 
In the case of even Khovanov homology this question was solved in~\cite{webster} using categorification of
tensor products and the WRT invariant and in \cite{LQR} using categorical Howe duality. 

\smallskip

In this paper we construct a supercategorification of the Jones invariant for tangles 
using higher representation theory. 
In particular, we define a supercategory in the spirit of Khovanov and Lauda's diagrammatics   
that can be seen as a superalgebra version of KLR algebras~\cite{KL1,R1} of level 2 for the $A_n$ quiver. 
We present our supercategory in the form of a graphical calculus reminiscent of the thick calculus for
categorified $\slt$ ~\cite{KLMS} and
$\sln$~\cite{stosic} (see also~\cite{EKL} for a thick calculus for the odd nilHecke algebra).
Our supercategory admits cyclotomic quotients that supercategorify irreducibles of $U_q(\glk)$ of level 2.  

\smallskip 

We use cyclotomic quotients of our supercategories as input to Tubbenhauer's~\cite{tubbenhauer}
approach to Khovanov-Rozansky 
homologies. It is based on $q$-Howe duality and uses only the lower half of the quantum group $U_q(\glk)$
to produce an invariant of tangles. 
In our case we obtain an invariant that shares several similarities with odd Khovanov homology when restricted to links. 
For example, it decomposes as a direct sum of two copies of a reduced homology and it 
produces chronological Frobenius algebras, analogous to the ones that can be
extracted from~\cite{ORS} (see~\cite{putyra} for explanations).  
Both theories coincide over $\bZ/2 \bZ$.   
We also give computational evidence that our invariant is distinct from even Khovanov homology and that support the 
conjecture that for every link $L$ it coincides with the odd Khovanov homology of $L$.

\subsection*{Acknowledgements}
We thank Daniel Tubbenhauer, Kris Putyra and Gr\'egoire Naisse for interesting discussions.
The author was supported by the Fonds de la Recherche Scientifique - FNRS under Grant no. MIS-F.4536.19.

%
%
\section{The supercategory $\fR$}\label{sec:skewKLR}

\subsection{The supercategory $\fR(\nu)$}

We follow~\cite{BrundanEllis-Monoidal} regarding supercategories.  
For objects $X,Y$  in a supercategory $\cC$ we write $\Hom^0_\cC(X,Y)$ (resp. $\Hom^1_\cC(X,Y)$) for its space of
even (resp. odd) morphisms and we write $p(f)$ for the parity of $f\in\Hom^i_\cC(X,Y)$. 
If $\cC$ has additionally a $\bZ$-grading we denote by $q^sX$ a grading shift up of $X$ by $s$ units and we consider only 
morphisms that preserve the $\bZ$-grading. In this case we write $\Hom_\cC(X,Y)=\oplus_{s\in \bZ}\Hom_\cC(X,q^{s}Y)$. 
We follow the grading conventions in~\cite{LQR}, which are aligned with the tradition in link homology.
This means that a map of degree $s$ from $X$ to $Y$ yields a degree zero map from $X$ to $q^sY$. 

\medskip

Fix a unital ring $\Bbbk$. 
Let $\alpha_1,\dotsc,\alpha_n$ denote the simple roots of $\sln$ and $\brak{-,-}$
their inner product: 
$\brak{\alpha_i,\alpha_i}=2$, $\brak{\alpha_i,\alpha_{i\pm1}}=-1$, and $\brak{\alpha_i,\alpha_j}=0$ otherwise.
Fix also a choice of scalars $Q$ consisting of $r_i,t_{ij}\in\Bbbk^\times$ for all $i,j\in I:=\{1,\dotsc,n\}$, such that
$t_{ii}=1$ and $t_{ij}=t_{ji}$ when $\vert i- j\vert\neq 1$. 
Let also $p_{ij}$ be defined by $p_{ii} = p_{i+1,i} = 1$ and otherwise $p_{ij}=0$. 

\medskip

For each $\nu = \sum_{i \in I} \nu_i.i \in \bN_0[I]$, 
we consider the set of (colored) sequences of $\nu$,
\[
\cseq(\nu) :=
\bigl\{ i_1^{(\varepsilon_1)}\dotsm  i_r^{(\varepsilon_r)} \bigr\vert\, \varepsilon_s\in\{1,2\},\sum_s\varepsilon_si_s=\nu\bigr\} . 
\]
By convention we write simply $i_s$ for $i_s^{(1)}$.  Two sequences $\bi\in\cseq(\nu)$ and $\bj\in\cseq(\nu')$
can be concatenated into a sequence $\bi\bj$ in $\cseq(\nu+\nu')$.

\begin{defn}
  The supercategory $\fR(\nu)$ is defined by the following data:
\begin{enumerate}[(a)]
\item The objects of $\fR(\nu)$ are finite formal sums of grading shifts of elements of $\cseq(\nu)$.

\item The morphism space $\Hom_{\fR(\nu)}({\bi},{\bj})$ from ${\bi}$ to ${\bj}$ is the $\bZ$-graded $\Bbbk$-supervector space
  generated by vertical juxtaposition and horizontal juxtaposition of the diagrams below.
  Composition consists of vertical concatenation of diagrams.
  By convention we read diagrams from bottom to top and so, $ab$ consists of stacking the diagram for $a$
  atop the one for $b$. 
Diagrams are equipped with a Morse function that keeps trace of the relative height of the generators. 
We consider isotopy classes of such diagrams that do not change the relative height of generators. 
\end{enumerate}

\subsubsection*{Generators}$\,$\\ 
\n$\bullet$ Simple and double \emph{identities} 
\[
\tikz[very thick,scale=1.5,baseline={([yshift=.8ex]current bounding box.center)}]{
  \draw[mygreen] (1.5,-.5) node[below] {\tiny \textcolor{black}{$i$}}  -- (1.5,.5);
} \in\Hom^0_{\fR(\nu)}(i,i),
\mspace{100mu}
\tikz[very thick,scale=1.5,baseline={([yshift=.8ex]current bounding box.center)}]{
  \draw[doublegreen] (1.5,-.5) node[below] {\tiny $i$}  -- (1.5,.5);
} \in\Hom^0_{\fR(\nu)}(i^{(2)},i^{(2)}) , 
\]
\n$\bullet$ \emph{dots} 
\[
\tikz[very thick,scale=1.5,baseline={([yshift=0ex]current bounding box.center)}]{
  \draw[mygreen] (0,-.5) -- (0,.5)node[midway,fill=mygreen,circle,inner sep=2pt]{};
  \node at (-.2,-.4) {\tiny $i$}; 
} \in\Hom^1_{\fR(\nu)}(i,q^2i) , 
\]
\n$\bullet$ \emph{splitters}
\[
\tikz[very thick,scale=1.5,baseline={([yshift=-.8ex]current bounding box.center)}]{
  \draw[doublegreen] (0,-.4) -- (0,0);
  \draw[mygreen] (-.015,0) .. controls (0,.12) and (-.25,.12) .. (-.3,.5);
  \draw[mygreen] (.015,0) .. controls (0,.12) and (.25,.12) ..  (.3,.5);
\node at (-.2,-.25) {\tiny $i$}; 
}\in\Hom^1_{\fR(\nu)}(i^{(2)},q^{-1}ii) ,
\mspace{80mu}
\tikz[very thick,scale=1.5,baseline={([yshift=-.8ex]current bounding box.center)}]{
  \draw[doublegreen] (0,.4)  -- (0,0);
  \draw[mygreen] (-.015,0) .. controls (0,-.12) and (-.25,-.12) .. (-.3,-.5);
  \draw[mygreen] (.015,0) .. controls (0,-.12) and (.25,-.12) ..  (.3,-.5);
  \node at (-.2,.25) {\tiny $i$};
} \in\Hom^0_{\fR(\nu)}(ii,q^{-1}i^{(2)}) , 
\]
\n$\bullet$ and \emph{crossings}
\begingroup\allowdisplaybreaks
\begin{align*}
\tikz[very thick,baseline={([yshift=.9ex]current bounding box.center)}]{
       \draw[mygreen] (0,-.5) node[below] {\textcolor{black}{\tiny $i$}} .. controls (0,0) and (1,0) .. (1,.5);
       \draw[myblue] (1,-.5) node[below] {\textcolor{black}{\tiny $j$}}  .. controls (1,0) and (0,0) .. (0,.5);
  } &\in\Hom^{p_{ij}}_{\fR(\nu)}(ij,q^{-\brak{\alpha_i,\alpha_j}}ji) , 
&\tikz[very thick,baseline={([yshift=.9ex]current bounding box.center)}]{
       \draw[doublegreen] (0,-.5) node[below] {\tiny $i$} .. controls (0,0) and (1,0) .. (1,.5);
       \draw[myblue] (1,-.5) node[below] {\textcolor{black}{\tiny $j$}}  .. controls (1,0) and (0,0) .. (0,.5);
  } &\in\Hom^0_{\fR(\nu)}(i^{(2)}j,q^{-2\brak{\alpha_i,\alpha_j}}ji^{(2)}) , 
\\
\tikz[very thick,baseline={([yshift=.9ex]current bounding box.center)}]{
       \draw[mygreen] (0,-.5) node[below] {\textcolor{black}{\tiny $i$}} .. controls (0,0) and (1,0) .. (1,.5);
       \draw[doubleblue] (1,-.5) node[below] {\tiny $j$}  .. controls (1,0) and (0,0) .. (0,.5);
  }&\in\Hom^0_{\fR(\nu)}(ij^{(2)},q^{-2\brak{\alpha_i,\alpha_j}}j^{(2)}i) , 
  &\tikz[very thick,baseline={([yshift=.9ex]current bounding box.center)}]{
       \draw[doublegreen] (0,-.5) node[below] {\tiny $i$} .. controls (0,0) and (1,0) .. (1,.5);
       \draw[doubleblue] (1,-.5) node[below] {\tiny $j$}  .. controls (1,0) and (0,0) .. (0,.5);
  } &\in\Hom^0_{\fR(\nu)}(i^{(2)}j^{(2)},q^{-4\brak{\alpha_i,\alpha_j}}j^{(2)}i^{(2)}) . 
\end{align*}
\endgroup


\subsubsection*{Relations}
Morphisms are subject to the local relations~\eqref{eq:chronology} to~\eqref{eq:dpitchforks} below.\\

\n$\bullet$ For all $f,g$:

\begin{equation}\label{eq:chronology}
\tikz[very thick,scale=.6,baseline={([yshift=.7ex]current bounding box.center)}]{
\draw (-1,-.5) -- (1,-.5);\draw (-1, .5) -- (1, .5);
\draw (-1,-.5) -- (-1,.5);\draw (1,-.5) -- (1,.5);
\node at (0,0) {\small $f$};
\draw (-.75,-2.5)node[below]{\tiny $i_1$} -- (-.75,-.5);
\node at (0.05, -2.3) {$\dotsm$};
\draw (.75,-2.5)node[below]{\tiny $i_k$} -- (.75,-.5);
\draw (-.75,1) -- (-.75,.5);
\node at (0.05, .85) {$\dotsm$};
\draw (.75,1) -- (.75,.5);}
\mspace{15mu} 
\tikz[very thick,scale=.6,baseline={([yshift=.8ex]current bounding box.center)}]{
\draw (1.5,-.5) -- (3.5,-.5);\draw (1.5, .5) -- (3.5, .5);
\draw (1.5,-.5) -- (1.5,.5);\draw (3.5,-.5) -- (3.5,.5);
\node at (2.5,0) {\small $g$};
\draw (1.75,-1)node[below]{\tiny $i_1$} -- (1.75,-.5);
\node at (2.55, -.85) {$\dotsm$};
\draw (3.25,-1)node[below]{\tiny $i_k$} -- (3.25,-.5);
\draw (1.75,2.5) -- (1.75,.5);
\node at (2.55, 2.3) {$\dotsm$};
\draw (3.25,2.5) -- (3.25,.5);}
=
\tikz[very thick,scale=.6,baseline={([yshift=.8ex]current bounding box.center)}]{
\draw (-1,-.5) -- (1,-.5);\draw (-1, .5) -- (1, .5);
\draw (-1,-.5) -- (-1,.5);\draw (1,-.5) -- (1,.5);
\node at (0,0) {\small $f$};
\draw (-.75,-1.75)node[below]{\tiny $i_1$} -- (-.75,-.5); \node at (0.05, -1.5) {$\dotsm$};
\draw (.75,-1.75)node[below]{\tiny $i_k$} -- (.75,-.5);
\draw (-.75,1.75) -- (-.75,.5); \node at (0.05, 1.5) {$\dotsm$};
\draw (.75,1.75) -- (.75,.5);}
\mspace{15mu} 
\tikz[very thick,scale=.6,baseline={([yshift=.8ex]current bounding box.center)}]{
\draw (-1,-.5) -- (1,-.5);\draw (-1, .5) -- (1, .5);
\draw (-1,-.5) -- (-1,.5);\draw (1,-.5) -- (1,.5);
\node at (0,0) {\small $g$};
\draw (-.75,-1.75)node[below]{\tiny $i_1$} -- (-.75,-.5); \node at (0.05, -1.5) {$\dotsm$};
\draw (.75,-1.75)node[below]{\tiny $i_k$} -- (.75,-.5);
\draw (-.75,1.75) -- (-.75,.5); \node at (0.05, 1.5) {$\dotsm$};
\draw (.75,1.75) -- (.75,.5);}
= (-1)^{p(f)p(g)}
\tikz[very thick,scale=.6,baseline={([yshift=.8ex]current bounding box.center)}]{
\draw (1.5,-.5) -- (3.5,-.5);\draw (1.5, .5) -- (3.5, .5);
\draw (1.5,-.5) -- (1.5,.5);\draw (3.5,-.5) -- (3.5,.5);
\node at (2.5,0) {\small $f$};
\draw (1.75,-1)node[below]{\tiny $i_1$} -- (1.75,-.5);
\node at (2.55, -.85) {$\dotsm$};
\draw (3.25,-1)node[below]{\tiny $i_k$} -- (3.25,-.5);
\draw (1.75,2.5) -- (1.75,.5);
\node at (2.55, 2.3) {$\dotsm$};
\draw (3.25,2.5) -- (3.25,.5);}
\mspace{15mu}
\tikz[very thick,scale=.6,baseline={([yshift=.8ex]current bounding box.center)}]{
\draw (-1,-.5) -- (1,-.5);\draw (-1, .5) -- (1, .5);
\draw (-1,-.5) -- (-1,.5);\draw (1,-.5) -- (1,.5);
\node at (0,0) {\small $g$};
\draw (-.75,-2.5)node[below]{\tiny $i_1$} -- (-.75,-.5);
\node at (0.05, -2.3) {$\dotsm$};
\draw (.75,-2.5)node[below]{\tiny $i_k$} -- (.75,-.5);
\draw (-.75,1) -- (-.75,.5);
\node at (0.05, .85) {$\dotsm$};
\draw (.75,1) -- (.75,.5);}
\end{equation}

\n$\bullet$ For all $i,j,k\in I$: 
\begin{equation}\label{eq:dotnil}
\tikz[very thick,scale=1.0,baseline={([yshift=.8ex]current bounding box.center)}]{
  \draw[mygreen] (0,-.3)   -- (0,.5)
  node [near start,fill=mygreen,circle,inner sep=2pt]{}
  node [near end,fill=mygreen,circle,inner sep=2pt]{};
  \draw[mygreen] (0,.5) to (0,.65);\draw[mygreen] (0,-.45) node[below] {\textcolor{black}{\tiny $i$}} to (0,-.3);
}
= 0 . 
\end{equation}

\begin{align}\label{eq:klrR2}
\tikz[very thick,xscale=1,yscale=1,baseline={([yshift=.8ex]current bounding box.center)}]{
\draw[mygreen]  +(0,-.75) node[below] {\textcolor{black}{\tiny $i$}} .. controls (0,-.375) and (1,-.375) .. (1,0) .. controls (1,.375) and (0, .375) .. (0,.75);
\draw[myblue]  +(1,-.75) node[below] {\textcolor{black}{\tiny $j$}} .. controls (1,-.375) and (0,-.375) .. (0,0) .. controls (0,.375) and (1, .375) .. (1,.75);
 }
	 \mspace{10mu}=\mspace{10mu}
\begin{cases}
          \mspace{68mu} 0 & \text{ if }i = j,
          \\ \\ 
          \mspace{42mu}
         t_{ij} \tikz[very thick,xscale=1,yscale=0.75,baseline={([yshift=.8ex]current bounding box.center)}]{
	      \draw[mygreen]  (0,-.75) node[below] {\textcolor{black}{\tiny $i$}} -- (0,.75);
	      \draw[myblue]  (1,-.75) node[below] {\textcolor{black}{\tiny $j$}} -- (1,.75);
	 } & \text{ if }\vert i-j \vert > 1,
           \\ \\
         t_{ij} \tikz[very thick,xscale=1,yscale=0.75,baseline={([yshift=.8ex]current bounding box.center)}]{
     \draw[mygreen]  (0,-.75) node[below] {\textcolor{black}{\tiny $i$}} -- (0,.75) node [pos=0.7,fill=mygreen, circle,inner sep=2pt]{};
     \draw[myblue]  (1,-.75) node[below] {\textcolor{black}{\tiny $j$}} -- (1,.75);
	  }
          +\ t_{ji}
           \tikz[very thick,xscale=1,yscale=0.75,baseline={([yshift=.8ex]current bounding box.center)}]{
      \draw[mygreen]  (0,-.75) node[below] {\textcolor{black}{\tiny $i$}} -- (0,.75);
      \draw[myblue]  (1,-.75) node[below] {\textcolor{black}{\tiny $j$}} -- (1,.75) node [pos=0.7,fill=myblue, circle,inner sep=2pt]{};
	  } & \text{ if }\vert i-j\vert = 1,
\end{cases}
\end{align}
\begin{gather}\label{eq:dotslides}
  \allowdisplaybreaks
\tikz[very thick,baseline={([yshift=.9ex]current bounding box.center)}]{
       \draw[mygreen] (0,-.5) node[below] {\textcolor{black}{\tiny $i$}} .. controls (0,0) and (1,0) .. (1,.5);
       \draw[myblue] (1,-.5) node[below] {\textcolor{black}{\tiny $j$}}  .. controls (1,0) and (0,0) .. (0,.5) node [near end,fill=myblue,circle,inner sep=2pt]{};
	}
  \mspace{5mu} =  (-1)^{p_{ij}} 
	\tikz[very thick,baseline={([yshift=.9ex]current bounding box.center)}]{
       \draw[mygreen] (0,-.5) node[below] {\textcolor{black}{\tiny $i$}} .. controls (0,0) and (1,0) .. (1,.5);
       \draw[myblue] (1,-.5) node[below] {\textcolor{black}{\tiny $j$}}  .. controls (1,0) and (0,0) .. (0,.5) node [near start,fill=myblue,circle,inner sep=2pt]{};
	}  
	\mspace{60mu}
	\tikz[very thick,baseline={([yshift=.9ex]current bounding box.center)}]{
       \draw[mygreen] (0,-.5) node[below] {\textcolor{black}{\tiny $i$}} .. controls (0,0) and (1,0) .. (1,.5) node [near start,fill=mygreen,circle,inner sep=2pt]{};
       \draw[myblue] (1,-.5) node[below] {\textcolor{black}{\tiny $j$}}  .. controls (1,0) and (0,0) .. (0,.5);
	}
        \mspace{5mu} = (-1)^{p_{ij}}
	\tikz[very thick,baseline={([yshift=.9ex]current bounding box.center)}]{
       \draw[mygreen] (0,-.5) node[below] {\textcolor{black}{\tiny $i$}} .. controls (0,0) and (1,0) .. (1,.5) node [near end,fill=mygreen,circle,inner sep=2pt]{};
       \draw[myblue] (1,-.5) node[below] {\textcolor{black}{\tiny $j$}}  .. controls (1,0) and (0,0) .. (0,.5);
	} 
\mspace{30mu}\text{ for $i\neq j$,}
\\[3ex]\label{eq:dotjumpneib}
t_{i,i+1}\!\!
\tikz[very thick,baseline={([yshift=.9ex]current bounding box.center)}]{
       \draw[mygreen] (0,-.5) node[below] {\textcolor{black}{\tiny $i+1$}} .. controls (0,0) and (1,0) .. (1,.5);
       \draw[myblue] (1,-.5) node[below] {\textcolor{black}{\tiny $i$}}  .. controls (1,0) and (0,0) .. (0,.5) node [near end,fill=myblue,circle,inner sep=2pt]{};
	}
  \mspace{5mu} + \mspace{5mu} 
  t_{i+1,i}\!\!
  \tikz[very thick,baseline={([yshift=.9ex]current bounding box.center)}]{
       \draw[mygreen] (0,-.5) node[below] {\textcolor{black}{\tiny $i+1$}} .. controls (0,0) and (1,0) .. (1,.5) node [near end,fill=mygreen,circle,inner sep=2pt]{};
       \draw[myblue] (1,-.5) node[below] {\textcolor{black}{\tiny $i$}}  .. controls (1,0) and (0,0) .. (0,.5);
	}
          \mspace{5mu} = \mspace{5mu} 0 
\\[3ex]\label{eq:dotslide-nilH}
 	 \tikz[very thick,baseline={([yshift=.9ex]current bounding box.center)}]{
       \draw[mygreen] (0,-.5) node[below] {\textcolor{black}{\tiny $i$}} .. controls (0,0) and (1,0) .. (1,.5);
       \draw[mygreen] (1,-.5) node[below] {\textcolor{black}{\tiny $i$}}  .. controls (1,0) and (0,0) .. (0,.5) node [near end,fill=mygreen,circle,inner sep=2pt]{};
	}
	\mspace{10mu}+\mspace{10mu} 
	\tikz[very thick,baseline={([yshift=.9ex]current bounding box.center)}]{
       \draw[mygreen] (0,-.5) node[below] {\textcolor{black}{\tiny $i$}} .. controls (0,0) and (1,0) .. (1,.5);
       \draw[mygreen] (1,-.5) node[below] {\textcolor{black}{\tiny $i$}}  .. controls (1,0) and (0,0) .. (0,.5) node [near start,fill=mygreen,circle,inner sep=2pt]{};
	}
	\mspace{10mu}=\mspace{10mu}r_i
	\tikz[very thick,baseline={([yshift=.9ex]current bounding box.center)}]{
        \draw[mygreen] (0,-.5) node[below] {\textcolor{black}{\tiny $i$}} -- (0,.5);
        \draw[mygreen] (1,-.5) node[below] {\textcolor{black}{\tiny $i$}} -- (1,.5);
	}  
	\mspace{10mu}=\mspace{10mu}
	\tikz[very thick,baseline={([yshift=.9ex]current bounding box.center)}]{
       \draw[mygreen] (0,-.5) node[below] {\textcolor{black}{\tiny $i$}} .. controls (0,0) and (1,0) .. (1,.5) node [near start,fill=mygreen,circle,inner sep=2pt]{};
       \draw[mygreen] (1,-.5) node[below] {\textcolor{black}{\tiny $i$}}  .. controls (1,0) and (0,0) .. (0,.5);
	} 
	\mspace{10mu}+\mspace{10mu}
	\tikz[very thick,baseline={([yshift=.9ex]current bounding box.center)}]{
       \draw[mygreen] (0,-.5) node[below] {\textcolor{black}{\tiny $i$}} .. controls (0,0) and (1,0) .. (1,.5) node [near end,fill=mygreen,circle,inner sep=2pt]{};
       \draw (1,-.5)[mygreen] node[below] {\textcolor{black}{\tiny $i$}}  .. controls (1,0) and (0,0) .. (0,.5);
	}
\end{gather}

\begin{gather}\label{eq:klrR3}
  \allowdisplaybreaks
	\tikz[very thick,baseline={([yshift=.9ex]current bounding box.center)}]{
		\draw[mygreen]  +(0,0)node[below] {\textcolor{black}{\tiny $i$}} .. controls (0,0.5) and (2, 1) ..  +(2,2);
		\draw[myblue]  +(2,0)node[below] {\textcolor{black}{\tiny $k$}} .. controls (2,1) and (0, 1.5) ..  +(0,2);
		\draw[myred]  (1,0)node[below] {\textcolor{black}{\tiny $j$}} .. controls (1,0.5) and (0, 0.5) ..  (0,1) .. controls (0,1.5) and (1, 1.5) ..  (1,2);
        }
 	\mspace{5mu}=(-1)^{p_{jk}p_{ik}+p_{jk}p_{ij}+p_{ik}p_{ij}}
	\tikz[very thick,baseline={([yshift=.9ex]current bounding box.center)}]{
		\draw[mygreen]  +(0,0)node[below] {\textcolor{black}{\tiny $i$}} .. controls (0,1) and (2, 1.5) ..  +(2,2);
		\draw[myblue]  +(2,0)node[below] {\textcolor{black}{\tiny $k$}} .. controls (2,.5) and (0, 1) ..  +(0,2);
		\draw[myred]  (1,0)node[below]{\textcolor{black} {\tiny $j$}} .. controls (1,0.5) and (2, 0.5) ..  (2,1) .. controls (2,1.5) and (1, 1.5) ..  (1,2);
        }
         \mspace{40mu}\text{ unless }i=k\text{ and }\vert i-j\vert = 1 ,
         \\[3ex]\label{eq:R3serre}
         \tikz[very thick,baseline={([yshift=.9ex]current bounding box.center)}]{
		\draw[mygreen]  +(0,0)node[below] {\textcolor{black}{\tiny $i$}} .. controls (0,0.5) and (2, 1) ..  +(2,2);
		\draw[mygreen]  +(2,0)node[below] {\textcolor{black}{\tiny $i$}} .. controls (2,1) and (0, 1.5) ..  +(0,2);
		\draw[myblue]  (1,0)node[below] {\textcolor{black}{\tiny $j$}} .. controls (1,0.5) and (0, 0.5) ..  (0,1) .. controls (0,1.5) and (1, 1.5) ..  (1,2);
        }
 	\mspace{10mu}+\mspace{10mu}
	 \tikz[very thick,baseline={([yshift=.9ex]current bounding box.center)}]{
		\draw[mygreen]  +(0,0)node[below] {\textcolor{black}{\tiny $i$}} .. controls (0,1) and (2, 1.5) ..  +(2,2);
		\draw[mygreen]  +(2,0)node[below] {\textcolor{black}{\tiny $i$}} .. controls (2,.5) and (0, 1) ..  +(0,2);
		\draw[myblue]  (1,0)node[below] {\textcolor{black}{\tiny $j$}} .. controls (1,0.5) and (2, 0.5) ..  (2,1) .. controls (2,1.5) and (1, 1.5) .. (1,2);
	 }
         \mspace{10mu}=\mspace{10mu}
	 r_it_{ij}\ \tikz[very thick,baseline={([yshift=.9ex]current bounding box.center)}]{
	 \draw[mygreen] (  0,0) node[below] {\textcolor{black}{\tiny $i$}} -- (0,2);
    	 \draw[myblue] (1,0) node[below] {\textcolor{black}{\tiny $j$}} -- (1,2);
         \draw[mygreen] (2,0) node[below] {\textcolor{black}{\tiny $i$}} -- (2,2);
	 }
         \mspace{50mu}\text{ if }\vert i-j\vert = 1 ,
\end{gather}

\begin{equation}\label{eq:dumbelXing}
\tikz[very thick,scale=1.5,baseline={([yshift=.8ex]current bounding box.center)}]{
  \draw[doublegreen] (0,-.2) -- (0,0);
  \draw[mygreen] (-.015,0) .. controls (0,.12) and (-.25,.12) .. (-.3,.5);
  \draw[mygreen] (.015,0) .. controls (0,.12) and (.25,.12) ..  (.3,.5);
  \draw[mygreen] (-.015,-.2) .. controls (0,-.32) and (-.25,-.32) .. (-.3,-.7)  node[below] {\textcolor{black}{\tiny $j$}} ;
  \draw[mygreen] (.015,-.2) .. controls (0,-.32) and (.25,-.32) ..  (.3,-.7)  node[below] {\textcolor{black}{\tiny $j$}} ;
}
=
\tikz[very thick,scale=1.5,baseline={([yshift=.8ex]current bounding box.center)}]{
 \draw[mygreen] (0.15,-.7) node[below] {\textcolor{black}{\tiny $j$}}  .. controls (0.25,-.1) and (.75,-.1) ..  (.85,.5);
 \draw[mygreen] (.85,-.7) node[below] {\textcolor{black}{\tiny $j$}}  .. controls (.75,-.1) and (.25,-.1) ..  (.15,.5);
 	}
\end{equation}

\begin{equation} \label{eq:digons}
\tikz[very thick,scale=1,baseline={([yshift=.8ex]current bounding box.center)}]{
  \draw[doublegreen] (0,-1)node[below] {\textcolor{black}{\tiny $j$}} -- (0,-.5);
  \draw[doublegreen] (0,.5) -- (0,1);
  \draw[mygreen] (-.015,-.5) .. controls (-.4,-.2) and (-.4,.2) .. (-.015,.5)node [midway,fill=mygreen,circle,inner sep=2pt]{};
  \draw[mygreen] (.015,-.5) .. controls (.4,-.2) and (.4,.2) ..  (.015,.5);
}
=
\tikz[very thick,scale=1,baseline={([yshift=.8ex]current bounding box.center)}]{
  \draw[doublegreen] (0,-1)node[below] {\textcolor{black}{\tiny $j$}} -- (0,1);
}
=
\tikz[very thick,scale=1,baseline={([yshift=.8ex]current bounding box.center)}]{
  \draw[doublegreen] (0,-1)node[below] {\textcolor{black}{\tiny $j$}} -- (0,-.5);
  \draw[doublegreen] (0,.5) -- (0,1);
  \draw[mygreen] (-.015,-.5) .. controls (-.4,-.2) and (-.4,.2) .. (-.015,.5);
  \draw[mygreen] (.015,-.5) .. controls (.4,-.2) and (.4,.2) ..  (.015,.5) node [midway,fill=mygreen,circle,inner sep=2pt]{};
}
\mspace{120mu} 
  \tikz[very thick,scale=1,baseline={([yshift=.8ex]current bounding box.center)}]{
  \draw[doublegreen] (0,-1)node[below] {\textcolor{black}{\tiny $j$}} -- (0,-.5);
  \draw[doublegreen] (0,.5) -- (0,1);
  \draw[mygreen] (-.015,-.5) .. controls (-.4,-.2) and (-.4,.2) .. (-.015,.5);
  \draw[mygreen] (.015,-.5) .. controls (.4,-.2) and (.4,.2) ..  (.015,.5);
}
= 0
\end{equation}

\begin{equation}
\tikz[very thick,scale=1,baseline={([yshift=.8ex]current bounding box.center)}]{
  \draw[doublegreen] (0,-1)node[below] {\tiny $j$} -- (0,-.5);
  \draw[mygreen] (-.015,-.5) .. controls (-.4,-.2) and (-.4,.2) .. (.35,1);
  \draw[mygreen] (.015,-.5) .. controls (.4,-.2) and (.4,.2) ..  (-.35,1);
}
= 0 =
\tikz[very thick,scale=1,baseline={([yshift=.8ex]current bounding box.center)}]{
  \draw[doublegreen] (0,1) -- (0,.5);
  \draw[mygreen] (-.015,.5) .. controls (-.4,.2) and (-.4,-.2) .. (.35,-1)node[below] {\textcolor{black}{\tiny $j$}};
  \draw[mygreen] (.015,.5) .. controls (.4,.2) and (.4,-.2) ..  (-.35,-1) node[below] {\textcolor{black}{\tiny $j$}};
}
\end{equation}

\begingroup\allowdisplaybreaks
\begin{align}\label{eq:lpitchforks}
\tikz[very thick,scale=1.5,baseline={([yshift=.8ex]current bounding box.center)}]{
  \draw[doublegreen] (1,-.5)node[below] {\tiny $k$} -- (0.35,.15);
  \draw[mygreen] (0.35,.15)  .. controls (.2,.2) and (.1,.2) ..  (0,.3);
  \draw[mygreen] (0.35,.15)  .. controls (.3,.3) and (.3,.4) ..  (.2,.5);
  \draw[myblue] (0,-.5) node[below] {\textcolor{black}{\tiny $j$}}   .. controls (0,0) and (1,0) ..  (1,.5);
 	}
&=
\tikz[very thick,scale=1.5,baseline={([yshift=.8ex]current bounding box.center)}]{
  \draw[doublegreen] (1,-.5)node[below] {\tiny $k$} -- (0.7,-.2);
  \draw[mygreen] (0.7,-.2)  .. controls (.1,.1) and (.1,.2) ..  (0,.3);
  \draw[mygreen] (0.7,-.2)  .. controls (.4,.4) and (.3,.4) ..  (.2,.5);
  \draw[myblue] (0,-.5) node[below] {\textcolor{black}{\tiny $j$}}   .. controls (0,0) and (1,0) ..  (1,.5);
 	}
&
\tikz[very thick,scale=1.5,baseline={([yshift=.8ex]current bounding box.center)}]{
  \draw[doublegreen] (-1,.5) -- (-0.35,-.15);
  \draw[mygreen] (-0.35,-.15)  .. controls (-.2,-.2) and (-.1,-.2) ..  (0,-.3);
  \draw[mygreen] (-0.35,-.15)  .. controls (-.3,-.3) and (-.3,-.4) ..  (-.2,-.5)node[below] {\textcolor{black}{\tiny $k$}};
  \draw[myblue] (0,.5) .. controls (0,0) and (-1,0) ..  (-1,-.5) node[below] {\textcolor{black}{\tiny $j$}};
 	}
&=
\tikz[very thick,scale=1.5,baseline={([yshift=.8ex]current bounding box.center)}]{
  \draw[doublegreen] (-1,.5) -- (-0.7,.2);
  \draw[mygreen] (-0.7,.2)  .. controls (-.1,-.1) and (-.1,-.2) ..  (0,-.3);
  \draw[mygreen] (-0.7,.2)  .. controls (-.4,-.4) and (-.3,-.4) ..  (-.2,-.5) node[below] {\textcolor{black}{\tiny $k$}};
  \draw[myblue] (0,.5)  .. controls (0,0) and (-1,0) ..  (-1,-.5)node[below] {\textcolor{black}{\tiny $j$}};
 	}
\\ \label{eq:rpitchforks}
\tikz[very thick,scale=1.5,baseline={([yshift=.8ex]current bounding box.center)}]{
  \draw[doublegreen] (1,.5) -- (0.35,-.15);
  \draw[mygreen] (0.35,-.15)  .. controls (.2,-.2) and (.1,-.2) ..  (0,-.3);
  \draw[mygreen] (0.35,-.15)  .. controls (.3,-.3) and (.3,-.4) ..  (.2,-.5)node[below] {\textcolor{black}{\tiny $k$}};
  \draw[myblue] (0,.5)  .. controls (0,0) and (1,0) ..  (1,-.5) node[below] {\textcolor{black}{\tiny $j$}};
 	}
&=
\tikz[very thick,scale=1.5,baseline={([yshift=.8ex]current bounding box.center)}]{
  \draw[doublegreen] (1,.5) -- (0.7,.2);
  \draw[mygreen] (0.7,.2)  .. controls (.1,-.1) and (.1,-.2) ..  (0,-.3);
  \draw[mygreen] (0.7,.2)  .. controls (.4,-.4) and (.3,-.4) ..  (.2,-.5)node[below] {\textcolor{black}{\tiny $k$}};
  \draw[myblue] (0,.5) .. controls (0,0) and (1,0) ..  (1,-.5)node[below] {\textcolor{black}{\tiny $j$}};
 	}
&
\tikz[very thick,scale=1.5,baseline={([yshift=.8ex]current bounding box.center)}]{
  \draw[doublegreen] (-1,-.5)node[below] {\tiny $k$} -- (-0.35,.15);
  \draw[mygreen] (-0.35,.15)  .. controls (-.2,.2) and (-.1,.2) ..  (0,.3);
  \draw[mygreen] (-0.35,.15)  .. controls (-.3,.3) and (-.3,.4) ..  (-.2,.5);
  \draw[myblue] (0,-.5) node[below] {\textcolor{black}{\tiny $j$}} .. controls (0,0) and (-1,0) ..  (-1,.5);
 	}
&=
\tikz[very thick,scale=1.5,baseline={([yshift=.8ex]current bounding box.center)}]{
  \draw[doublegreen] (-1,-.5)node[below] {\tiny $k$} -- (-0.7,-.2);
  \draw[mygreen] (-0.7,-.2)  .. controls (-.1,.1) and (-.1,.2) ..  (0,.3);
  \draw[mygreen] (-0.7,-.2)  .. controls (-.4,.4) and (-.3,.4) ..  (-.2,.5);
  \draw[myblue] (0,-.5) node[below] {\textcolor{black}{\tiny $j$}} .. controls (0,0) and (-1,0) ..  (-1,.5);
 	}
\\ \label{eq:dpitchforks}
\tikz[very thick,scale=1.5,baseline={([yshift=.8ex]current bounding box.center)}]{
  \draw[doublegreen] (1,-.5)node[below] {\tiny $k$} -- (0.35,.15);
  \draw[mygreen] (0.35,.15)  .. controls (.2,.2) and (.1,.2) ..  (0,.3);
  \draw[mygreen] (0.35,.15)  .. controls (.3,.3) and (.3,.4) ..  (.2,.5);
  \draw[doubleblue] (0,-.5) node[below] {\tiny $j$}   .. controls (0,0) and (1,0) ..  (1,.5);
 	}
&=
\tikz[very thick,scale=1.5,baseline={([yshift=.8ex]current bounding box.center)}]{
  \draw[doublegreen] (1,-.5)node[below] {\tiny $k$} -- (0.7,-.2);
  \draw[mygreen] (0.7,-.2)  .. controls (.1,.1) and (.1,.2) ..  (0,.3);
  \draw[mygreen] (0.7,-.2)  .. controls (.4,.4) and (.3,.4) ..  (.2,.5);
  \draw[doubleblue] (0,-.5) node[below] {\tiny $j$}   .. controls (0,0) and (1,0) ..  (1,.5);
 	}
&
\tikz[very thick,scale=1.5,baseline={([yshift=.8ex]current bounding box.center)}]{
  \draw[doublegreen] (-1,.5) -- (-0.35,-.15);
  \draw[mygreen] (-0.35,-.15)  .. controls (-.2,-.2) and (-.1,-.2) ..  (0,-.3);
  \draw[mygreen] (-0.35,-.15)  .. controls (-.3,-.3) and (-.3,-.4) ..  (-.2,-.5)node[below] {\textcolor{black}{\tiny $k$}};
  \draw[doubleblue] (0,.5) .. controls (0,0) and (-1,0) ..  (-1,-.5) node[below] {\tiny $j$};
 	}
&=
\tikz[very thick,scale=1.5,baseline={([yshift=.8ex]current bounding box.center)}]{
  \draw[doublegreen] (-1,.5) -- (-0.7,.2);
  \draw[mygreen] (-0.7,.2)  .. controls (-.1,-.1) and (-.1,-.2) ..  (0,-.3);
  \draw[mygreen] (-0.7,.2)  .. controls (-.4,-.4) and (-.3,-.4) ..  (-.2,-.5) node[below] {\textcolor{black}{\tiny $k$}};
  \draw[doubleblue] (0,.5)  .. controls (0,0) and (-1,0) ..  (-1,-.5)node[below] {\tiny $j$};
 	}
\end{align}    
\endgroup

This ends the definition of $\fR(\nu)$. 
\end{defn}

In~\fullref{sec:polyaction} below we show that $\fR(\nu)$ acts on a supercommutative ring.

\begin{defn}
We define the monoidal supercategory 
\[
\fR = \bigoplus\limits_{\nu\in\bN_0[I]}\fR(\nu) , 
\]
the monoidal structure given by horizontal composition of diagrams.  
\end{defn}

\subsection{Further relations in $\fR(\nu)$}
We have several consequences of the defining relations. 

\begin{lem}
 For all $i\in I$, 
\begingroup\allowdisplaybreaks
 \begin{align}\label{eq:dotjumps}
\tikz[very thick,baseline={([yshift=.9ex]current bounding box.center)}]{
        \draw[mygreen] (0,-.5) node[below] {\textcolor{black}{\tiny $i$}} -- (0,.5)node [midway,fill=mygreen,circle,inner sep=2pt]{};
        \draw[mygreen] (1,-.5) node[below] {\textcolor{black}{\tiny $i$}} -- (1,.5);
	}  
	\mspace{10mu}-\mspace{10mu}
	\tikz[very thick,baseline={([yshift=.9ex]current bounding box.center)}]{
        \draw[mygreen] (0,-.5) node[below] {\textcolor{black}{\tiny $i$}} -- (0,.5);
        \draw[mygreen] (1,-.5) node[below] {\textcolor{black}{\tiny $i$}} -- (1,.5)node [midway,fill=mygreen,circle,inner sep=2pt]{};
        } &=\ 0 , 
\\ \label{eq:threestrands}
\tikz[very thick,baseline={([yshift=.9ex]current bounding box.center)}]{
	 \draw[mygreen] (0,-.6) node[below] {\textcolor{black}{\tiny $i$}} -- (0,.6);
    	 \draw[mygreen] (.75,-.6) node[below] {\textcolor{black}{\tiny $i$}} -- (.75,.6);
         \draw[mygreen] (1.5,-.6) node[below] {\textcolor{black}{\tiny $i$}} -- (1.5,.6);
	 } &=\ 0 ,
\\ \label{eq:morethanthreestrands}
\tikz[very thick,scale=1.5,baseline={([yshift=.8ex]current bounding box.center)}]{
  \draw[doublegreen] (.5,-.5) node[below] {\tiny $i$}  -- (.5,.5);
  \draw[mygreen] (0,-.5) node[below] {\textcolor{black}{\tiny $i$}}  -- (0,.5);
	}
=
\tikz[very thick,scale=1.5,baseline={([yshift=.8ex]current bounding box.center)}]{
  \draw[doublegreen] (0,-.5) node[below] {\tiny $i$}  -- (0,.5);
  \draw[mygreen] (.5,-.5) node[below] {\textcolor{black}{\tiny $i$}}  -- (.5,.5);
	}
=
\tikz[very thick,scale=1.5,baseline={([yshift=.8ex]current bounding box.center)}]{
  \draw[doublegreen] (.5,-.5) node[below] {\tiny $i$}  -- (.5,.5);
  \draw[doublegreen] (0,-.5) node[below] {\tiny $i$}  -- (0,.5);
}
&=\ 0 . 
 \end{align}
 \endgroup
\end{lem}

\begin{proof}
By~\eqref{eq:dotnil} and~\eqref{eq:dotslide-nilH}, 
  \[
       r_i^{-1}\tikz[very thick,baseline={([yshift=.9ex]current bounding box.center)}]{
       \draw[mygreen] (0,-.5) node[below] {\textcolor{black}{\tiny $i$}} .. controls (0,0) and (1,0) .. (1,.5);
       \draw[mygreen] (1,-.5) node[below] {\textcolor{black}{\tiny $i$}}  .. controls (1,0) and (0,0) .. (0,.5);
       \node[fill=mygreen,circle,inner sep=2pt] at (0.06,.3) {};
       \node[fill=mygreen,circle,inner sep=2pt] at (0.28,.12) {};
	}
	\mspace{10mu}-\mspace{10mu}
	 r_i^{-1}\tikz[very thick,baseline={([yshift=.9ex]current bounding box.center)}]{
       \draw[mygreen] (0,-.5) node[below] {\textcolor{black}{\tiny $i$}} .. controls (0,0) and (1,0) .. (1,.5);
       \draw[mygreen] (1,-.5) node[below] {\textcolor{black}{\tiny $i$}}  .. controls (1,0) and (0,0) .. (0,.5);
       \node[fill=mygreen,circle,inner sep=2pt] at (0.94,-.32) {};
       \node[fill=mygreen,circle,inner sep=2pt] at (0.76,-.14) {};
  	} 
    \mspace{10mu}=\mspace{10mu} 
	\tikz[very thick,baseline={([yshift=.9ex]current bounding box.center)}]{
        \draw[mygreen] (0,-.5) node[below] {\textcolor{black}{\tiny $i$}} -- (0,.5)node [midway,fill=mygreen,circle,inner sep=2pt]{};;
        \draw[mygreen] (1,-.5) node[below] {\textcolor{black}{\tiny $i$}} -- (1,.5);
	}
	\mspace{10mu}-\mspace{10mu}
        \tikz[very thick,baseline={([yshift=.9ex]current bounding box.center)}]{
        \draw[mygreen] (0,-.5) node[below] {\textcolor{black}{\tiny $i$}} -- (0,.5);
        \draw[mygreen] (1,-.5) node[below] {\textcolor{black}{\tiny $i$}} -- (1,.5)node [midway,fill=mygreen,circle,inner sep=2pt]{};;
	}\mspace{10mu}=\mspace{10mu} 0 , 
        \]
which proves~\eqref{eq:dotjumps}. 

Also,
\begingroup\allowdisplaybreaks
\begin{align*}
\tikz[very thick,xscale=1,yscale=1,baseline={([yshift=.8ex]current bounding box.center)}]{
\draw[mygreen]  (-.75,-.75) node[below] {\textcolor{black}{\tiny $i$}} -- (-.75,.75);
\draw[mygreen]  (0,-.75) node[below] {\textcolor{black}{\tiny $i$}} -- (0,.75);
\draw[mygreen]  (.75,-.75) node[below] {\textcolor{black}{\tiny $i$}} -- (.75,.75);
	 }
\mspace{10mu} &= \mspace{10mu}
\tikz[very thick,xscale=1,yscale=1,baseline={([yshift=.8ex]current bounding box.center)}]{
\draw[mygreen]  (-.75,-.75) node[below] {\textcolor{black}{\tiny $i$}} -- (-.75,.75);
\draw[mygreen] (0,-.75) node[below] {\textcolor{black}{\tiny $i$}} .. controls (0,0) and (.75,0) .. (.75,.75)node [near start,fill=mygreen,circle,inner sep=2pt]{};;
\draw[mygreen] (.75,-.75) node[below] {\textcolor{black}{\tiny $i$}}  .. controls (.75,0) and (0,0) .. (0,.75);
}
\mspace{10mu}+\mspace{10mu}
\tikz[very thick,xscale=1,yscale=1,baseline={([yshift=.8ex]current bounding box.center)}]{
\draw[mygreen]  (-.75,-.75) node[below] {\textcolor{black}{\tiny $i$}} -- (-.75,.75);
\draw[mygreen] (0,-.75) node[below] {\textcolor{black}{\tiny $i$}} .. controls (0,0) and (.75,0) .. (.75,.75)node [near end,fill=mygreen,circle,inner sep=2pt]{};;
\draw[mygreen] (.75,-.75) node[below] {\textcolor{black}{\tiny $i$}}  .. controls (.75,0) and (0,0) .. (0,.75);
  }
\mspace{10mu}=\mspace{10mu}
\tikz[very thick,xscale=1,yscale=1,baseline={([yshift=.8ex]current bounding box.center)}]{
\draw[mygreen]  (-.75,-.75) node[below] {\textcolor{black}{\tiny $i$}} -- (-.75,.75);
  \draw[mygreen]  +(0,-.75) node[below] {\textcolor{black}{\tiny $i$}} .. controls (0,-.375) and (.75,-.375) .. (.75,0)node [near start,fill=mygreen,circle,inner sep=2pt]{} .. controls (.75,.375) and (0, .375) .. (0,.75);
  \draw[mygreen]  +(.75,-.75) node[below] {\textcolor{black}{\tiny $i$}} .. controls (.75,-.375) and (0,-.375) .. (0,0) .. controls (0,.375) and (.75, .375) .. (.75,.75);
    \node[fill=mygreen,circle,inner sep=2pt] at (0,0) {};
	 }
\mspace{10mu}+\mspace{10mu}
\tikz[very thick,xscale=1,yscale=1,baseline={([yshift=.8ex]current bounding box.center)}]{
\draw[mygreen]  (-.75,-.75) node[below] {\textcolor{black}{\tiny $i$}} -- (-.75,.75);
  \draw[mygreen]  +(0,-.75) node[below] {\textcolor{black}{\tiny $i$}} .. controls (0,-.375) and (.75,-.375) .. (.75,0) .. controls (.75,.375) and (0, .375) .. (0,.75);
  \draw[mygreen]  +(.75,-.75) node[below] {\textcolor{black}{\tiny $i$}} .. controls (.75,-.375) and (0,-.375) .. (0,0) .. controls (0,.375) and (.75, .375) .. (.75,.75)node [near end,fill=mygreen,circle,inner sep=2pt]{};
  \node[fill=mygreen,circle,inner sep=2pt] at (0,0) {};
}
\\
& = \mspace{10mu}
\tikz[very thick,xscale=1,yscale=1,baseline={([yshift=.8ex]current bounding box.center)}]{
  \draw[mygreen]  (-.75,-.75) node[below] {\tiny $i$} -- (-.75,.75)
   node[midway, fill=mygreen,circle,inner sep=2pt]{};;
  \draw[mygreen]  +(0,-.75) node[below] {\tiny $i$} .. controls (0,-.375) and (.75,-.375) .. (.75,0)node [near start,fill=mygreen,circle,inner sep=2pt]{} .. controls (.75,.375) and (0, .375) .. (0,.75);
  \draw[mygreen]  +(.75,-.75) node[below] {\tiny $i$} .. controls (.75,-.375) and (0,-.375) .. (0,0) .. controls (0,.375) and (.75, .375) .. (.75,.75);
	 }
\mspace{10mu}+\mspace{10mu}
\tikz[very thick,xscale=1,yscale=1,baseline={([yshift=.8ex]current bounding box.center)}]{
\draw[mygreen]  (-.75,-.75) node[below] {\textcolor{black}{\tiny $i$}} -- (-.75,.75)
   node[midway, fill=mygreen,circle,inner sep=2pt]{};;
\draw[mygreen]  +(0,-.75) node[below] {\textcolor{black}{\tiny $i$}} .. controls (0,-.375) and (.75,-.375) .. (.75,0) .. controls (.75,.375) and (0, .375) .. (0,.75);
  \draw[mygreen]  +(.75,-.75) node[below] {\textcolor{black}{\tiny $i$}} .. controls (.75,-.375) and (0,-.375) .. (0,0) .. controls (0,.375) and (.75, .375) .. (.75,.75)node [near end,fill=mygreen,circle,inner sep=2pt]{};
}\mspace{10mu}=\mspace{10mu} 0 ,
\end{align*}
\endgroup
and this proves~\eqref{eq:threestrands}. 
Relations~\eqref{eq:morethanthreestrands} are an easy consequence of~\eqref{eq:digons} together with~\eqref{eq:threestrands}. 
\end{proof}

\begin{lem}\label{lem:dotjumpdlb}
For all $i,j\in I$ with $\vert i-j\vert =1$,
\[
\tikz[very thick,scale=1.2,baseline={([yshift=.8ex]current bounding box.center)}]{
  \draw[mygreen] (.4,-.4) node[below] {\textcolor{black}{\tiny $i$}} --  (.4,.4);
  \draw[mygreen] (-.4,-.4)node[below] {\textcolor{black}{\tiny $i$}}  --  (-.4,.4)node[midway,fill=mygreen,circle,inner sep=2pt]{};
  \draw[doubleblue] (0,-.4)node[below] {\tiny $j$}  -- (0,.4);}
=
\tikz[very thick,scale=1.2,baseline={([yshift=.8ex]current bounding box.center)}]{
  \draw[mygreen] (.4,-.4) node[below] {\textcolor{black}{\tiny $i$}}  --  (.4,.4)node[midway,fill=mygreen,circle,inner sep=2pt]{};
  \draw[mygreen] (-.4,-.4)node[below] {\textcolor{black}{\tiny $i$}}  --  (-.4,.4);
  \draw[doubleblue] (0,-.4)node[below] {\textcolor{black}{\tiny $j$}}  -- (0,.4);}
\]
\end{lem}
\begin{proof}
  Start from the equality
  \[
  \tikz[very thick,scale=1.25,baseline={([yshift=.8ex]current bounding box.center)}]{
  \draw[doubleblue] (0,-.75)node[below] {\textcolor{black}{\tiny $j$}} -- (0,-.5);
  \draw[doubleblue] (0,.5) -- (0,.75);
  \draw[myblue] (-.02,-.5) .. controls (-.45,-.2) and (-.45,.2) .. (-.02,.5);
  \draw[myblue] (.02,-.5) .. controls (.45,-.2) and (.45,.2) ..  (.02,.5);
  \draw[mygreen] (-.5,-.75) node[below] {\textcolor{black}{\tiny $i$}}   .. controls (-.4,-.2) and (.5,0) ..  (.5,.75);
  \draw[mygreen] (.5,-.75) node[below] {\textcolor{black}{\tiny $i$}}  .. controls (.4,.2) and (-.5,0) ..  (-.5,.75);
  \node[fill=mygreen,circle,inner sep=1.8pt] at (-.075,.09) {};
  }
  =
  \tikz[very thick,scale=1.25,baseline={([yshift=.8ex]current bounding box.center)}]{
  \draw[doubleblue] (0,-.75)node[below] {\tiny $j$} -- (0,-.5);
  \draw[doubleblue] (0,.5) -- (0,.75);
  \draw[myblue] (-.02,-.5) .. controls (-.45,-.2) and (-.45,.2) .. (-.02,.5);
  \draw[myblue] (.02,-.5) .. controls (.45,-.2) and (.45,.2) ..  (.02,.5);
  \draw[mygreen] (-.5,-.75) node[below] {\textcolor{black}{\tiny $i$}}   .. controls (-.4,-.2) and (.5,0) ..  (.5,.75);
  \draw[mygreen] (.5,-.75) node[below] {\textcolor{black}{\tiny $i$}}  .. controls (.4,.2) and (-.5,0) ..  (-.5,.75);
  \node[fill=mygreen,circle,inner sep=1.8pt] at (.2,.09) {};
  }  
  \]
  Sliding up the dot on the left-hand side using~\eqref{eq:dotslides} and~\eqref{eq:chronology}, followed
  by~\eqref{eq:R3serre} to pass the $ii$-crossing to the left, and simplifying using
  \eqref{eq:klrR2} and \eqref{eq:digons} gives
  \[
  -r_it_{ij}t_{ji}
  \tikz[very thick,scale=1.2,baseline={([yshift=.8ex]current bounding box.center)}]{
  \draw[mygreen] (.4,-.4) node[below] {\textcolor{black}{\tiny $i$}}  --  (.4,.4);
  \draw[mygreen] (-.4,-.4)node[below] {\textcolor{black}{\tiny $i$}}  --  (-.4,.4)node[midway,fill=mygreen,circle,inner sep=2pt]{};
  \draw[doubleblue] (0,-.4)node[below] {\tiny $j$}  -- (0,.4);}
  \]
  Proceeding similarly on the right-hand side, but  sliding the $ii$-crossing to the right gives
  \[
  -r_it_{ij}t_{ji}
   \tikz[very thick,scale=1.2,baseline={([yshift=.8ex]current bounding box.center)}]{
  \draw[mygreen] (.4,-.4) node[below] {\textcolor{black}{\tiny $i$}}  --  (.4,.4)node[midway,fill=mygreen,circle,inner sep=2pt]{};
  \draw[mygreen] (-.4,-.4)node[below] {\textcolor{black}{\tiny $i$}}  --  (-.4,.4);
  \draw[doubleblue] (0,-.4)node[below] {\tiny $j$}  -- (0,.4);}
   \]
   and the claim follows. 
\end{proof}

\begin{lem} 
For all $i,j\in I$ with $\vert i-j\vert =1$,
  \[ 
  \tikz[very thick,baseline={([yshift=.9ex]current bounding box.center)}]{
       \draw[doublegreen] (0,-.5) node[below] {\tiny $i$} .. controls (0,0) and (1,0) .. (1,.5);
       \draw[doubleblue] (1,-.5) node[below] {\tiny $j$}  .. controls (1,0) and (0,0) .. (0,.5);
  } =\ 0 . 
\]
\end{lem}
\begin{proof}
We compute: 
\[
  \tikz[very thick,baseline={([yshift=.9ex]current bounding box.center)}]{
    \draw[doublegreen] (0,-.5) .. controls (0,0) and (1,0) .. (1,.5);
    \draw[doubleblue] (1,-.5)  .. controls (1,0) and (0,0) .. (0,.5);
    \draw[doublegreen] (0,-2) node[below] {\tiny $i$} -- (0,-.5);
    \draw[doubleblue] (1,-2) node[below] {\tiny $j$} -- (1,-.5);
  }   \overset{\eqref{eq:digons}}{=} 
  \tikz[very thick,scale=1,baseline={([yshift=.8ex]current bounding box.center)}]{
  \draw[doublegreen] (0,-.5) .. controls (0,0) and (1,0) .. (1,.5);
  \draw[doubleblue] (1,-.5)  .. controls (1,0) and (0,0) .. (0,.5);
  \draw[doublegreen] (0,-2)node[below] {\tiny $i$} -- (0,-1.5);
    \draw[doublegreen] (0,-.75) -- (0,-.5);
  \draw[mygreen] (-.015,-1.5) .. controls (-.4,-1.3) and (-.4,-.9) .. (-.015,-.75);
\node[fill=mygreen,circle,inner sep=2pt] at (.275,-1.125)  {};
  \draw[mygreen] (.015,-1.5) .. controls (.4,-1.3) and (.4,-.9) ..  (.015,-.75);
  \draw[doubleblue] (1,-2)node[below] {\tiny $j$} -- (1,-1.75);
  \draw[myblue] (1.0,-1.75) .. controls (.6,-1.3) and (.6,-.9) .. (1.0,-.5);
\node[fill=myblue,circle,inner sep=2pt] at (0.875,-1.6)  {};
  \draw[myblue] (1.015,-1.75) .. controls (1.4,-1.3) and (1.4,-.9) ..  (1.015,-.5);
  } \overset{\eqref{eq:dpitchforks}}{=} 
  \tikz[very thick,scale=1,baseline={([yshift=.8ex]current bounding box.center)}]{
  \draw[doublegreen] (0,-.5) .. controls (0,0) and (1,0) .. (1,.5);
  \draw[doubleblue] (0.1,.25) -- (0,.5);
  \draw[doublegreen] (0,-2)node[below] {\tiny $i$} -- (0,-1.5);
  \draw[doublegreen] (0,-.75) -- (0,-.5);
  \draw[mygreen] (-.015,-1.5) .. controls (-.4,-1.3) and (-.4,-.9) .. (-.015,-.75);
\node[fill=mygreen,circle,inner sep=2pt] at (.275,-1.125)  {};
  \draw[mygreen] (.015,-1.5) .. controls (.4,-1.3) and (.4,-.9) ..  (.015,-.75);
  \draw[doubleblue] (1,-2)node[below] {\tiny $j$} -- (1,-1.75);
  \draw[myblue] (1.0,-1.75) .. controls (.2,-1.0)  and (.3,.1) .. (0.075,.25);
  \draw[myblue] (1.015,-1.75) .. controls (1.4,-1.3) and (1.4,-.4) ..  (0.115,.25);
\node[fill=myblue,circle,inner sep=2pt] at (0.85,-1.6)  {};
  }
  \overset{\eqref{eq:rpitchforks}}{=}  
  \tikz[very thick,scale=1,baseline={([yshift=.8ex]current bounding box.center)}]{
  \draw[doubleblue] (0.1,.25) -- (0,.5);
  \draw[doublegreen] (0,-.5) .. controls (0,0) and (1,0) .. (1,.5);
  \draw[doublegreen] (0,-2)node[below] {\tiny $i$} -- (0,-1.5);
  \draw[doublegreen] (0,-.75) -- (0,-.5);
  \draw[mygreen] (-.015,-1.5) .. controls (-.4,-1.3) and (-.4,-.9) .. (-.015,-.75);
\node[fill=mygreen,circle,inner sep=2pt] at (-.2,-1.35)  {};
  \draw[mygreen] (.015,-1.5) .. controls (.4,-1.3) and (.4,-.9) ..  (.015,-.75);
  \draw[doubleblue] (1,-2)node[below] {\tiny $j$} -- (1,-1.75);
  \draw[myblue] (1.0,-1.75) .. controls (-1.2,-.5)  and (-.4,-.25) .. (0.075,.25);
  \draw[myblue] (1.015,-1.75) .. controls (1.4,-1.3) and (1.4,-.4) ..  (0.115,.25);
\node[fill=myblue,circle,inner sep=2pt] at (0.82,-1.672)  {};
  }   
\]
which is zero if $i=j\pm 1$ by~\eqref{eq:dotslides},~\eqref{eq:dotjumpneib} and~\eqref{eq:dotnil}.  
\end{proof}

The following are easy consequences of the defining relations of $\fR(\nu)$.

\begin{lem} 
For all $i$, $j\in I$,
  \[ 
  \tikz[very thick,baseline={([yshift=.9ex]current bounding box.center)}]{
    \draw[doublegreen] (0,-.5) node[below] {\tiny $i$} .. controls (0,0) and (1,0) .. (1,.5);
    \draw[myblue] (1,-.5) node[below] {\textcolor{black}{\tiny $j$}}  .. controls (1,0) and (0,0) .. (0,.5)node [near end,fill=myblue,circle,inner sep=2pt]{};
  }
  =
    \tikz[very thick,baseline={([yshift=.9ex]current bounding box.center)}]{
       \draw[doublegreen] (0,-.5) node[below] {\tiny $i$} .. controls (0,0) and (1,0) .. (1,.5);
       \draw[myblue] (1,-.5) node[below] {\textcolor{black}{\tiny $j$}}  .. controls (1,0) and (0,0) .. (0,.5)node [near start,fill=myblue,circle,inner sep=2pt]{};
    }
    \mspace{80mu}
\tikz[very thick,baseline={([yshift=.9ex]current bounding box.center)}]{
    \draw[mygreen] (0,-.5) node[below] {\textcolor{black}{\tiny $i$}} .. controls (0,0) and (1,0) .. (1,.5)node [near start,fill=mygreen,circle,inner sep=2pt]{};
    \draw[doubleblue] (1,-.5) node[below] {\tiny $j$}  .. controls (1,0) and (0,0) .. (0,.5);
  }
  =
 \tikz[very thick,baseline={([yshift=.9ex]current bounding box.center)}]{
    \draw[mygreen] (0,-.5) node[below] {\textcolor{black}{\tiny $i$}} .. controls (0,0) and (1,0) .. (1,.5)node [near end,fill=mygreen,circle,inner sep=2pt]{};
    \draw[doubleblue] (1,-.5) node[below] {\tiny $j$}  .. controls (1,0) and (0,0) .. (0,.5);
  } 
\]
\end{lem}

\begin{lem}
For all $i$, $j\in I$,
\begingroup\allowdisplaybreaks
\begin{align*}
\tikz[very thick,xscale=1,yscale=1,baseline={([yshift=.8ex]current bounding box.center)}]{
\draw[mygreen]  +(0,-.75) node[below] {\textcolor{black}{\tiny $i$}} .. controls (0,-.375) and (1,-.375) .. (1,0) .. controls (1,.375) and (0, .375) .. (0,.75);
\draw[doubleblue]  +(1,-.75) node[below] {\tiny $j$} .. controls (1,-.375) and (0,-.375) .. (0,0) .. controls (0,.375) and (1, .375) .. (1,.75);
 }
 \mspace{10mu}=\mspace{10mu}
\begin{cases}
    t_{ij}^2 \tikz[very thick,xscale=1,yscale=0.75,baseline={([yshift=.8ex]current bounding box.center)}]{
    \draw[mygreen]  (0,-.75) node[below] {\textcolor{black}{\tiny $i$}} -- (0,.75);
     \draw[doubleblue]  (1,-.75) node[below] {\tiny $j$} -- (1,.75);
    } & \text{ if }\vert i-j \vert > 1,
    \\ \\
    \mspace{42mu} 0 & \text{ otherwise},
\end{cases}
\\[2ex]
\tikz[very thick,xscale=1,yscale=1,baseline={([yshift=.8ex]current bounding box.center)}]{
\draw[doublegreen]  +(0,-.75) node[below] {\tiny $i$} .. controls (0,-.375) and (1,-.375) .. (1,0) .. controls (1,.375) and (0, .375) .. (0,.75);
\draw[myblue]  +(1,-.75) node[below] {\textcolor{black}{\tiny $j$}} .. controls (1,-.375) and (0,-.375) .. (0,0) .. controls (0,.375) and (1, .375) .. (1,.75);
 }
 \mspace{10mu}=\mspace{10mu}
\begin{cases}
    t_{ij}^2 \tikz[very thick,xscale=1,yscale=0.75,baseline={([yshift=.8ex]current bounding box.center)}]{
    \draw[doublegreen]  (0,-.75) node[below] {\tiny $i$} -- (0,.75);
     \draw[myblue]  (1,-.75) node[below] {\textcolor{black}{\tiny $j$}} -- (1,.75);
    } & \text{ if }\vert i-j \vert > 1,
    \\ \\
    \mspace{42mu} 0 & \text{ otherwise},
\end{cases}
\\[2ex]
\tikz[very thick,xscale=1,yscale=1,baseline={([yshift=.8ex]current bounding box.center)}]{
\draw[doublegreen]  +(0,-.75) node[below] {\tiny $i$} .. controls (0,-.375) and (1,-.375) .. (1,0) .. controls (1,.375) and (0, .375) .. (0,.75);
\draw[doubleblue]  +(1,-.75) node[below] {\tiny $j$} .. controls (1,-.375) and (0,-.375) .. (0,0) .. controls (0,.375) and (1, .375) .. (1,.75);
 }
 \mspace{10mu}=\mspace{10mu}
\begin{cases}
    t_{ij}^4 \tikz[very thick,xscale=1,yscale=0.75,baseline={([yshift=.8ex]current bounding box.center)}]{
    \draw[doublegreen]  (0,-.75) node[below] {\tiny $i$} -- (0,.75);
     \draw[doubleblue]  (1,-.75) node[below] {\tiny $j$} -- (1,.75);
    } & \text{ if }\vert i-j \vert > 1,
    \\ \\
    \mspace{42mu} 0 & \text{ otherwise} . 
\end{cases}
\end{align*}
\endgroup
\end{lem}

\begin{lem}
If $\vert i - j \vert = 1$ and $i=k$, 
  \[
         \tikz[very thick,baseline={([yshift=.9ex]current bounding box.center)}]{
		\draw[mygreen]  +(0,0)node[below] {\textcolor{black}{\tiny $i$}} .. controls (0,0.5) and (2, 1) ..  +(2,2);
		\draw[mygreen]  +(2,0)node[below] {\textcolor{black}{\tiny $k$}} .. controls (2,1) and (0, 1.5) ..  +(0,2);
		\draw[doubleblue]  (1,0)node[below] {\tiny $j$} .. controls (1,0.5) and (0, 0.5) ..  (0,1) .. controls (0,1.5) and (1, 1.5) ..  (1,2);
        }
 	\mspace{10mu}-\mspace{10mu}
	 \tikz[very thick,baseline={([yshift=.9ex]current bounding box.center)}]{
		\draw[mygreen]  +(0,0)node[below] {\textcolor{black}{\tiny $i$}} .. controls (0,1) and (2, 1.5) ..  +(2,2);
		\draw[mygreen]  +(2,0)node[below] {\textcolor{black}{\tiny $k$}} .. controls (2,.5) and (0, 1) ..  +(0,2);
		\draw[doubleblue]  (1,0)node[below] {\tiny $j$} .. controls (1,0.5) and (2, 0.5) ..  (2,1) .. controls (2,1.5) and (1, 1.5) .. (1,2);
	 }
         \mspace{10mu}=\mspace{10mu}
	 r_it_{ij}^2\ \tikz[very thick,baseline={([yshift=.9ex]current bounding box.center)}]{
	   \draw[mygreen] (0,0) node[below] {\textcolor{black}{\tiny $i$}} -- (0,2)
           node[midway,fill=mygreen,circle,inner sep=2pt]{};;
    	 \draw[doubleblue] (1,0) node[below] {\tiny $j$} -- (1,2);
         \draw[mygreen] (2,0) node[below] {\textcolor{black}{\tiny $k$}} -- (2,2);
	 }
         \mspace{10mu}-\mspace{10mu}
	 r_it_{ij}^2\ \tikz[very thick,baseline={([yshift=.9ex]current bounding box.center)}]{
	 \draw[mygreen] (0,0) node[below] {\textcolor{black}{\tiny $i$}} -- (0,2);
    	 \draw[doubleblue] (1,0) node[below] {\tiny $j$} -- (1,2);
         \draw[mygreen] (2,0) node[below] {\textcolor{black}{\tiny $k$}} -- (2,2)
         node[midway,fill=mygreen,circle,inner sep=2pt]{};;
	 } 
\]
If $i\neq j\neq k$ and at least one of the strands is double,
then the right hand side is zero. 
\end{lem}

Let
\[
\seq(\nu) :=
\bigl\{ i_1^{(\varepsilon_1)}\dotsm  i_r^{(\varepsilon_r)}\in\cseq(\nu) \bigr\vert\, \varepsilon_s=1 \bigr\} \subset\cseq(\nu) . 
\]

The superalgebra
\[
\overline{\fR}(\nu) = \bigoplus\limits_{\bi,\bj \in\seq(\nu)}\Hom_{\fR(\nu)}(\bi,\bj) , 
\]
is the sub-superalgebra of the $\Hom$-superalgebra of $\fR(\nu)$ consisting of all diagrams having only simple strands. 
If we interpret $\overline{\fR}(\nu)$ as a superalgebra version of a level 2 cyclotomic KLR algebra for $\sln$ then
$\fR(\nu)$ can be seen as version of the thick calculus~\cite{KLMS,stosic} for this superalgebra. 
It is not hard to see that both the center and the supercenter of $\overline{\fR}(\nu)$ are zero.

\subsection{Cyclotomic quotients}

Fix a $\sln$-weight $\Lambda$ and denote by $R^\Lambda(\nu)$, $\overline{\fR}^\Lambda(\nu)$ and $\fR^\Lambda(\nu)$ the 
cyclotomic quotients of $R(\nu)$, $\overline{\fR}(\nu)$ and $\fR(\nu)$. The following is immediate. 
\begin{lem}\label{lem:Risomod}
If $\Lambda$ is of level 2 then the algebras $\overline{\fR}^\Lambda(\nu)\otimes_{\bZ}(\bZ/2\bZ)$ and 
  $R^\Lambda(\nu)\otimes_{\bZ}(\bZ/2\bZ)$ are isomorphic
  (after collapsing the $\bZ/2\bZ$ grading of $\overline{\fR}^\Lambda(\nu)$). 
\end{lem}

\medskip

We depict a morphism of $\fR^\lambda(\nu)$ by decorating the rightmost region of each diagram $D$ with the weight $\Lambda$.
This defines weights for all regions of $D$.

\smallskip

The supercategory $\fR^\Lambda :=\oplus_{\nu\in\bN_0[I]}\fR^\Lambda(\nu)$ is not monoidal anymore, but it is a 
(left) module category over $\fR$, where $\fR$ acts by adding diagrams of $\fR$ to the left of diagrams from $\fR^\Lambda$. 
This is expressed by a bifunctor
\begin{equation} \label{eq:bifunctor}
\Phi\colon  \fR \times \fR^\lambda \to  \fR^\lambda.
\end{equation}

\subsection{A super 2-category}

There is a super 2-category around $\fR(\nu)$, paralleling the case of Khovanov--Lauda and Rouquier.  
An element $\bi=i_1^{(\varepsilon_1)}\dotsm i_r^{(\varepsilon_r)}$ in $\cseq(\nu)$ corresponds to a  
root $\alpha_{\bi}:= \sum_s\varepsilon_s\alpha_s$. 
Let $\Lambda(n,d):=\{\mu\in\{0,1,2\}^n\vert \mu_1+\dotsm+\mu_n = d\}$. 

\smallskip

Define $\cR(n,d)$ as the super 2-category with objects the elements of
$\Lambda(n,d)$ and
with morphism supercategories $\HOM_{\cR(n,d)}(\mu,\mu')$ the various $\fR(\nu)$.
In other words, a 1-morphisms $\mu\to\mu'$ is a sequence $\bi$ such that $\mu'-\mu=\alpha_{\bi}$ 
and the 2-morphism space $\bi\to\bj$ is $\Hom_{\fR(\nu)}(\bi,\bj)$.

\smallskip

Similarly we define the super 2-category $\cR^\Lambda(n,d)$ by using the cyclotomic quotient with respect with
the integral dominant weight $\Lambda$.
Both super 2-categories $\cR^\Lambda(n,d)$ have diagrammatic presentations with regions labeled by objects $\Lambda$.
The 2-morphisms in $\cR^\lambda(n,d)$ are presented as a collection of 2-morphisms in 
$\cR(n,d)$ with rightmost region decorated with $\Lambda$, subjected to the same relations together with the
cyclotomic condition.
This defines a label for every region of a diagram of $\cR^\Lambda(n,d)$.

\smallskip

For later use, we denote
\[
F_\bi\lambda:=F_{i_1^{(\varepsilon_1)}\dotsm i_r^{(\varepsilon_r)}}\lambda:=F_{i_1}^{(\varepsilon_1)}\dotsm F_{i_r}^{(\varepsilon_r)}\lambda
\] 
the 1-morphisms of $\cR^\Lambda(n,d)$ and, by abuse of notation, the objects of $\fR^\Lambda$.

\subsection{Action on a supercommutative ring}\label{sec:polyaction}

We now construct an action of $\fR(\nu)$ on exterior spaces. 

\subsubsection{Demazure operators on an exterior algebra}

Let $V=\bV(y_1,\dotsc, y_d)$ be the exterior algebra in $d$ variables.
This algebra is naturally graded by word length. Denote by $\vert z\vert$ the degree of the homogeneous
element $z$. 

The symmetric group $\fS_d$ acts on $V$ by the permutation action,
\[
wy_i=y_{w(i)}
\]
for all $w\in\fS_d$.

Define operators $\partial_i$ for $i=1,\dotsc, d-1$ on $V$ by the following rules:
\begin{align*}
\partial_i(y_k)
&= \begin{cases}
  1 & i = k, k+1 ,
  \\
  0 & \text{otherwise} ,
\end{cases}
\intertext{and}
\partial_i(fg) &= \partial_i(f)g + (-1)^{\vert f\vert}f\partial_i(g) , 
\end{align*}
for all $f$, $g\in V$ such that $fg\neq 0$. 

The following can be checked through a simple computation. 
\begin{lem}
  The operators $\partial_i$ satisfy the relations $\partial_i^2=0$, $\partial_i\partial_j+\partial_j\partial_i=0$ if
  $\vert i -j \vert >1$, and $\partial_i\partial_{i+1}\partial_i=\partial_{i+1}\partial_{i}\partial_{i+1}$. 
\end{lem}

\subsubsection{An action of $\fR(\nu)$ on supercommutative rings}
For $\bi\in\cseq(\nu)$ let
\[
P\bi=\bV(x_{1,1},x_{1,\varepsilon_1},\dotsc , x_{d,1},x_{d,\varepsilon_d} )\bi , 
\]
be an exterior algebra in $\sum_i\nu_i$ generators, and set
\[
P(\nu) =\bigoplus_{\bi\in\cseq(\nu)}P\bi .
\]
We extend the action of $\fS_d$ from $V$ to $P(\nu)$ by declaring that 
\[
w x_{r,1}=x_{w(r),1} ,
\mspace{20mu}
wx_{r,\varepsilon_r}= x_{w(r),\varepsilon_{r+1}}, 
\]
or $w\in\fS_d$.  

Below we denote by $\partial_{u,z}$ the Demazure operator with respect to the variables $u$ and $z$. 

\medskip

To the object $\bi\in\fR(\nu)$ we associate the idempotent $\bi\in P\bi$. 
The defining generators of $\fR(\nu)$ act on $P$ as follows. 
A diagram $D$ acts as zero on $P\bi$ unless the sequence of labels in the bottom of $D$ is $\bi$. 

\smallskip 
\n$\bullet$ Dots 
\[
\tikz[very thick,scale=1.5,baseline={([yshift=0ex]current bounding box.center)}]{
  \draw[mygreen] (0,-.5) -- (0,.5)node[midway,fill=mygreen,circle,inner sep=2pt]{};
  \node at (-.2,-.4) {\tiny $i_r$}; 
} \colon p\bi\mapsto x_{r,1}p\bi , 
\]
\n$\bullet$ Splitters
\begin{equation}\label{eq:splittersaction}
\tikz[very thick,scale=1.5,baseline={([yshift=-.8ex]current bounding box.center)}]{
  \draw[doublegreen] (0,-.4) -- (0,0);
  \draw[mygreen] (-.015,0) .. controls (0,.12) and (-.25,.12) .. (-.3,.5);
  \draw[mygreen] (.015,0) .. controls (0,.12) and (.25,.12) ..  (.3,.5);
\node at (-.2,-.25) {\tiny $i_r$}; 
} \colon p\bi\mapsto \partial_{x_{r,1},x_{r,2}}(p)\bi ,
\mspace{80mu}
\tikz[very thick,scale=1.5,baseline={([yshift=-.8ex]current bounding box.center)}]{
  \draw[doublegreen] (0,.4)  -- (0,0);
  \draw[mygreen] (-.015,0) .. controls (0,-.12) and (-.25,-.12) .. (-.3,-.5);
  \draw[mygreen] (.015,0) .. controls (0,-.12) and (.25,-.12) ..  (.3,-.5);
  \node at (-.2,.25) {\tiny $i_r$};
} \colon p\bi\mapsto x_{r,1}\partial_{x_{r,1},x_{r,2}}(p)\bi , 
\end{equation}

\n$\bullet$ Crossings
\begingroup\allowdisplaybreaks
\begin{align}\nonumber
\tikz[very thick,scale=1,baseline={([yshift=.8ex]current bounding box.center)}]{
  \draw[myblue] (0,-.5) node[below] {\textcolor{black}{\tiny $i_r$}}   .. controls (0,0) and (1,0) ..  (1,.5);
  \draw[mygreen] (1,-.5) node[below] {\textcolor{black}{\tiny $i_{r+1}$}}  .. controls (1,0) and (0,0) ..  (0,.5);
} & \colon p\bi \mapsto 
\begin{cases}
  r_{i_r}\partial_{x_{r,1},x_{r+1,1}}(p)\bi & \text{if}\quad i_r=i_{r+1},
  \\[1ex]
   (t_{i_{r+1}i_r}x_{r,1}+t_{i_ri_{r+1}}x_{r+1,1})s_r(p\bi) & \text{if}\quad i_r=i_{r+1}+1, 
  \\[1ex]
  s_r(p\bi)  & \text{else} , 
  \end{cases}
\\[2ex] \label{eq:thickX22}
\tikz[very thick,baseline={([yshift=.9ex]current bounding box.center)}]{
       \draw[doublegreen] (0,-.5) node[below] {\tiny $i_r$} .. controls (0,0) and (1,0) .. (1,.5);
       \draw[doubleblue] (1,-.5) node[below] {\tiny $i_{r+1}$}  .. controls (1,0) and (0,0) .. (0,.5);
  } &\colon p\bi\mapsto 
\begin{cases}
  0  & \text{if}\quad i_r=i_{r+1},\ \text{ or }\quad i_s=i_{s+1}+1, 
  \\[1ex]
  s_r(p\bi)  & \text{else} , 
  \end{cases}
\\[2ex]  \label{eq:thickX21}
\tikz[very thick,baseline={([yshift=.9ex]current bounding box.center)}]{
       \draw[doublegreen] (0,-.5) node[below] {\tiny $i_r$} .. controls (0,0) and (1,0) .. (1,.5);
       \draw[myblue] (1,-.5) node[below] {\textcolor{black}{\tiny $i_{r+1}$}}  .. controls (1,0) and (0,0) .. (0,.5);
} &\colon p\bi \mapsto 
\begin{cases}
  0  & \text{if}\quad i_r=i_{r+1},
  \\[1ex]
 f_{2,1}(x_{r,1},x_{r,2},x_{r+1,1}) s_r(p\bi) & \text{if}\quad i_s=i_{s+1}+1, 
  \\[1ex]
  s_r(p\bi)  & \text{else} , 
  \end{cases}
\\[2ex]  \label{eq:thickX12}
\tikz[very thick,baseline={([yshift=.9ex]current bounding box.center)}]{
       \draw[mygreen] (0,-.5) node[below] {\textcolor{black}{\tiny $i_r$}} .. controls (0,0) and (1,0) .. (1,.5);
       \draw[doubleblue] (1,-.5) node[below] {\tiny $i_{r+1}$}  .. controls (1,0) and (0,0) .. (0,.5);
} &\colon p\bi \mapsto 
\begin{cases}
  0  & \text{if}\quad i_r=i_{r+1},
  \\[1ex]
 f_{1,2}(x_{r,1},x_{r+1,1},x_{r+1,2}) s_r(p\bi) & \text{if}\quad i_s=i_{s+1}+1, 
  \\[1ex]
  s_r(p\bi)  & \text{else} , 
  \end{cases}
\end{align}
\endgroup
where
\begin{align*}
   f_{2,1}(x_{r,1},x_{r,2},x_{r+1,1}) &= t_{i_ri_{r+1}}t_{i_{r+1}i_r} x_{r,1}x_{r+1,1}+ t_{i_ri_{r+1}}t_{i_{r+1}i_r} x_{r,1}x_{r,2} + t_{i_{r+1}i_r}^2 x_{r,2}x_{r+1,1} . 
\\
f_{1,2}(x_{r,1},x_{r+1,1},x_{r+1,2}) &=
-t_{i_ri_{r+1}}^2 x_{r,1}x_{r,2}+ t_{i_ri_{r+1}}t_{i_{r+1}i_r} x_{r,2}x_{r+1,1} - t_{i_{r+1}i_r}t_{i_{r+1}i_r} x_{r,1}x_{r+1,1} . 
\end{align*}

\begin{prop}\label{prop:basis}
The assignment above defines an action of $\fR(\nu)$ on $P(\nu)$. 
\end{prop}

\begin{proof}
By a long and rather tedious computation one can check that the operators above satisfy the defining relations of $\fR(\nu)$.

The relations involving the action of the generators of $\overline{\fR}(\nu)$ 
are easy to check by direct computation. 
For example, for $\nu=2i+j$, with $j=i+1$ we have
\begingroup\allowdisplaybreaks
\begin{align*}
\tikz[very thick,baseline={([yshift=.9ex]current bounding box.center)}]{
\draw[mygreen]  +(0,0)node[below] {\textcolor{black}{\tiny $i$}} .. controls (0,0.5) and (2, 1) ..  +(2,2);
\draw[mygreen]  +(2,0)node[below]{\textcolor{black} {\tiny $i$}} .. controls (2,1) and (0, 1.5) ..  +(0,2);
\draw[myblue] (1,0)node[below] {\textcolor{black}{\tiny $j$}} ..
controls (1,0.5) and (0, 0.5) ..  (0,1) .. controls (0,1.5) and (1, 1.5) ..  (1,2);
      } (f) &= (t_{ij}x_1+t_{ji}x_2)s_1r_i\partial_2s_1(f) ,
\intertext{and}
\tikz[very thick,baseline={([yshift=.9ex]current bounding box.center)}]{
\draw[mygreen]  +(0,0)node[below] {\textcolor{black}{\tiny $i$}} .. controls (0,1) and (2, 1.5) ..  +(2,2);
\draw[mygreen]  +(2,0)node[below] {\textcolor{black}{\tiny $i$}} .. controls (2,.5) and (0, 1) ..  +(0,2);
\draw[myblue]  (1,0)node[below] {\textcolor{black}{\tiny $j$}} ..
controls (1,0.5) and (2, 0.5) ..  (2,1) .. controls (2,1.5) and (1, 1.5) .. (1,2);
}(f) &= s_2r_i\partial_1(t_{ij}x_2+t_{ji}x_3)s_2(f)
=  r_it_{ij}f - (t_{ij}x_1+t_{ji}x_2)s_1r_i\partial_2s_1(f)  , 
\end{align*}
\endgroup
and so, for any $f(x_1,x_2,x_3)\in Piji$, 
\[
         \tikz[very thick,baseline={([yshift=.9ex]current bounding box.center)}]{
		\draw[mygreen]  +(0,0)node[below] {\textcolor{black}{\tiny $i$}} .. controls (0,0.5) and (2, 1) ..  +(2,2);
		\draw[mygreen]  +(2,0)node[below] {\textcolor{black}{\tiny $i$}} .. controls (2,1) and (0, 1.5) ..  +(0,2);
		\draw[myblue]  (1,0)node[below] {\textcolor{black}{\tiny $j$}} .. controls (1,0.5) and (0, 0.5) ..  (0,1) .. controls (0,1.5) and (1, 1.5) ..  (1,2);
        }(f)
 	\mspace{10mu}+\mspace{10mu}
	 \tikz[very thick,baseline={([yshift=.9ex]current bounding box.center)}]{
		\draw[mygreen]  +(0,0)node[below] {\textcolor{black}{\tiny $i$}} .. controls (0,1) and (2, 1.5) ..  +(2,2);
		\draw[mygreen]  +(2,0)node[below] {\textcolor{black}{\tiny $i$}} .. controls (2,.5) and (0, 1) ..  +(0,2);
		\draw[myblue]  (1,0)node[below] {\textcolor{black}{\tiny $j$}} .. controls (1,0.5) and (2, 0.5) ..  (2,1) .. controls (2,1.5) and (1, 1.5) .. (1,2);
	 }(f)
         \mspace{10mu}=\mspace{10mu}
	 r_it_{ij}\
         \tikz[very thick,baseline={([yshift=.9ex]current bounding box.center)}]{
           \draw[mygreen] (  0,0) node[below] {\textcolor{black}{\tiny $i$}} -- (0,2);
 \draw[myblue] (1,0) node[below] {\textcolor{black}{\tiny $j$}} -- (1,2);
\draw[mygreen] (2,0) node[below] {\textcolor{black}{\tiny $i$}} -- (2,2);
} (f)  . 
 \]

Setting as in~\cite{KLMS}, 
\[
\tikz[very thick,scale=1.5,baseline={([yshift=.8ex]current bounding box.center)}]{
  \draw[doublegreen] (1.5,-.5) node[below] {\tiny $i$}  -- (1.5,.5);
} :=
\tikz[very thick,xscale=1,yscale=1,baseline={([yshift=.8ex]current bounding box.center)}]{
\draw[mygreen] (0,-.75) node[below] {\textcolor{black}{\tiny $i$}}  .. controls (0,0) and (.75,0) .. (.75,.75);
\draw[mygreen] (.75,-.75) node[below] {\textcolor{black}{\tiny $i$}}  .. controls (.75,0) and (0,0) .. (0,.75)node [near end,fill=mygreen,circle,inner sep=2pt]{};;
} 
\mspace{80mu}
\tikz[very thick,scale=1.5,baseline={([yshift=.8ex]current bounding box.center)}]{
  \draw[doublegreen] (0,-.5)node[below] {\tiny $i$} -- (0,0);
  \draw[mygreen] (-.015,0) .. controls (0,.12) and (-.25,.12) .. (-.3,.5);
  \draw[mygreen] (.015,0) .. controls (0,.12) and (.25,.12) ..  (.3,.5);
}:=
\tikz[very thick,xscale=1,yscale=1,baseline={([yshift=.8ex]current bounding box.center)}]{
\draw[mygreen] (0,-.75) node[below] {\textcolor{black}{\tiny $i$}} .. controls (0,0) and (.75,0) .. (.75,.75);
\draw[mygreen] (.75,-.75) node[below] {\textcolor{black}{\tiny $i$}}  .. controls (.75,0) and (0,0) .. (0,.75);
}
\mspace{80mu}
\tikz[very thick,scale=1.5,baseline={([yshift=.8ex]current bounding box.center)}]{
  \draw[doublegreen] (0,0)  -- (0,.5);
  \draw[mygreen] (-.015,0) .. controls (0,-.12) and (-.25,-.12) .. (-.3,-.5)node[below] {\textcolor{black}{\tiny $i$}};
  \draw[mygreen] (.015,0) .. controls (0,-.12) and (.25,-.12) ..  (.3,-.5)node[below] {\textcolor{black}{\tiny $i$}};
  }:=
\tikz[very thick,xscale=1,yscale=1,baseline={([yshift=.8ex]current bounding box.center)}]{
\draw[mygreen] (0,-.75) node[below] {\textcolor{black}{\tiny $i$}}  .. controls (0,0) and (.75,0) .. (.75,.75);
\draw[mygreen] (.75,-.75) node[below] {\textcolor{black}{\tiny $i$}}  .. controls (.75,0) and (0,0) .. (0,.75)node [near end,fill=mygreen,circle,inner sep=2pt]{};;
} 
\]
and
\[
  \tikz[very thick,scale=1.2,baseline={([yshift=.9ex]current bounding box.center)}]{
       \draw[mygreen] (0,-.5) node[below] {\textcolor{black}{\tiny $i$}} .. controls (0,0) and (1,0) .. (1,.5);
       \draw[doubleblue] (1,-.5) node[below] {\tiny $j$}  .. controls (1,0) and (0,0) .. (0,.5);
  } :=
\tikz[very thick,scale=1.2,baseline={([yshift=.8ex]current bounding box.center)}]{
 \draw[mygreen] (0,-.5) node[below] {\textcolor{black}{\tiny $i$}} .. controls (0,0) and (1,0) .. (1,.5);
 \draw[doubleblue] (0.15,.35) -- (0,.5);  \draw[doubleblue] (.875,-.35) -- (1,-.5)node[below] {\tiny $j$};
 \draw[myblue] (.15,.35) .. controls (0,.1)  and (0,-.2) .. (0.875,-.35);
 \draw[myblue] (.15,.35) .. controls (1,.2)  and (1,-.1) .. (0.875,-.35);
 \node[fill=myblue,circle,inner sep=2pt] at (.55,-.275) {};
}
\mspace{80mu}
  \tikz[very thick,,scale=1.2,baseline={([yshift=.9ex]current bounding box.center)}]{
       \draw[myblue] (-0,-.5) node[below] {\textcolor{black}{\tiny $j$}} .. controls (0,0) and (-1,0) .. (-1,.5);
       \draw[doublegreen] (-1,-.5) node[below] {\tiny $i$}  .. controls (-1,0) and (0,0) .. (0,.5);
  } :=
\tikz[very thick,scale=1.2,baseline={([yshift=.8ex]current bounding box.center)}]{
 \draw[myblue] (0,-.5) node[below] {\textcolor{black}{\tiny $j$}} .. controls (0,0) and (-1,0) .. (-1,.5);
 \draw[doublegreen] (-0.15,.35) -- (0,.5);  \draw[doublegreen] (-.875,-.35) -- (-1,-.5)node[below] {\tiny $i$};
 \draw[mygreen] (-.15,.35) .. controls (0,.1)  and (0,-.2) .. (-0.875,-.35);
 \draw[mygreen] (-.15,.35) .. controls (-1,.2)  and (-1,-.1) .. (-0.875,-.35);
 \node[fill=mygreen,circle,inner sep=2pt] at (-.55,-.275) {};
}
\mspace{80mu}
  \tikz[very thick,,scale=1.2,baseline={([yshift=.9ex]current bounding box.center)}]{
       \draw[doubleblue] (-0,-.5) node[below] {\tiny $j$} .. controls (0,0) and (-1,0) .. (-1,.5);
       \draw[doublegreen] (-1,-.5) node[below] {\tiny $i$}  .. controls (-1,0) and (0,0) .. (0,.5);
  } :=
\tikz[very thick,scale=1.2,baseline={([yshift=.8ex]current bounding box.center)}]{
 \draw[doublegreen] (-0.15,.35) -- (0,.5);  \draw[doublegreen] (-.875,-.35) -- (-1,-.5)node[below] {\tiny $i$};
 \draw[mygreen] (-.15,.35) .. controls (-.1,.1)  and (-.1,-.2) .. (-0.875,-.35);
 \draw[mygreen] (-.15,.35) .. controls (-.9,.2)  and (-.9,-.1) .. (-0.875,-.35);
 \node[fill=mygreen,circle,inner sep=1.5pt] at (-.7,-.32) {};
 \draw[doubleblue] (-.85,.35) -- (-1,.5);  \draw[doubleblue] (-.125,-.35) -- (0,-.5)node[below] {\tiny $j$};
 \draw[myblue] (-.85,.35) .. controls (-.9,.1)  and (-.9,-.2) .. (-.125,-.35);
 \draw[myblue] (-.85,.35) .. controls (-.1,.2)  and (-.1,-.1) .. (-.125,-.35);
 \node[fill=myblue,circle,inner sep=1.5pt] at (-.72,.32) {};
}
\] 
then it follows that the action of the generators of $\fR(\nu)$  on $P(\nu)$ is given by the
operators~\eqref{eq:splittersaction}, \eqref{eq:thickX22}, \eqref{eq:thickX21} and ~\eqref{eq:thickX12}
and satisfy the defining relations of $\fR(\nu)$.   
\end{proof}

%
%
\section{A topological invariant}\label{sec:tangOKH}

In~\cite{tubbenhauer}  $q$-skew Howe duality is used to show how to write as a web
in a form 
that uses only the
lower part of $U_q(\glk)$. In this language, the formula for the $\slt$-commutator becomes one of Luzstig's higher quantum Serre relations from~\cite[\S7]{Lu}.
It is also proved in~\cite{tubbenhauer}
that this results in a well defined evaluation
of closed webs allowing to write any link diagram as a linear combination of words in the various $F_i$ in $U^-:=U_q^-(\glk)$.

\smallskip

This allows a categorification of webs using only (cyclotomic) KLR algebras~\cite{KL1,R1} instead of the whole
2-quantum group  $\cU(\glk)$~\cite{KL3,R1}. 
In this context, the unit and co-unit maps of the several adjunctions in $\cU(\glk)$ that are used as differentials in 
the Khovanov--Rozansky chain complex can be written as composition with elements of the KLR algebra. 
Taking cyclotomic KLR algebras of level 2 gives Khovanov homology. 
The approach in~\cite{tubbenhauer} is easily adapted to tangles, which we do in 
in this section for level 2 in the context of the supercategories introduced in~\fullref{sec:skewKLR}.

\subsection{Supercategorification of $\glt$-webs and flat tangles}\label{sec:scatflat}

Our webs have strands labeled from $\{0,1,2\}$ which we depict as ``invisible'', ``simple'', and ``double'',
as in the example below. All the strands point either up or to the right and sometimes we omit the orientations in 
the pictures.
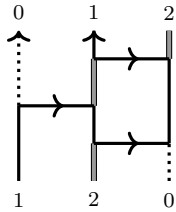
\[    \tikz[very thick,scale=.25,baseline={([yshift=-.8ex]current bounding box.center)}]{
	\draw (-2,-4)node[below] {\tiny $1$} to (-2,0);
	\draw [directed=1] (-2,0) [dotted] to (-2,4)node[above] {\tiny $0$};
	\draw [doubleblack] (2,-4)node[below] {\tiny $2$} to (2,-2);	\draw (2,-2) to (2,0);
	\draw [directed=1] (2,2.5) to (2,4)node[above] {\tiny $1$};
      	\draw [doubleblack] (2,0) to (2,2.5);
	\draw [rdirected=.55] (2,0) to (-2,0);	\draw [rdirected=.55] (6,-2) to (2,-2);
	\draw [dotted] (6,-4)node[below] {\tiny $0$} to (6,-2);	\draw (6,-2) to (6,2.5);
       	\draw [doubleblack] (6,2.5) to (6,4)node[above] {\tiny $2$};
	\draw [rdirected=.55] (6,2.5) to (2,2.5); }
\]

\smallskip

For $\lambda=(\lambda_1,\dotsm ,\lambda_k)\in\{0,1,2\}^k$ and $\epsilon\in\{0,1\}$ with
$\vert\lambda\vert=2\ell+\epsilon$, we put $\Lambda=(2)^{\ell}\epsilon=(2,\dots,2,\epsilon,0,\dotsc ,0)$
and we define 
\[
\fW(\lambda) = \HOM_{\cR^\Lambda(k,\vert\lambda\vert)}(\Lambda,\lambda) . 
\]

\medskip

Let $W$ be a $\glt$-web with all ladders pointing to the right.
Suppose that $W$ has the bottom boundary labelled $\lambda$ and the
top boundary labelled $\mu$, whith $\lambda,\mu\in\{ 0,1,2\}^k$ and  $\vert\lambda\vert =\vert\mu\vert$.
We write $W$ as a word in the $F_i$ in $U^-_q(\glk)$ applied to a vector $v_\lambda$ of $\glk$-weight $\lambda$. 
\[
\tikz[very thick,baseline={([yshift=-.8ex]current bounding box.center)}]{
\draw (-1,-.5) -- (1,-.5);\draw (-1, .5) -- (1, .5);
\draw (-1,-.5) -- (-1,.5);\draw (1,-.5) -- (1,.5);
\node at (0,0) {$W$};
\draw (-.75,-1)node[below]{\tiny $\lambda_1$} -- (-.75,-.5);
\node at (0, -.85) {$\dotsm$};
\draw (.75,-1)node[below]{\tiny $\lambda_k$} -- (.75,-.5);
\draw (-.75,1)node[above]{\tiny $\mu_1$} -- (-.75,.5);
\node at (0, .85) {$\dotsm$};
\draw (.75,1)node[above]{\tiny $\mu_k$} -- (.75,.5);
} = F_{i_1}\dotsm F_{i_r}(v_\lambda) . 
\]
This gives a 1-morphism $F(W)$ in $\cR(k,\vert\lambda\vert)$. 
Composition of 1-morphisms in $\cR(k,\vert\lambda\vert)$ defines a superfunctor
\[
\fF(W)\colon \fW(\lambda) \to \fW(\mu) .  
\]
 If $\lambda$ is dominant and $\mu$ is antidominant then $\fF(W)$ is a superfunctor
 from $\Bbbk\smod$ to $\Bbbk\smod$ that is, a direct sum of grading shifts
 of the identity superfunctor. 
 In this case,  there is a canonical 1-morphism $F_{{can}}(W)$ in $\Hom_{\cR^\Lambda(k,\vert\lambda\vert)}(\lambda,\mu)$ 
\begin{equation}\label{eq:canF} 
  F_{can} =
  F_{ (k-\ell-1)^{(2)}\dotsm (1)^{(2)}}
    \dotsm
  F_{ (k-3)^{(2)}\dotsm (\ell-1)^{(2)}}
  F_{ (k-2)^{(2)}\dotsm \ell^{(2)} }
  F_{ (k-1)^{(\epsilon)}\dotsm (\ell+1)^{(\epsilon)} }
  (2)^{\ell}\epsilon , 
\end{equation}
which in terms of webs takes the form of the following example:    
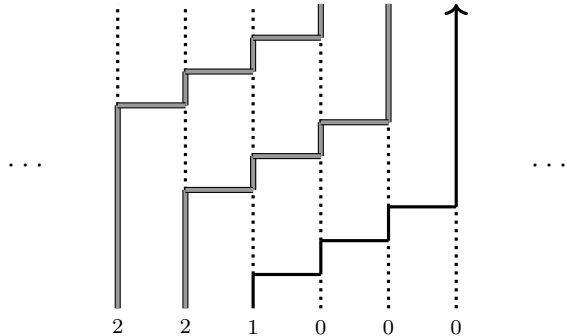
\[
\dotsm\mspace{30mu} 
\tikz[very thick,scale=.225,baseline={([yshift=0ex]current bounding box.center)}]{
\draw[doubleblack] (-2,-15)node[below] {\tiny $2$} to (-2,-3);\draw [dotted](-2,-3) to (-2,3); 
\draw[doubleblack] ( 2,-15)node[below] {\tiny $2$} to (2,-8);\draw[dotted] (2,-8) to (2,-3);
      \draw[doubleblack] (2,-3) to (2,-1);\draw[dotted] (2,-1) to (2,3);   
\draw (6,-15)node[below] {\tiny $1$} to (6,-13);\draw[dotted] (6,-13) to (6,-8);
      \draw[doubleblack] (6,-8) to (6,-6);\draw[doubleblack] (6,-8) to (6,-6);\draw[dotted] (6,-6) to (6,-1);
      \draw[doubleblack] (6,-1) to (6,1);\draw[dotted] (6,1) to (6,3);
\draw [dotted] (10,-15)node[below] {\tiny $0$} to (10,-13);\draw (10,-13) to (10,-11);
      \draw[dotted] (10,-11) to (10,-6);\draw[doubleblack] (10,-6) to (10,-4);\draw[dotted] (10,-4) to (10,1);
       \draw[doubleblack] (10,1) to (10,3);   
\draw [dotted] (14,-15)node[below] {\tiny $0$} to (14,-11);\draw (14,-11) to (14,-9);
       \draw[dotted] (14,-9) to (14,-4);\draw[doubleblack] (14,-4) to (14,3);
\draw [dotted] (18,-15)node[below] {\tiny $0$} to (18,-9);\draw [->] (18,-9) to (18,3);
\draw [doubleblack] (-2,-3) to (2,-3);
\draw [doubleblack] (2,-1) to (6,-1);
\draw [doubleblack] (6,1) to (10,1);
\draw [doubleblack] (2,-8) to (6,-8);
\draw [doubleblack] (6,-6) to (10,-6);
\draw [doubleblack] (10,-4) to (14,-4);
\draw (6,-13) to (10,-13);
\draw (10,-11) to (14,-11);
\draw (14,-9) to (18,-9);
}
\mspace{30mu}\dotsm  
\]
We have that $\fF(W)=\Hom_{\cR^\lambda(k,\vert\lambda\vert)}(\lambda,\mu)$ is isomorphic to the graded
$\Bbbk$-supervector space 
$\Hom_{\fR^\Lambda}(F_{can}(W),F(W))$.

\subsection{The chain complex}

As explained in~\cite{tubbenhauer} any oriented tangle diagram $T$ can be written in the form of a web $W_T$ with
all horizontal strands pointing to the right. In this case we say that \emph{$T$ is in $F$-form}.

\begin{ex}\label{ex:hopf}
For the Hopf link we have the following web diagram. 
\[
    \tikz[very thick,scale=.25,baseline={([yshift=-.8ex]current bounding box.center)}]{
\draw [doubleblack](-2,-4)node[below] {\tiny $2$} to (-2,2);\draw (-2,2) to (-2,7);\draw[dotted] (-2,7) to (-2,11.5);
\draw [doubleblack]( 2,-4)node[below] {\tiny $2$} to (2,-2.5);	\draw (2,-2.5) to ( 2,1);
\draw (2,3) to (2,5);\draw[dotted] (2,5) to (2,7);\draw (2,7) to (2,9);\draw[dotted] (2,9) to (2,11.5);   
\draw [dotted] (6,-4)node[below] {\tiny $0$} to (6,-2.5);\draw (6,-2.5) to ( 6,-.5);\draw (6,2) to (6,4); 
     \draw (6,6) to (6,9);\draw[doubleblack] (6,9) to (6,11.5); 
\draw [dotted] (10,-4)node[below] {\tiny $0$} to (10,-.5);\draw (10,-.5) to (10,5);\draw[doubleblack] (10,5) to (10,11.5);
	\draw [directed=.55] (2,-2.5) to (6,-2.5);
	\draw [directed=.55] (6,-.5) to (10,-.5);
	\draw [directed=.35] (-2,2) to (6,2);
	\draw [dotted] (6,-.5) to ( 6,1.5);
      	\draw [directed=.35] (2,5) to (10,5);
       	\draw [directed=.55] (-2,7) to (2,7);
       	\draw [directed=.55] (2,9) to (6,9);
\draw (2,2.5) to (2,4);
\draw[green] (-5,.5) -- (13,.5);\draw[green] (-5,3.25) -- (13,3.25);\draw[green] (-5,6.25) -- (13,6.25);
    }
\]
\end{ex}

Suppose the bottom boundary of $W_T$ is $(\lambda_1,\dotsm,\lambda_k)$ and the top boundary is $(\mu_1,\dotsm,\mu_k)$. 
Let $\Kom(\lambda,\mu)$ be the category of complexes of $\HOM_{\cR(k,\vert\lambda\vert)}(\fW(\lambda),\fW(\mu))$
and $\Kom_{/h}(\lambda,\mu)$ its homotopy category (these are not supercategories).  
To each tangle in $F$-form as above we associate an object in $\Kom_{/h}(\lambda,\mu)$ as follows.
The usual constructions with chain complexes (homomorphisms, homotopies, cones, etc.) work in the same way
as with nonsupercategories. Since we are in a supercategory, some signs have to be introduced when taking tensor products (further details will appear in a follow-up paper).

We first chop the diagram vertically in such  way that each slice contains either a web without crossings, or 
a single crossing together with vertical pieces (as in~\fullref{ex:hopf}).
Each slice then gives either a superfunctor or a complex of superfunctors, as explained below.
By composition we get a complex
${\fF}(W_T)$ of superfunctors from $\fW(\lambda)$ to $\fW(\mu)$. 

\subsubsection{Basic tangles}
\begin{itemize}
\item If $T$ is a flat tangle, then we're done by~\fullref{sec:scatflat}. 

\item To the positive crossing we associate the chain complex
\begin{equation}\label{eq:pXing}
\tikz[very thick,scale=.25,baseline={([yshift=-.6ex]current bounding box.center)}]{
\draw (2,-1) to (2,1);\draw[->] (2,3) to (2,5);   
\draw (-1,2) to (2,2);\draw[->] (2,2) to (5,2);
}
  \mapsto 
q^{-1}\fF\left(  \tikz[very thick,scale=.25,baseline={([yshift=-.8ex]current bounding box.center)}]{
	\draw (-2,-4)node[below] {\tiny $1$} to (-2,0);
	\draw [directed=1] (-2,0) [dotted] to (-2,4)node[above] {\tiny $0$};
	\draw (2,-4)node[below] {\tiny $1$} to (2,-2.5);
	\draw (2,-2.5)[dotted] to (2,0);
	\draw [directed=1] (2,0) to (2,4)node[above] {\tiny $1$};
	\draw [rdirected=.55] (2,0) to (-2,0);
	\draw [dotted] (6,-4)node[below] {\tiny $0$} to (6,-2.5);
	\draw (6,-2.5) to (6,2);
       	\draw [directed=1] (6,2) to (6,4)node[above] {\tiny $1$};
	\draw [rdirected=.55] (6,-2.5) to (2,-2.5);
}\right) \xra{\tikz[very thick,baseline={([yshift=.9ex]current bounding box.center)}]{
       \draw[mygreen] (0,-.5) .. controls (0,0) and (1,0) .. (1,.5);
       \draw[myblue] (1,-.5) .. controls (1,0) and (0,0) .. (0,.5);
       \node at (-.15,-.425) {\tiny $1$}; \node at (1.15,-.425) {\tiny $2$}; 
  }}
\fF\left(
    \tikz[very thick,scale=.25,baseline={([yshift=-.8ex]current bounding box.center)}]{
	\draw (-2,-4)node[below] {\tiny $1$} to (-2,0);
	\draw [directed=1] (-2,0) [dotted] to (-2,4)node[above] {\tiny $0$};
	\draw (2,-4)node[below] {\tiny $1$} to (2,0);
	\draw [directed=1] (2,2.5) to (2,4)node[above] {\tiny $1$};
      	\draw [doubleblack] (2,0) to (2,2.5);
	\draw [rdirected=.55] (2,0) to (-2,0);
	\draw [dotted] (6,-4)node[below] {\tiny $0$} to (6,2.5);
       	\draw [directed=1] (6,2.5) to (6,4)node[above] {\tiny $1$};
	\draw [rdirected=.55] (6,2.5) to (2,2.5);
    }\right) 
    \end{equation}
with the leftmost term in homological degree zero.     
Algebraically this can be written
\[
\beta_+ \mapsto q^{-1}F_1F_2(1,1,0) \xra{\ \tau_1 \ } F_2F_1(1,1,0) ,
\]
where $\tau$ is the diagram above. 

\item To the negative crossing we associate the chain complex
\begin{equation} \label{eq:nXing}
\tikz[very thick,scale=.25,baseline={([yshift=-.6ex]current bounding box.center)}]{
\draw (2,-1) to (2,2);\draw[->] (2,2) to (2,5);   
\draw (-1,2) to (1,2);\draw[->] (3,2) to (5,2);
}  \mapsto
\fF\left(
    \tikz[very thick,scale=.25,baseline={([yshift=-.8ex]current bounding box.center)}]{
	\draw (-2,-4)node[below] {\tiny $1$} to (-2,0);
	\draw [directed=1] (-2,0) [dotted] to (-2,4)node[above] {\tiny $0$};
	\draw (2,-4)node[below] {\tiny $1$} to (2,0);
	\draw [directed=1] (2,2.5) to (2,4)node[above] {\tiny $1$};
      	\draw [doubleblack] (2,0) to (2,2.5);
	\draw [rdirected=.55] (2,0) to (-2,0);
	\draw [dotted] (6,-4)node[below] {\tiny $0$} to (6,2.5);
       	\draw [directed=1] (6,2.5) to (6,4)node[above] {\tiny $1$};
	\draw [rdirected=.55] (6,2.5) to (2,2.5);
    }\right)\xra{\tikz[very thick,baseline={([yshift=.9ex]current bounding box.center)}]{
       \draw[myblue] (0,-.5) .. controls (0,0) and (1,0) .. (1,.5);
       \draw[mygreen] (1,-.5) .. controls (1,0) and (0,0) .. (0,.5);
       \node at (-.15,-.425) {\tiny $2$}; \node at (1.15,-.425) {\tiny $1$}; 
  }}q\fF\left(
  \tikz[very thick,scale=.25,baseline={([yshift=-.8ex]current bounding box.center)}]{
	\draw (-2,-4)node[below] {\tiny $1$} to (-2,0);
	\draw [directed=1] (-2,0) [dotted] to (-2,4)node[above] {\tiny $0$};
	\draw (2,-4)node[below] {\tiny $1$} to (2,-2.5);
	\draw (2,-2.5)[dotted] to (2,0);
	\draw [directed=1] (2,0) to (2,4)node[above] {\tiny $1$};
	\draw [rdirected=.55] (2,0) to (-2,0);
	\draw [dotted] (6,-4)node[below] {\tiny $0$} to (6,-2.5);
	\draw (6,-2.5) to (6,2);
       	\draw [directed=1] (6,2) to (6,4)node[above] {\tiny $1$};
	\draw [rdirected=.55] (6,-2.5) to (2,-2.5);
  }\right)
\end{equation}
with the rightmost term in homological degree zero.
Algebraically
 \[ 
\beta_-  \mapsto   F_2F_1(1,1,0)\xra{\ \tau_1 \ }qF_1F_2(1,1,0) . 
\]
\end{itemize}

\begin{rem}
  Caution should be taken when applying~\eqref{eq:pXing} and~\eqref{eq:nXing}: when passing from a tangle diagram to it
  $F$-form some crossings may change from positive to negative and vice-versa.
  To have an invariant of all tangles some grading shifts have to be introduced locally whenever this occurs. 
  We shift~\eqref{eq:nXing} by $-1$ in the $q$-grading and $1$ in the homological grading when it comes from a positive crossing  and the opposite whenever~\eqref{eq:nXing} comes from a positive crossing. 
\end{rem}
  
\subsubsection{The normalized complex}

Let $n_{\pm}$ be the number of positive/negative crossings in $W_T$ and let $w=n_+-n_-$ be the writhe of $W_T$. 
We define the normalized complex 
\begin{equation}\label{eq:nrmlzcplx}
\fF(W_T) := q^{2w}\overline{\fF}(W_T) . 
\end{equation}

\subsection{Topological invariance}

\begin{thm}\label{thm:invariance}
For every tangle diagram $T$ the homotopy type of $\fF(W_T)$ is invariant under the Reidemeister moves. 
\end{thm}

\begin{thm} 
  For every link $L$ the homology of $\fF(L)$ is a $\bZ$-graded supermodule over $\bZ$ whose
  graded Euler characteristic equals the Jones polynomial. 
\end{thm}

\begin{proof}[Proof of~\fullref{thm:invariance}] 
The following is immediate.
\begin{lem}
  For $\beta_{\pm}$ a positive/negative crossing let $W_t$ and $W_b$ be the following tangles in $F$-form:
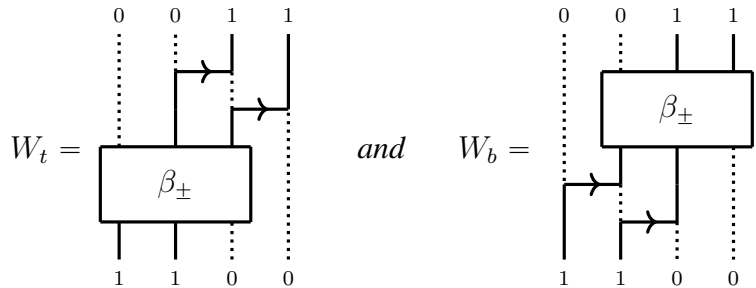
\[ W_t = \
\tikz[very thick,baseline={([yshift=-.8ex]current bounding box.center)}]{
\draw (-1,-.5) -- (1,-.5);\draw (-1, .5) -- (1, .5);
\draw (-1,-.5) -- (-1,.5);\draw (1,-.5) -- (1,.5);
\node at (0,0) {$\beta_{\pm}$};
\draw (-.75,-1)node[below]{\tiny $1$} -- (-.75,-.5);
\draw (0,-1)node[below]{\tiny $1$} -- (0,-.5);
\draw [dotted] (.75,-1)node[below]{\tiny $0$}  -- (.75,-.5);
\draw [dotted] (-.75,.5) -- (-.75,2)node[above]{\tiny $0$};
\draw (0,1) -- (0,.5);\draw (0,1) -- (0,1.5);\draw[rdirected=.55] (.75,1.5) -- (0,1.5);\draw (.75,1.5) -- (.75,2)node[above]{\tiny $1$};
\draw (.75,1) -- (.75,.5);\draw[rdirected=.55] (1.5,1) -- (.75,1);\draw (1.5,1) -- (1.5,2)node[above]{\tiny $1$};
\draw [dotted] (1.5,-1)node[below]{\tiny $0$} -- (1.5,1);
\draw [dotted] (.75,1) -- (.75,1.5);\draw [dotted] (0,1.5) -- (0,2)node[above]{\tiny $0$};
} 
\mspace{30mu}\text{and}\mspace{30mu}
W_b = \
\tikz[very thick,baseline={([yshift=-.8ex]current bounding box.center)}]{
\draw (1,.5) -- (-1,.5);\draw (1, -.5) -- (-1,-.5);
\draw (1,.5) -- (1,-.5);\draw (-1,.5) -- (-1,-.5);
\node at (0,0) {$\beta_{\pm}$};
\draw (.75,1)node[above]{\tiny $1$} -- (.75,.5);
\draw (0,1)node[above]{\tiny $1$} -- (0,.5);
\draw [dotted] (-.75,1)node[above]{\tiny $0$}  -- (-.75,.5);
\draw [dotted] (.75,-.5) -- (.75,-2)node[below]{\tiny $0$};
\draw (0,-1) -- (0,-.5);\draw (0,-1) -- (0,-1.5);\draw[rdirected=.55] (0,-1.5) -- (-.75,-1.5);\draw (-.75,-1.5) -- (-.75,-2)node[below]{\tiny $1$};
\draw (-.75,-1) -- (-.75,-.5);\draw[rdirected=.55] (-.75,-1) -- (-1.5,-1);\draw (-1.5,-1) -- (-1.5,-2)node[below]{\tiny $1$};
\draw [dotted] (-1.5,1)node[above]{\tiny $0$} -- (-1.5,-1);
\draw [dotted] (-.75,-1) -- (-.75,-1.5);\draw [dotted] (0,-1.5) -- (0,-2)node[below]{\tiny $0$};
} 
\]
Then the complexes $\fF(W_t)$ and $\fF(W_b)$  are isomorphic. 
\end{lem}

%
%
\begin{lem}[Reidemeister I]\label{lem:reidIa}
Consider diagrams $D_1^+$ and $D_0$ that differ as below. 
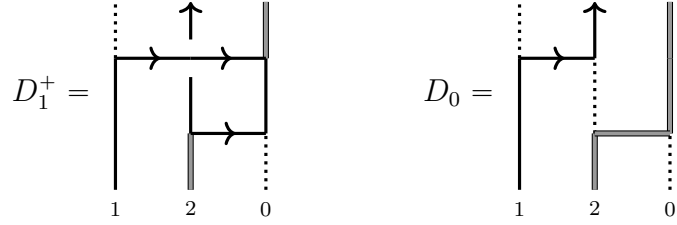
\[
D_1^+ = 
\tikz[very thick,scale=.25,baseline={([yshift=.8ex]current bounding box.center)}]{
\draw (-2,-5)node[below] {\tiny $1$} to (-2,2);\draw [dotted](-2,2) to (-2,5); 
\draw [doubleblack]( 2,-5)node[below] {\tiny $2$} to (2,-2);\draw (2,-2) to (2,1);\draw[->] (2,3) to (2,5);   
\draw [dotted] (6,-5)node[below] {\tiny $0$} to (6,-2);\draw (6,-2) to (6,2); \draw[doubleblack] (6,2) to (6,5);  
	\draw [directed=.6] (2,-2) to (6,-2);
	\draw [directed=.6] (-2,2) to (2,2);\draw [directed=.55] (2,2) to (6,2);
}\mspace{80mu}
D_0 = 
\tikz[very thick,scale=.25,baseline={([yshift=.8ex]current bounding box.center)}]{
\draw (-2,-5)node[below] {\tiny $1$} to (-2,2);\draw [dotted](-2,2) to (-2,5); 
\draw [doubleblack]( 2,-5)node[below] {\tiny $2$} to (2,-2);\draw[dotted] (2,-2) to (2,2);\draw[->] (2,2) to (2,5);   
\draw [dotted] (6,-5)node[below] {\tiny $0$} to (6,-2);\draw[doubleblack] (6,-2) to (6,2); \draw[doubleblack] (6,2) to (6,5);  
	\draw [doubleblack] (2,-2) to (6,-2);
	\draw [directed=.6] (-2,2) to (2,2);
}
\]
Then $\fF(D_1^+)$ and $\fF(D_0)$ are isomorphic in $\Kom_{/h}\big( (1,2,0),(0,1,2) \big)$. 
\end{lem}
\begin{proof}
  We have 
\[
\begin{tikzpicture}
  \node at (-3.9, 0) {$\overline{\fF}(D_1^+) = q^{-1}F_1F_2F_2(1,2,0)$};
  \node at ( 3, 0) {$F_1F_1F_2(1,2,0).$};
  \draw[semithick,directed=1] (-1.2,0) -- (1.5,0);
\node at (.1,.65) {
\tikz[very thick,baseline={([yshift=.9ex]current bounding box.center)}]{
       \draw[mygreen] (0,-.5) .. controls (0,0) and (1,0) .. (1,.5);
       \draw[myblue] (1,-.5) .. controls (1,0) and (0,0) .. (0,.5);
       \draw[myblue] (1.75,-.5) -- (1.75,.5);
       \node at (-.15,-.425) {\tiny $1$}; \node at (1.15,-.425) {\tiny $2$}; \node at (1.9,-.425) {\tiny $2$};}
};  
\end{tikzpicture}  
\]

The first term is isomorphic to $F_1F^{(2)}_2(1,2,0)\oplus q^{-2}F_1F^{(2)}_2(1,2,0)$ via the map 
\[
\begin{tikzpicture}
  \node at (-5.5, 0) {$F_1F^{(2)}_2(1,2,0)\oplus q^{-2}F_1F^{(2)}_2(1,2,0)$};
  \node at ( 4, 0) {$q^{-1}F_1F_2^2(1,2,0) ,$};
  \draw[semithick,directed=1] (-2,0) -- (2.3,0);
  \node at (0.25,-.25) {\tiny $\simeq$};
\node at (.1,.9) {
\Bigg(
\tikz[very thick,scale=1.2,baseline={([yshift=.8ex]current bounding box.center)}]{
  \draw[mygreen] (-.75,-.5 )node[below] {\textcolor{black}{\tiny $1$}} -- (-.75,.5);
  \draw[doubleblue] (0,-.5)node[below] {\tiny $2$} -- (0,0);
  \draw[myblue] (-.015,0) .. controls (0,.12) and (-.25,.12) .. (-.3,.5);
  \draw[myblue] (.015,0) .. controls (0,.12) and (.25,.12) ..  (.3,.5);
}\  , \  
\tikz[very thick,scale=1.2,baseline={([yshift=.8ex]current bounding box.center)}]{
  \draw[mygreen] (-.75,-.5 )node[below] {\textcolor{black}{\tiny $1$}} -- (-.75,.5);
  \draw[doubleblue] (0,-.5)node[below] {\tiny $2$} -- (0,0);
  \draw[myblue] (-.015,0) .. controls (0,.12) and (-.25,.12) .. (-.3,.5);
  \draw[myblue] (.015,0) .. controls (0,.12) and (.25,.12) ..  (.3,.5);
\node[fill=myblue,circle,inner sep=2pt] at (-.225,.25) {}; 
}\Bigg)
};  
\end{tikzpicture}  
\]
while for the second term there is an isomorphism
\[
\begin{tikzpicture}
  \node at (-5.5, 0) {$F_2F_1F_2(1,2,0)$};
  \node at (0.65, 0) {$F_1F^{(2)}_2(1,2,0),$};
  \draw[semithick,directed=1] (-3.75,0) -- (-1.1,0);
\node at (-2.5,-.25) {\tiny $\simeq$};
\node at (-2.5,1) {
\tikz[very thick,scale=1.2,baseline={([yshift=.8ex]current bounding box.center)}]{
 \draw[mygreen] (-.225,-.5) node[below] {\textcolor{black}{\tiny $1$}}  .. controls (-.25,-.1) and (-.75,-.1) ..  (-.75,.5);
  \draw[doubleblue] (0,.5) -- (0,0);
  \draw[myblue] (-.015,0) .. controls (0,-.12) and (-.25,.06) .. (-.75,-.5)node[below] {\textcolor{black}{\tiny $2$}};
  \draw[myblue] (.015,0) .. controls (0,-.12) and (.25,-.12) ..  (.3,-.5)node[below] {\textcolor{black}{\tiny $2$}}; }
};  
\end{tikzpicture} 
\]
so that $\overline{\fF}(D_1^+)$ is isomorphic to the complex 
\[
\begin{tikzpicture}
  \node at (-7, 0) {$\begin{pmatrix}
  F_1F_2^{(2)}(1,2,0)
  \\[1.5ex]
  q^{-2}F_1F_2^{(2)}(1,2,0) 
\end{pmatrix}$};
  \node at (1.65, 0) {$F_1F_2^{(2)}(1,2,0) .$};
  \draw[semithick,directed=1] (-4.75,0) -- (-0.1,0);
\node at (-2.5,0.75) {
$\bigg(
  t_{2,1}\tikz[very thick,baseline={([yshift=-.2ex]current bounding box.center)}]{
       \draw[mygreen] (0,-.5) -- (0,.5);
       \draw[doubleblue] (.5,-.5) -- (.5,.5);
       \node at (-.15,-.425) {\tiny $1$}; \node at (.675,-.425) {\tiny $2$};   
  } ,\ \ 
  t_{1,2}\tikz[very thick,baseline={([yshift=-.2ex]current bounding box.center)}]{
       \draw[mygreen] (0,-.5)node[midway,fill=mygreen,circle,inner sep=2pt]{} -- (0,.5);
       \draw[doubleblue] (.5,-.5) -- (.5,.5);
       \node at (-.15,-.425) {\tiny $1$}; \node at (.675,-.425) {\tiny $2$};   
  }\Big)$
};  
\end{tikzpicture}
\]

By Gaussian elimination one gets that the complex $\overline{\fF}(D_1^+)$ is homotopy equivalent to the one term complex
$q^{-2}F_1F_2^{(2)}(1,2,0)$ concentrated in homological degree zero, which after normalization 
is $\fF(D_0)$.  
\end{proof}
The other types of Reidemeister I move can be verified similarly. For example, replacing the positive crossing by
a negative crossing in~\fullref{lem:reidIa} and using the inverses of the various isomorphisms above 
results in a complex isomorphic to $\overline{\fF}(D_1^-)$ that is homotopy equivalent to the 1-term complex
$q^2F_1F_2^{(2)}(1,2,0)$ concentrated in homological degree zero.

\begin{lem}[Reidemeister IIa]
Consider diagrams $D_1$ and $D_0$ that differ as below. 
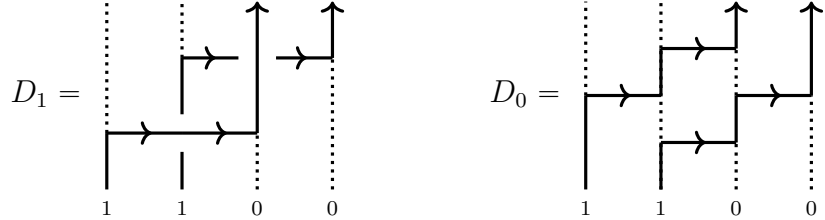
\[
D_1 = 
\tikz[very thick,scale=.25,baseline={([yshift=.8ex]current bounding box.center)}]{
\draw (-2,-5)node[below] {\tiny $1$} to (-2,-2);\draw [dotted](-2,-2) to (-2,5); 
\draw ( 2,-5)node[below] {\tiny $1$} to (2,-3);\draw (2,-1) to (2,2);\draw [dotted] (2,2) to (2,5);   
\draw [dotted] (6,-5)node[below] {\tiny $0$} to (6,-2);\draw[->] (6,-2) to (6,5);  
\draw [dotted] (10,-5)node[below] {\tiny $0$} to (10,2);\draw[->] (10,2) to (10,5);        
        \draw [directed=.6] (-2,-2) to (2,-2);\draw [directed=.6] (2,-2) to (6,-2);
	\draw [directed=.6] (2,2) to (5,2);\draw [directed=.55] (7,2) to (10,2);
}\mspace{80mu}
D_0 = 
\tikz[very thick,scale=.25,baseline={([yshift=.8ex]current bounding box.center)}]{
\draw (-2,-5)node[below] {\tiny $1$} to (-2,0);\draw [dotted](-2,0) to (-2,5); 
\draw ( 2,-5)node[below] {\tiny $1$} to (2,-2.5);\draw [dotted]( 2,-5) to (2,2.5);
      \draw (2,0) to (2,2.5);\draw [dotted] (2,2.5) to (2,5);   
\draw [dotted] (6,-5)node[below] {\tiny $0$} to (6,-2.5);\draw (6,-2.5) to (6,0);
      \draw[dotted] (6,0) to (6,2.5);\draw[->] (6,2.5) to (6,5);  
\draw [dotted] (10,-5)node[below] {\tiny $0$} to (10,0);\draw[->] (10,0) to (10,5);
\draw [directed=.6](2,-2.5) to (6,-2.5);\draw [directed=.6] (2,2.5) to (6,2.5);
\draw [directed=.6] (-2,0) to (2,0);\draw [directed=.6] (6,0) to (10,0);
}
\]
Then $\fF(D_1)$ and $\fF(D_0)$ are isomorphic in $\Kom_{/h}\big( (1,1,0,0),(0,0,1,1) \big)$.
  \end{lem}
\begin{proof}
In the following we write $\mu$ instead of $(1,1,0,0)$.  
The complex $\overline{\fF}(D_1)$ is 
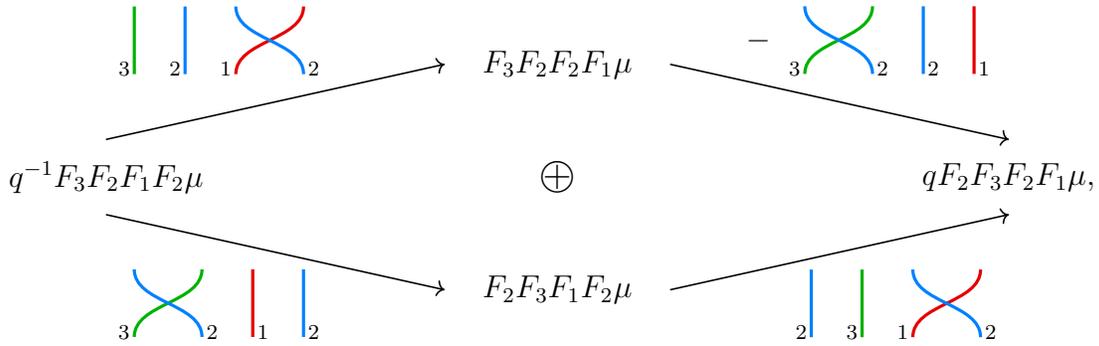
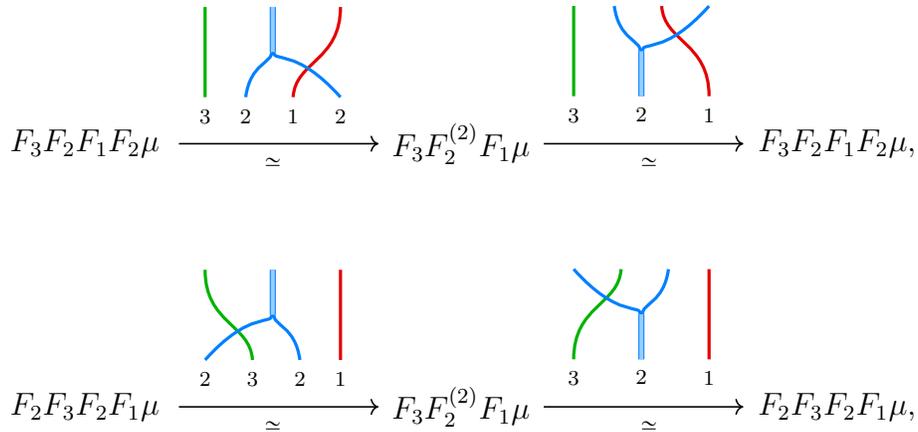
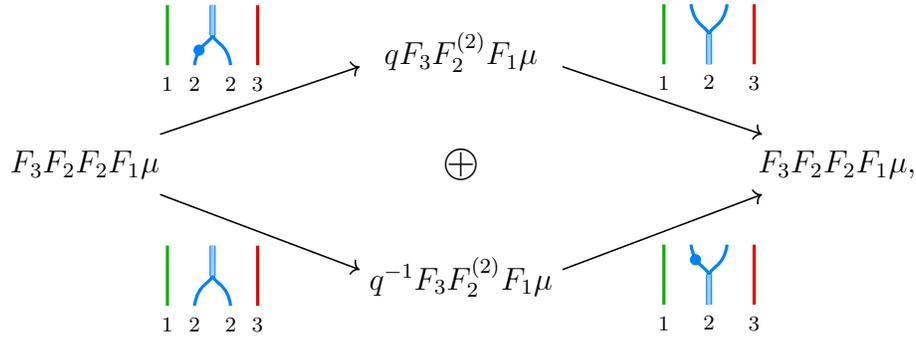
\[
\begin{tikzpicture}
  \node at (-6, 0) {$q^{-1}F_3F_2F_1F_2\mu$};
  \node at ( 0, 1.5) {$F_3F_2F_2F_1\mu$};
  \node at ( 0,-1.5) {$F_2F_3F_1F_2\mu$}; 
  \node at ( 6, 0) {$qF_2F_3F_2F_1\mu,$};
  \node at (0,0) {$\bigoplus$};
  \draw[semithick,directed=1] (-6, 0.5) -- (-1.5, 1.5);
  \draw[semithick,directed=1] (-6,-0.5) -- (-1.5,-1.5);
  \draw[semithick,directed=1] (1.5, 1.5) -- (6,  0.5);
  \draw[semithick,directed=1] (1.5,-1.5) -- (6, -0.5);
\node at (-4.5,1.75) {
   \tikz[very thick,scale=0.9,baseline={([yshift=.9ex]current bounding box.center)}]{
     \draw[mygreen] (-1.5,-.5) -- (-1.5,.5);\node at (-1.65,-.425) {\tiny $3$};
     \draw[myblue] (-.75,-.5) -- (-.75,.5);\node at (-.90,-.425) {\tiny $2$};
     \draw[myred] (0,-.5) .. controls (0,0) and (1,0) .. (1,.5); \node at (-.15,-.425) {\tiny $1$};
     \draw[myblue] (1,-.5) .. controls (1,0) and (0,0) .. (0,.5);  \node at (1.15,-.425) {\tiny $2$};}  };  
\node at (4.2,1.75) {$-$
   \tikz[very thick,scale=0.9,baseline={([yshift=-.2ex]current bounding box.center)}]{
     \draw[mygreen] (0,-.5) .. controls (0,0) and (1,0) .. (1,.5); \node at (-.15,-.425) {\tiny $3$};
     \draw[myblue] (1,-.5) .. controls (1,0) and (0,0) .. (0,.5);  \node at (1.15,-.425) {\tiny $2$};  
     \draw[myblue] (1.75,-.5) -- (1.75,.5);\node at (1.90,-.425) {\tiny $2$};
     \draw[myred] (2.5,-.5) -- (2.5,.5);\node at (2.65,-.425) {\tiny $1$};}  };
\node at (-4.5,-1.75) {
   \tikz[very thick,scale=0.9,baseline={([yshift=-.2ex]current bounding box.center)}]{
     \draw[mygreen] (0,-.5) .. controls (0,0) and (1,0) .. (1,.5); \node at (-.15,-.425) {\tiny $3$};
     \draw[myblue] (1,-.5) .. controls (1,0) and (0,0) .. (0,.5);  \node at (1.15,-.425) {\tiny $2$};  
     \draw[myred] (1.75,-.5) -- (1.75,.5);\node at (1.90,-.425) {\tiny $1$};
     \draw[myblue] (2.5,-.5) -- (2.5,.5);\node at (2.65,-.425) {\tiny $2$};}  };
\node at (4.5,-1.75) {
   \tikz[very thick,scale=0.9,baseline={([yshift=.9ex]current bounding box.center)}]{
     \draw[myblue] (-1.5,-.5) -- (-1.5,.5);\node at (-1.65,-.425) {\tiny $2$};
     \draw[mygreen] (-.75,-.5) -- (-.75,.5);\node at (-.90,-.425) {\tiny $3$};
     \draw[myred] (0,-.5) .. controls (0,0) and (1,0) .. (1,.5); \node at (-.15,-.425) {\tiny $1$};
     \draw[myblue] (1,-.5) .. controls (1,0) and (0,0) .. (0,.5);  \node at (1.15,-.425) {\tiny $2$};}  };  
\end{tikzpicture}
\]
From the isomorphisms
\[
\begin{tikzpicture}
  \node at (-5, 0) {$F_3F_2F_1F_2\mu$};
  \node at ( 0, 0) {$F_3F_2^{(2)}F_1\mu$};
  \node at ( 5, 0) {$F_3F_2F_1F_2\mu,$};  
  \draw[semithick,directed=1] (-3.75,0) -- (-1.1,0);
  \draw[semithick,directed=1] (1.1,0)  --  (3.75,0);
\node at (-2.5,-.25) {\tiny $\simeq$};\node at (2.5,-.25) {\tiny $\simeq$};
\node at (-2.5,1) {
\tikz[very thick,scale=1.2,baseline={([yshift=.8ex]current bounding box.center)}]{
 \draw[myred] (.225,-.5) node[below] {\textcolor{black}{\tiny $1$}}  .. controls (.25,-.1) and (.75,-.1) ..  (.75,.5);
  \draw[doubleblue] (0,.5) -- (0,0);
  \draw[myblue] (.015,0) .. controls (0,-.12) and (.25,.06) .. (.75,-.5)node[below] {\textcolor{black}{\tiny $2$}};
  \draw[myblue] (-.015,0) .. controls (0,-.12) and (-.25,-.12) ..  (-.3,-.5)node[below] {\textcolor{black}{\tiny $2$}};
  \draw[mygreen] (-.75,-.5) node[below] {\textcolor{black}{\tiny $3$}} -- (-.75,.5);}
};  
\node at (2.4,1) {
\tikz[very thick,scale=1.2,baseline={([yshift=.8ex]current bounding box.center)}]{
 \draw[myred] (.225,.5)  .. controls (.25,.1) and (.75,.1) ..  (.75,-.5)node[below] {\textcolor{black}{\tiny $1$}};
  \draw[doubleblue] (0,-.5)node[below] {\textcolor{black}{\tiny $2$}} -- (0,0);
  \draw[myblue] (.015,0) .. controls (0,.12) and (.25,-.06) .. (.75,.5);
  \draw[myblue] (-.015,0) .. controls (0,.12) and (-.25,.12) ..  (-.3,.5);
  \draw[mygreen] (-.75,.5) -- (-.75,-.5) node[below] {\textcolor{black}{\tiny $3$}};}
};
\end{tikzpicture}
\]
\[
\begin{tikzpicture}
  \node at (-5, 0) {$F_2F_3F_2F_1\mu$};
  \node at ( 0, 0) {$F_3F_2^{(2)}F_1\mu$};
  \node at ( 5, 0) {$F_2F_3F_2F_1\mu,$};  
  \draw[semithick,directed=1] (-3.75,0) -- (-1.1,0);
  \draw[semithick,directed=1] (1.1,0)  --  (3.75,0);
\node at (-2.5,-.25) {\tiny $\simeq$};\node at (2.5,-.25) {\tiny $\simeq$};
\node at (-2.5,1) {
  \tikz[very thick,scale=1.2,baseline={([yshift=.8ex]current bounding box.center)}]{
 \draw[mygreen] (-.225,-.5) node[below] {\textcolor{black}{\tiny $3$}}  .. controls (-.25,-.1) and (-.75,-.1) ..  (-.75,.5);
  \draw[doubleblue] (0,.5) -- (0,0);
  \draw[myblue] (-.015,0) .. controls (0,-.12) and (-.25,.06) .. (-.75,-.5)node[below] {\textcolor{black}{\tiny $2$}};
  \draw[myblue] (.015,0) .. controls (0,-.12) and (.25,-.12) ..  (.3,-.5)node[below] {\textcolor{black}{\tiny $2$}};
  \draw[myred] (.75,-.5) node[below] {\textcolor{black}{\tiny $1$}} -- (.75,.5);}
};  
\node at (2.4,1) {
\tikz[very thick,scale=1.2,baseline={([yshift=.8ex]current bounding box.center)}]{
 \draw[mygreen] (-.225,.5)  .. controls (-.25,.1) and (-.75,.1) ..  (-.75,-.5)node[below] {\textcolor{black}{\tiny $3$}};
  \draw[doubleblue] (0,-.5)node[below] {\textcolor{black}{\tiny $2$}} -- (0,0);
  \draw[myblue] (-.015,0) .. controls (0,.12) and (-.25,-.06) .. (-.75,.5);
  \draw[myblue] (.015,0) .. controls (0,.12) and (.25,.12) ..  (.3,.5);
  \draw[myred] (.75,.5) -- (.75,-.5) node[below] {\textcolor{black}{\tiny $1$}};}  
};
\end{tikzpicture}
\]
and
\[
\begin{tikzpicture}
  \node at (-5, 0) {$F_3F_2F_2F_1\mu$};
  \node at ( 0, 1.5) {$qF_3F_2^{(2)}F_1\mu$};
  \node at ( 0,-1.5) {$q^{-1}F_3F_2^{(2)}F_1\mu$};
  \node at ( 5, 0) {$F_3F_2F_2F_1\mu,$};  
  \node at ( 0, 0) {$\bigoplus$};
  \draw[semithick,directed=1] (-4,.4) -- (-1.35,1.3);
  \draw[semithick,directed=1] (1.35,1.3)  --  (4,.4);
  \draw[semithick,directed=1] (-4,-.4) -- (-1.35,-1.4);
  \draw[semithick,directed=1] (1.35,-1.4)  --  (4,-.4);
\node at (-3.3,1.5) {
\tikz[very thick,scale=.8,baseline={([yshift=.8ex]current bounding box.center)}]{
  \draw[mygreen] (-.75,-.5 )node[below] {\textcolor{black}{\tiny $1$}} -- (-.75,.5);
  \draw[doubleblue] (0,.5) -- (0,0);
  \draw[myblue] (-.015,0) .. controls (0,-.12) and (-.25,-.12) .. (-.3,-.5)node[below] {\textcolor{black}{\tiny $2$}};
  \draw[myblue] (.015,0) .. controls (0,-.12) and (.25,-.12) ..  (.3,-.5)node[below] {\textcolor{black}{\tiny $2$}};
  \node[fill=myblue,circle,inner sep=1.5pt] at (-.225,-.25) {};
  \draw[myred] (.75,-.5 )node[below] {\textcolor{black}{\tiny $3$}} -- (.75,.5);}
};  
\node at (-3.3,-1.7) {
\tikz[very thick,scale=.8,baseline={([yshift=.8ex]current bounding box.center)}]{
  \draw[mygreen] (-.75,-.5 )node[below] {\textcolor{black}{\tiny $1$}} -- (-.75,.5);
  \draw[doubleblue] (0,.5) -- (0,0);
  \draw[myblue] (-.015,0) .. controls (0,-.12) and (-.25,-.12) .. (-.3,-.5)node[below] {\textcolor{black}{\tiny $2$}};
  \draw[myblue] (.015,0) .. controls (0,-.12) and (.25,-.12) ..  (.3,-.5)node[below] {\textcolor{black}{\tiny $2$}};
  \draw[myred] (.75,-.5 )node[below] {\textcolor{black}{\tiny $3$}} -- (.75,.5);}
};
\node at (3.3,1.5) {
\tikz[very thick,scale=.8,baseline={([yshift=.8ex]current bounding box.center)}]{
  \draw[mygreen] (-.75,.5 ) -- (-.75,-.5)node[below] {\textcolor{black}{\tiny $1$}};
  \draw[doubleblue] (0,-.5)node[below] {\tiny $2$} -- (0,0);
  \draw[myblue] (-.015,0) .. controls (0,.12) and (-.25,.12) .. (-.3,.5);
  \draw[myblue] (.015,0) .. controls (0,.12) and (.25,.12) ..  (.3,.5);
  \draw[myred] (.75,.5 ) -- (.75,-.5)node[below] {\textcolor{black}{\tiny $3$}};}
};
\node at (3.3,-1.7) {
\tikz[very thick,scale=.8,baseline={([yshift=.8ex]current bounding box.center)}]{
  \draw[mygreen] (-.75,.5 ) -- (-.75,-.5)node[below] {\textcolor{black}{\tiny $1$}};
  \draw[doubleblue] (0,-.5)node[below] {\tiny $2$} -- (0,0);
  \draw[myblue] (-.015,0) .. controls (0,.12) and (-.25,.12) .. (-.3,.5);
  \draw[myblue] (.015,0) .. controls (0,.12) and (.25,.12) ..  (.3,.5);
  \draw[myred] (.75,.5 ) -- (.75,-.5)node[below] {\textcolor{black}{\tiny $3$}};
  \node[fill=myblue,circle,inner sep=1.5pt] at (-.225,.25) {};}
};  
\end{tikzpicture}
\]
and simplifying the maps using the relations in $\fR(\nu)$ one gets that $\fF(D_1)$ is isomorphic to the complex
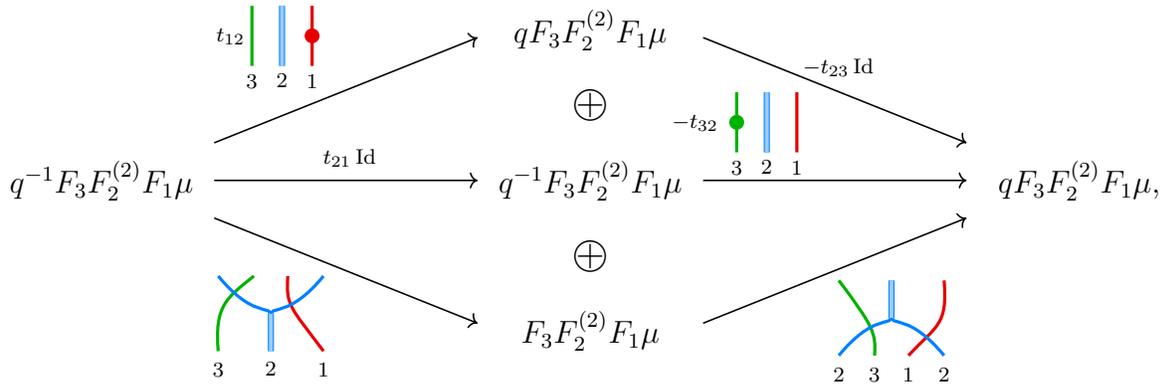
\[
\begin{tikzpicture}
  \node at (-6.5, 0) {$q^{-1}F_3F_2^{(2)}F_1\mu$};
  \node at ( 0, 2) {$qF_3F_2^{(2)}F_1\mu$};
   \node at (0,1) {$\bigoplus$};  
  \node at ( 0, 0) {$q^{-1} F_3F_2^{(2)}F_1\mu$};
   \node at (0,-1) {$\bigoplus$};  
  \node at ( 0, -2) {$F_3F_2^{(2)}F_1\mu$};
  \node at ( 6.5, 0) {$qF_3F_2^{(2)}F_1\mu,$};
  \draw[semithick,directed=1] (-5, .5) -- (-1.5, 1.9);
  \draw[semithick,directed=1] (-5, 0) -- (-1.5,  0);
  \draw[semithick,directed=1] (-5,-.5) -- (-1.5,-1.9);
  \draw[semithick,directed=1] (1.5,1.9) -- (5, .5);
  \draw[semithick,directed=1] (1.5, 0) -- (5,  0);
  \draw[semithick,directed=1] (1.5,-1.9) -- (5,-.5);
\node at (-4.25,1.75) {\tiny{$t_{12}\!\!\!$}
  \tikz[very thick,scale=1,baseline={([yshift=.8ex]current bounding box.center)}]{
  \draw[myred] (.4,-.4) node[below] {\textcolor{black}{\tiny $1$}}  --  (.4,.4)node[midway,fill=myred,circle,inner sep=2pt]{};
  \draw[mygreen] (-.4,-.4)node[below] {\textcolor{black}{\tiny $3$}}  --  (-.4,.4);
  \draw[doubleblue] (0,-.4)node[below] {\tiny $2$}  -- (0,.4);}
};
\node at (-4.25,-2) {
\tikz[very thick,scale=1,baseline={([yshift=.8ex]current bounding box.center)}]{
  \draw[myred] (.225,.5) .. controls (.2,.1) and (.25,.1) ..  (.7,-.5)node[below] {\textcolor{black}{\tiny $1$}};
  \draw[mygreen] (-.225,.5) .. controls (-.6,.2) and (-.75,.1) ..  (-.7,-.5)node[below] {\textcolor{black}{\tiny $3$}};
  \draw[doubleblue] (0,-.5)node[below] {\tiny $2$} -- (0,0);
  \draw[myblue] (-.015,0) .. controls (0,.1) and (-.25,-.05) .. (-.7,.5);
  \draw[myblue] ( .015,0) .. controls (0,.1) and ( .25,-.05) ..  (.7,.5);}
};
\node at (-3.2,.3) {\tiny{$t_{21}\id$}
};
\node at (2,.6) {\tiny{$-t_{32}$}
  \tikz[very thick,scale=1,baseline={([yshift=.8ex]current bounding box.center)}]{
  \draw[myred] (.4,-.4) node[below] {\textcolor{black}{\tiny $1$}}  --  (.4,.4);
  \draw[mygreen] (-.4,-.4)node[below] {\textcolor{black}{\tiny $3$}}  --  (-.4,.4)node[midway,fill=mygreen,circle,inner sep=2pt]{};
  \draw[doubleblue] (0,-.4)node[below] {\tiny $2$}  -- (0,.4);
  }};
\node at (3.3,1.5) {\tiny{$-t_{23}\id$}
};
\node at (4,-2.05) {
\tikz[very thick,scale=1,baseline={([yshift=.8ex]current bounding box.center)}]{
  \draw[myred] (.225,-.5)node[below] {\textcolor{black}{\tiny $1$}} .. controls (.6,-.1) and (.75,-.1) ..  (.7,.5);
  \draw[mygreen] (-.225,-.5)node[below] {\textcolor{black}{\tiny $3$}} .. controls (-.2,-.2) and (-.25,-.1) ..  (-.7,.5);
  \draw[doubleblue] (0,.5) -- (0,0);
  \draw[myblue] (-.015,0) .. controls (0,-.1) and (-.25,.05) .. (-.7,-.5)node[below] {\textcolor{black}{\tiny $2$}};
  \draw[myblue] ( .015,0) .. controls (0,-.1) and ( .25,.05) ..  (.7,-.5)node[below] {\textcolor{black}{\tiny $2$}};}
};
\end{tikzpicture}
\]
By Gaussian elimination of the acyclic two-term complexes $q^{-1}F_3F_2^{(2)}F_1\mu\xra{t_{21}\id}q^{-1}F_3F_2^{(2)}F_1\mu$
and $qF_3F_2^{(2)}F_1\mu\xra{-t_{23}\id}qF_3F_2^{(2)}F_1\mu$
one obtains that $\overline{\fF}(D_1)$ is homotopy equivalent to the complex
\[
\begin{tikzpicture}
  \node at (-2.5,0) {$0$};
  \node at ( 0,0) {$F_3F_2^{(2)}F_1\mu$};
  \node at ( 2.5,-.02) {$0 ,$};
  \draw[semithick,directed=1] (-2.2, 0) -- (-1.25, 0);
  \draw[semithick,directed=1] (1.25, 0) -- (2.2,  0);
\end{tikzpicture}  
\]
with the middle-term in homological degree zero.
\end{proof}

\begin{lem}
[Reidemeister III]
Consider diagrams $D_L$ and $D_R$ that differ as below.
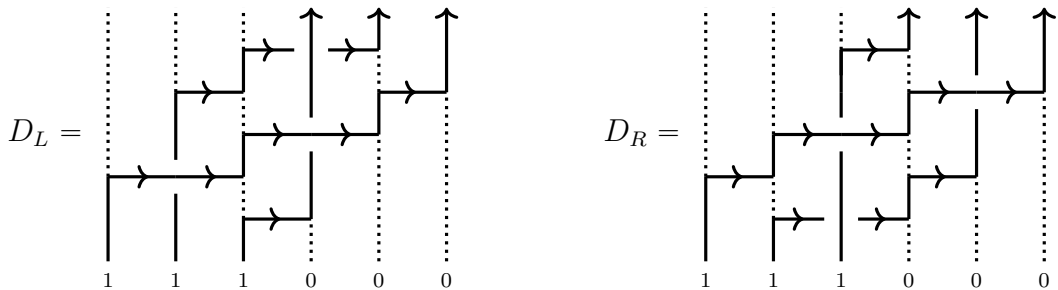
\[
D_L = 
\tikz[very thick,scale=.225,baseline={([yshift=.8ex]current bounding box.center)}]{
\draw (-2,-7.5)node[below] {\tiny $1$} to (-2,-2.5);\draw [dotted](-2,-2.5) to (-2,7.5); 
\draw ( 2,-7.5)node[below] {\tiny $1$} to (2,-3.5);\draw (2,-1.5) to (2,2.5);\draw [dotted] (2,2.5) to (2,7.5);   
\draw (6,-7.5)node[below] {\tiny $1$} to (6,-5);\draw[dotted] (6,-5) to (6,-2.5);\draw (6,-2.5) to (6,0);
   \draw[dotted] (6,0) to (6,2.5);\draw (6,2.5) to (6,5);\draw[dotted] (6,5) to (6,7.5);
\draw [dotted] (10,-7.5)node[below] {\tiny $0$} to (10,-5);\draw (10,-5) to (10,-1);\draw[->] (10,1) to (10,7.5);        
\draw [dotted] (14,-7.5)node[below] {\tiny $0$} to (14,0);\draw (14,0) to (14,2.5);\draw [dotted] (14,2.5) to (14,5);
   \draw [->] (14,5) to (14,7.5);
\draw [dotted] (18,-7.5)node[below] {\tiny $0$} to (18,2.5);\draw [->] (18,2.5) to (18,7.5);
\draw [directed=.6] (-2,-2.5) to (2,-2.5);\draw [directed=.6] (2,-2.5) to (6,-2.5);
\draw [directed=.6] (6,0) to (10,0);\draw [directed=.6] (10,0) to (14,0);
\draw [directed=.6] (6,5) to (9,5);\draw [directed=.55] (11,5) to (14,5);
\draw [directed=.55] (6,-5) to (10,-5);\draw [directed=.55] (2,2.5) to (6,2.5);\draw [directed=.55] (14,2.5) to (18,2.5);
}
\mspace{80mu}
D_R = 
\tikz[very thick,scale=.225,baseline={([yshift=.8ex]current bounding box.center)}]{
\draw (-2,-7.5)node[below] {\tiny $1$} to (-2,-2.5);\draw [dotted](-2,-2.5) to (-2,7.5); 
\draw ( 2,-7.5)node[below] {\tiny $1$} to (2,-5);\draw[dotted] (2,-5) to (2,-2.5);\draw (2,-2.5) to (2,0);
\draw[dotted] (2,0) to (2,7.5);   
\draw (6,-7.5)node[below] {\tiny $1$} to (6,-1);\draw (6,1) to (6,2.5);\draw (6,3.5) to (6,5);
  \draw (6,2.5) to (6,5);\draw[dotted] (6,5) to (6,7.5);
\draw [dotted] (10,-7.5)node[below] {\tiny $0$} to (10,-5);\draw (10,-5) to (10,-2.5);\draw[dotted] (10,-2.5) to (10,0);
  \draw (10,0) to (10,2.5);\draw[dotted] (10,2.5) to (10,5);\draw[->] (10,5) to (10,7.5);   
\draw [dotted] (14,-7.5)node[below] {\tiny $0$} to (14,-2.5);\draw (14,-2.5) to (14,1.5);\draw [->] (14,3.5) to (14,7.5);
\draw [dotted] (18,-7.5)node[below] {\tiny $0$} to (18,2.5);\draw [->] (18,2.5) to (18,7.5);
\draw [directed=.6] (2,0) to (6,0);\draw [directed=.6] (6,0) to (10,0);
\draw [directed=.6] (10,2.5) to (14,2.5);\draw [directed=.6] (14,2.5) to (18,2.5);
\draw [directed=.6] (2,-5) to (5,-5);\draw [directed=.55] (7,-5) to (10,-5);
\draw [directed=.55] (6,5) to (10,5);\draw [directed=.55] (10,-2.5) to (14,-2.5);\draw [directed=.55] (-2,-2.5) to (2,-2.5);
}
\]
Then $\fF(D_L)$ and $\fF(D_R)$ are isomorphic in $\Kom_{/h}\big( (1,1,1,0,0,0),(0,0,0,1,1,1) \big)$.
\end{lem}
\begin{proof}
  The proof is inspired by~\cite[Lemma 7.9]{putyra} (see also~\cite[\S4.3.3]{putyraThesis} for further details). The complex associated to $D_L$ is the mapping cone of the map
\[
q^{-1}\fF\left(
\tikz[very thick,scale=.225,baseline={([yshift=0ex]current bounding box.center)}]{
\draw (-2,-7.5)node[below] {\tiny $1$} to (-2,-2.5);\draw [dotted](-2,-2.5) to (-2,7.5); 
\draw ( 2,-7.5)node[below] {\tiny $1$} to (2,-3.5);\draw (2,-1.5) to (2,3);\draw [dotted] (2,3) to (2,7.5);   
\draw (6,-7.5)node[below] {\tiny $1$} to (6,-5);\draw[dotted] (6,-5) to (6,-2.5);\draw (6,-2.5) to (6,1);
   \draw[dotted] (6,1) to (6,3);\draw (6,3) to (6,5);\draw[dotted] (6,5) to (6,7.5);
   \draw [dotted] (10,-7.5)node[below] {\tiny $0$} to (10,-5);\draw (10,-5) to (10,-1);\draw[dotted] (10,-1) to (10,1);
   \draw[->] (10,1) to (10,7.5);        
\draw [dotted] (14,-7.5)node[below] {\tiny $0$} to (14,-1);\draw (14,-1) to (14,2.5);\draw [dotted] (14,2.5) to (14,5);
   \draw [->] (14,5) to (14,7.5);
\draw [dotted] (18,-7.5)node[below] {\tiny $0$} to (18,2.5);\draw [->] (18,2.5) to (18,7.5);
\draw [directed=.6] (-2,-2.5) to (2,-2.5);\draw [directed=.6] (2,-2.5) to (6,-2.5);
\draw [directed=.6] (6,1) to (10,1);\draw [directed=.6] (10,-1) to (14,-1);
\draw [directed=.6] (6,5) to (9,5);\draw [directed=.55] (11,5) to (14,5);
\draw [directed=.55] (6,-5) to (10,-5);\draw [directed=.55] (2,3) to (6,3);\draw [directed=.55] (14,2.5) to (18,2.5);
}\right)
\xra{\dotsm\tikz[very thick,scale=.8,baseline={([yshift=0ex]current bounding box.center)}]{
       \draw[myblue] (0,-.5) .. controls (0,0) and (1,0) .. (1,.5);
       \draw[mygreen] (1,-.5) .. controls (1,0) and (0,0) .. (0,.5);
       \node at (-.15,-.425) {\tiny $3$}; \node at (1.15,-.425) {\tiny $4$}; 
  }\dotsm}
\fF\left(
\tikz[very thick,scale=.225,baseline={([yshift=.8ex]current bounding box.center)}]{
\draw (-2,-7.5)node[below] {\tiny $1$} to (-2,-2.5);\draw [dotted](-2,-2.5) to (-2,7.5); 
\draw ( 2,-7.5)node[below] {\tiny $1$} to (2,-3.5);\draw (2,-1.5) to (2,2.5);\draw [dotted] (2,2.5) to (2,7.5);   
\draw (6,-7.5)node[below] {\tiny $1$} to (6,-5);\draw[dotted] (6,-5) to (6,-2.5);\draw (6,-2.5) to (6,-1);
   \draw[dotted] (6,0) to (6,2.5);\draw (6,2.5) to (6,5);\draw[dotted] (6,5) to (6,7.5);
   \draw [dotted] (10,-7.5)node[below] {\tiny $0$} to (10,-5);\draw (10,-5) to (10,-1);\draw[doubleblack] (10,-1) to (10,1);
   \draw[->] (10,1) to (10,7.5);        
\draw [dotted] (14,-7.5)node[below] {\tiny $0$} to (14,1);\draw (14,1) to (14,3);\draw [dotted] (14,3) to (14,5);
   \draw [->] (14,5) to (14,7.5);
\draw [dotted] (18,-7.5)node[below] {\tiny $0$} to (18,3);\draw [->] (18,3) to (18,7.5);
\draw [directed=.6] (-2,-2.5) to (2,-2.5);\draw [directed=.6] (2,-2.5) to (6,-2.5);
\draw [directed=.6] (6,-1) to (10,-1);\draw [directed=.6] (10,1) to (14,1);
\draw [directed=.6] (6,5) to (9,5);\draw [directed=.55] (11,5) to (14,5);
\draw [directed=.55] (6,-5) to (10,-5);\draw [directed=.55] (2,2.5) to (6,2.5);\draw [directed=.55] (14,3) to (18,3);
}
\right)
\]
An easy exercise shows that the second complex is isomorphic to the complex
\[
\fF\left(\tikz[very thick,scale=.225,baseline={([yshift=0ex]current bounding box.center)}]{
\draw (-2,-7.5)node[below] {\tiny $1$} to (-2,-3);\draw [dotted](-2,-3) to (-2,7.5); 
\draw ( 2,-7.5)node[below] {\tiny $1$} to (2,-5);\draw[dotted] (2,-5) to (2,-3);\draw (2,-3) to (2,-1);
\draw[dotted] (2,-1) to (2,7.5);   
\draw (6,-7.5)node[below] {\tiny $1$} to (6,-1);\draw[doubleblack] (6,-1) to (6,1);\draw (6,1) to (6,2.5);\draw (6,3.5) to (6,5);
  \draw (6,2.5) to (6,5);\draw[dotted] (6,5) to (6,7.5);
\draw [dotted] (10,-7.5)node[below] {\tiny $0$} to (10,-5);\draw (10,-5) to (10,-2.5);\draw[dotted] (10,-2.5) to (10,1);
  \draw (10,1) to (10,3);\draw[dotted] (10,3) to (10,5);\draw[->] (10,5) to (10,7.5);   
\draw [dotted] (14,-7.5)node[below] {\tiny $0$} to (14,-2.5);\draw (14,-2.5) to (14,2);\draw [->] (14,4) to (14,7.5);
\draw [dotted] (18,-7.5)node[below] {\tiny $0$} to (18,3);\draw [->] (18,3) to (18,7.5);
\draw [directed=.6] (2,-1) to (6,-1);\draw [directed=.6] (6,1) to (10,1);
\draw [directed=.6] (10,3) to (14,3);\draw [directed=.6] (14,3) to (18,3);
\draw [directed=.6] (2,-5) to (5,-5);\draw [directed=.55] (7,-5) to (10,-5);
\draw [directed=.55] (6,5) to (10,5);\draw [directed=.55] (10,-2.5) to (14,-2.5);\draw [directed=.55] (-2,-3) to (2,-3);
}\right)
\]
In~\cite[\S4.3.3]{putyraThesis} it is explained in detail how to use an isomorphism like this
together with the maps associated to two Reidemeister 2
moves on the first complex to prove that $\fF(D_L)$ is homotopy equivalent to $\fF(D_R)$. 
\end{proof}

This finishes the proof of~\fullref{thm:invariance}. 
\end{proof}

\subsection{Not even Khovanov homology}

We now show that for links the invariant $\cH(L)$ is distinct from even Khovanov homology and shares common properties with odd Khovanov homology.

\subsubsection{Reduced homology}\label{sec:red}

\begin{thm}\label{thm:reduced}
For every link $L$ there is an invariant $H_{\text{reduced}}(L)$ with the property 
\[
H(L) \simeq q H_{\text{reduced}}(L) \oplus q^{-1}H_{\text{reduced}}(L) .
\]
\end{thm}

The proof of~\fullref{thm:reduced} follows a reasoning analogous to the proof of Theorem 3.2.A.
in~\cite{shumakovitch-torsion}, for the analogous decomposition for Khovanov homology over $\bZ/2\bZ$ in terms of 
reduced Khovanov homology.   

Before proving the theorem we do some preparation.
Recall that for $D$ a diagram of $L$ the chain groups of $\fF(D)$ are the various $\Bbbk$-supervector spaces
$\Hom_{\cR^\Lambda}(F_{can},F(W))$, where $W$ runs over all the resolutions of $D$. 

If we write $F_{can}=F_{i_1^{(2)}i_2^{(2)}\dotsm i_{k}^{(2)}}$ then $\Hom_{\cR^\Lambda}(F_{can}, F_{i_1i_1i_2i_2\dotsm i_ki_k})$ is
spanned by
\[
\Bigg\{
\tikz[very thick,scale=.9,baseline={([yshift=0.5ex]current bounding box.center)}]{
  \draw[doubleblue] (0,-.75)node[below] {\tiny $i_1$} -- (0,-.5);
  \draw[myblue] (-.015,-.5) .. controls (0,-.48) and (-.25,-.48) .. (-.3,.5);
  \draw[myblue] (.015,-.5) .. controls (0,.-.48) and (.25,-.48) ..  (.3,.5);
  \node[fill=myblue,circle,inner sep=1.5pt] at (-.16,-.33) {};
\node at (-.36,-.48) {\tiny $\delta_1$};
  \draw[doublegreen] (1,-.75)node[below] {\tiny $i_2$} -- (1,-.15);
  \draw[mygreen] (.985,-.15) .. controls (.85,0) and (.75,.12) .. (.7,.5);
  \draw[mygreen] (1.015,-.15) .. controls (1.15,0) and (1.25,.12) ..  (1.3,.5);
  \node[fill=mygreen,circle,inner sep=1.5pt] at (0.87,0) {};
\node at (.67,-.1) {\tiny $\delta_2$};
  \node at (2,-.25) {$\dotsm$};
  \draw[doublered] (3,-.75)node[below] {\tiny $i_k$} -- (3,0);
  \draw[myred] (2.985,0) .. controls (3,.12) and (2.75,.12) .. (2.7,.5);
  \draw[myred] (3.015,0) .. controls (3,.12) and (3.25,.12) ..  (3.3,.5);
  \node[fill=myred,circle,inner sep=1.5pt] at (2.775,.25) {};
\node at (2.67,.1) {\tiny $\delta_k$};
}\mspace{10mu},\mspace{5mu}
\delta_1,\dotsc,\delta_k\in\{0,1\}
\Bigg\} .
\]
Introduce linear maps $X$ and $\Delta$ on $\Hom_{\cR^\Lambda}(F_{can}, F_{i_1i_1i_2i_2\dotsm i_ki_k})$
as follows.
Map $\Delta$ is defined on the factors as 
\[
\Delta
\Bigg( \dotsm
\tikz[very thick,scale=.9,baseline={([yshift=-.5ex]current bounding box.center)}]{
\draw[doublegreen] (1,-.75) -- (1,-.15);
  \draw[mygreen] (.985,-.15) .. controls (.85,0) and (.75,.12) .. (.7,.5);
  \draw[mygreen] (1.015,-.15) .. controls (1.15,0) and (1.25,.12) ..  (1.3,.5);
  \node[fill=mygreen,circle,inner sep=1.5pt] at (0.87,0) {};
}\dotsm\Bigg) =  0 , 
\mspace{60mu}
\Delta\Bigg( \dotsm
\tikz[very thick,scale=.9,baseline={([yshift=-.5ex]current bounding box.center)}]{
\draw[doublegreen] (1,-.75) -- (1,-.15);
  \draw[mygreen] (.985,-.15) .. controls (.85,0) and (.75,.12) .. (.7,.5);
  \draw[mygreen] (1.015,-.15) .. controls (1.15,0) and (1.25,.12) ..  (1.3,.5);
}\dotsm\Bigg) =
\dotsm
\tikz[very thick,scale=.9,baseline={([yshift=-.5ex]current bounding box.center)}]{
\draw[doublegreen] (1,-.75) -- (1,-.15);
  \draw[mygreen] (.985,-.15) .. controls (.85,0) and (.75,.12) .. (.7,.5);
  \draw[mygreen] (1.015,-.15) .. controls (1.15,0) and (1.25,.12) ..  (1.3,.5);
  \node[fill=mygreen,circle,inner sep=1.5pt] at (0.87,0) {}; }
\dotsm , 
\]
and extended to  $\Hom_{\cR^\Lambda}(F_{can}, F_{i_1i_1i_2i_2\dotsm i_ki_k})$ using the Leibniz rule.
The map $X$ is defined by
\[
X\Bigg(
\tikz[very thick,scale=.9,baseline={([yshift=0.5ex]current bounding box.center)}]{
  \draw[doubleblue] (0,-.75)node[below] {\tiny $i_1$} -- (0,-.5);
  \draw[myblue] (-.015,-.5) .. controls (0,-.48) and (-.25,-.48) .. (-.3,.5);
  \draw[myblue] (.015,-.5) .. controls (0,.-.48) and (.25,-.48) ..  (.3,.5);
  \node[fill=myblue,circle,inner sep=1.5pt] at (-.16,-.33) {};
\node at (-.36,-.48) {\tiny $\delta_1$};
  \draw[doublegreen] (1,-.75)node[below] {\tiny $i_2$} -- (1,-.15);
  \draw[mygreen] (.985,-.15) .. controls (.85,0) and (.75,.12) .. (.7,.5);
  \draw[mygreen] (1.015,-.15) .. controls (1.15,0) and (1.25,.12) ..  (1.3,.5);
  \node[fill=mygreen,circle,inner sep=1.5pt] at (0.87,0) {};
\node at (.67,-.1) {\tiny $\delta_2$};
  \node at (2,-.25) {$\dotsm$};
  \draw[doublered] (3,-.75)node[below] {\tiny $i_k$} -- (3,0);
  \draw[myred] (2.985,0) .. controls (3,.12) and (2.75,.12) .. (2.7,.5);
  \draw[myred] (3.015,0) .. controls (3,.12) and (3.25,.12) ..  (3.3,.5);
  \node[fill=myred,circle,inner sep=1.5pt] at (2.775,.25) {};
\node at (2.67,.1) {\tiny $\delta_k$};
}
\Bigg) =
\begin{cases}\ 
\tikz[very thick,scale=.9,baseline={([yshift=0.5ex]current bounding box.center)}]{
  \draw[doubleblue] (0,-.75)node[below] {\tiny $i_1$} -- (0,-.5);
  \draw[myblue] (-.015,-.5) .. controls (0,-.48) and (-.25,-.48) .. (-.3,.5);
  \draw[myblue] (.015,-.5) .. controls (0,.-.48) and (.25,-.48) ..  (.3,.5);
  \draw[doublegreen] (1,-.75)node[below] {\tiny $i_2$} -- (1,-.15);
  \draw[mygreen] (.985,-.15) .. controls (.85,0) and (.75,.12) .. (.7,.5);
  \draw[mygreen] (1.015,-.15) .. controls (1.15,0) and (1.25,.12) ..  (1.3,.5);
  \node[fill=mygreen,circle,inner sep=1.5pt] at (0.87,0) {};
\node at (.67,-.1) {\tiny $\delta_2$};
  \node at (2,-.25) {$\dotsm$};
  \draw[doublered] (3,-.75)node[below] {\tiny $i_k$} -- (3,0);
  \draw[myred] (2.985,0) .. controls (3,.12) and (2.75,.12) .. (2.7,.5);
  \draw[myred] (3.015,0) .. controls (3,.12) and (3.25,.12) ..  (3.3,.5);
  \node[fill=myred,circle,inner sep=1.5pt] at (2.775,.25) {};
\node at (2.67,.1) {\tiny $\delta_k$};
} & \text{if }\ \delta_1 = 1 ,  \\[1.5ex]
\mspace{60mu}0 & \text{otherwise}.
\end{cases}
\]

Since 
\[
\Hom_{\cR^\Lambda}(F_{can},F(W)) \simeq \Hom_{\cR^\Lambda}(F_{can}, F_{i_1i_1i_2i_2\dotsm i_ki_k})
\times
\Hom_{\cR^\Lambda}(F_{i_1i_1i_2i_2\dotsm i_ki_k},F(W))
\]
the maps $\Delta$ and $X$ induce maps on $\Hom_{\cR^\Lambda}(F_{can},F(W))$, denoted by the same symbols. 

\begin{lem}
  Both maps $X$ and $\Delta$ commute with the differential of $\fF(D)$, $\Delta^2=0$, and moreover
  \(   X\Delta+\Delta X = \id_{\fF(D)}  \). 
\end{lem}
\begin{proof}
Straightforward. 
\end{proof}

\begin{proof}[Proof of~\fullref{thm:reduced}]
We have that $\Delta$ is acyclic and therefore
\[
\fF(D) \simeq  \ker(\Delta) \oplus q^2\ker(\Delta) , 
\]
and so the claim follows by setting $\fF_{\text{reduced}}(D)= q\ker(\Delta)$. 
\end{proof}

\subsubsection{A chronological Frobenius algebra}

We now examine the behavior of the functor $\fF$ under merge and splitting of circles.
First define maps $\imath$ and $\varepsilon$,  
\begin{equation*} 
\fF\left(
  \tikz[very thick,scale=.2,baseline={([yshift=0ex]current bounding box.center)}]{
\draw [doubleblack] (-2,0)node[below] {\tiny $2$} to (-2,2);\draw (-2,2) to (-2,5);\draw[dotted] (-2,5) to (-2,7); 
\draw[dotted] (2,0)node[below] {\tiny $0$} to (2,2);
    \draw (2,2) to (2,5); \draw[doubleblack] (2,5) to (2,7);\draw (-2,2) to (2,2);\draw (-2,5) to (2,5);
}\right)\mspace{10mu}
\tikz[semithick,scale=.2,baseline={([yshift=-.8ex]current bounding box.center)}]{
\draw[->] (0,1) to (5,1);\draw[->] (5,-1) to (0,-1);
\node at (2.5,2.1) {\tiny $\varepsilon$};\node at (2.5,-2.1) {\tiny $\imath$};
}\mspace{10mu}
\fF\left(
\tikz[very thick,scale=.2,baseline={([yshift=0ex]current bounding box.center)}]{
\draw [doubleblack] (-2,0)node[below] {\tiny $2$} to (-2,3.5);\draw[dotted] (-2,3.5) to (-2,7); 
\draw[dotted] (2,0)node[below] {\tiny $0$} to (2,3.5);\draw[doubleblack] (2,3.5) to (2,7);
\draw[doubleblack] (-2,3.5) to (2,3.5);
}\right)
\end{equation*}
as 
\[
\imath\colon F_1^{(2)}(2,0)
\xra{\ \ 
  \tikz[very thick,scale=.85,baseline={([yshift=0ex]current bounding box.center)}]{
  \draw[doublered] (1,-.5) -- (1,0);
  \draw[myred] (.985,0) .. controls (1,.12) and (.75,.12) .. (.7,.5);
  \draw[myred] (1.015,0) .. controls (1,.12) and (1.25,.12) ..  (1.3,.5);
  \node at (1.25,-.4) {\tiny $1$}; }\ \
}F_1^2(2,0)
\mspace{80mu}
\varepsilon\colon F_1^2(2,0)
\xra{\ \
\tikz[very thick,scale=.8,baseline={([yshift=.8ex]current bounding box.center)}]{
  \draw[doublered] (0,.5) -- (0,0);
  \draw[myred] (-.015,0) .. controls (0,-.12) and (-.25,-.12) .. (-.3,-.5);
  \draw[myred] (.015,0) .. controls (0,-.12) and (.25,-.12) ..  (.3,-.5);
  \node at (0.25,.44) {\tiny $1$};}
  \ \ }
F_1^{(2)}(2,0) . 
\]
Note that, contrary to~\cite{ORS}, $p(\imath)=1$ and $p(\varepsilon)=0$. 

\smallskip

We now consider the following two cases~\eqref{eq:frob} and~\eqref{eq:frobb}  below.
\begin{equation}\tag{$a$}\label{eq:frob}
\fF\left(
  \tikz[very thick,scale=.2,baseline={([yshift=0ex]current bounding box.center)}]{
\draw [doubleblack] (-2,-4.5)node[below] {\tiny $2$} to (-2,2);\draw (-2,2) to (-2,5);\draw[dotted] (-2,5) to (-2,6.5); 
\draw [doubleblack] (2,-4.5)node[below] {\tiny $2$} to (2,-3);\draw (2,-3) to (2,0);\draw[dotted] (2,0) to (2,2);
  \draw (2,2) to (2,5); \draw[doubleblack] (2,5) to (2,6.5);   
\draw [dotted] (6,-4.5)node[below] {\tiny $0$} to (6,-3);\draw (6,-3) to (6,0);\draw[doubleblack] (6,0) to (6,6.5);  
	\draw (2,-3) to (6,-3);\draw (2,0) to (6,0);
        \draw (-2,2) to (2,2);\draw (-2,5) to (2,5);
}\right)\mspace{10mu}
\tikz[semithick,scale=.2,baseline={([yshift=-.8ex]current bounding box.center)}]{
\draw[ ->] (0,1) to (5,1);\draw[->] (5,-1) to (0,-1);
\node at (2.5,2.1) {\tiny $\mu$};\node at (2.5,-2.1) {\tiny $\delta$};
}\mspace{10mu}
\fF\left(
\tikz[very thick,scale=.2,baseline={([yshift=0ex]current bounding box.center)}]{
\draw [doubleblack] (-2,-4.5)node[below] {\tiny $2$} to (-2,0);\draw (-2,0) to (-2,5);\draw[dotted] (-2,5) to (-2,6.5); 
\draw [doubleblack] (2,-4.5)node[below] {\tiny $2$} to (2,-3);\draw (2,-3) to (2,0);\draw[doubleblack] (2,0) to (2,2);
  \draw (2,2) to (2,5); \draw[doubleblack] (2,5) to (2,6.5);   
\draw [dotted] (6,-4.5)node[below] {\tiny $0$} to (6,-3);\draw (6,-3) to (6,2);\draw[doubleblack] (6,2) to (6,6.5);  
	\draw (2,-3) to (6,-3);\draw (2,2) to (6,2);
        \draw (-2,0) to (2,0);\draw (-2,5) to (2,5);
}\right)
\end{equation}

The maps $\mu$ and $\delta$ are given by 
\[
\mu\colon F_1^2F_2^2(2,2,0)
\xra{
   \tikz[very thick,scale=0.9,baseline={([yshift=.9ex]current bounding box.center)}]{
     \draw[myblue] (1.75,-.5) -- (1.75,.5);\node at (1.9,-.425) {\tiny $2$};
     \draw[myred] (-.75,-.5) -- (-.75,.5);\node at (-.90,-.425) {\tiny $1$};
     \draw[myred] (0,-.5) .. controls (0,0) and (1,0) .. (1,.5); \node at (-.15,-.425) {\tiny $1$};
     \draw[myblue] (1,-.5) .. controls (1,0) and (0,0) .. (0,.5);  \node at (1.15,-.425) {\tiny $2$};}
}
F_1F_2F_1F_2(2,2,0) , 
\]
and
\[
\delta\colon F_1F_2F_1F_2(2,2,0)
\xra{
   \tikz[very thick,scale=0.9,baseline={([yshift=.9ex]current bounding box.center)}]{
     \draw[myblue] (1.75,-.5) -- (1.75,.5);\node at (1.9,-.425) {\tiny $2$};
     \draw[myred] (-.75,-.5) -- (-.75,.5);\node at (-.90,-.425) {\tiny $1$};
     \draw[myblue] (0,-.5) .. controls (0,0) and (1,0) .. (1,.5); \node at (-.15,-.425) {\tiny $2$};
     \draw[myred] (1,-.5) .. controls (1,0) and (0,0) .. (0,.5);  \node at (1.15,-.425) {\tiny $1$};}
}
F_1^2F_2^2(2,2,0) . 
\]
We have $p(\mu)=0$ and $p(\delta)=1$. Decomposing $F_1^2F_2^2(2,2,0)$ and $F_1F_2F_1F_2(2,2,0)$ into a direct sum
of several copies of $F_1^{(2)}F_2^{(2)}(2,2,0)$ with the appropriate grading shifts we fix bases 
\[
\Biggl\langle\
\tikz[very thick,scale=.85,baseline={([yshift=1.1ex]current bounding box.center)}]{
  \draw[doublered] (0,-.75)node[below] {\tiny $1$} -- (0,-.5);
  \draw[myred] (-.015,-.5) .. controls (0,-.48) and (-.25,-.48) .. (-.3,.5);
  \draw[myred] (.015,-.5) .. controls (0,.-.48) and (.25,-.48) ..  (.3,.5);
  \draw[doubleblue] (1,-.75)node[below] {\tiny $2$} -- (1,0);
  \draw[myblue] (.985,0) .. controls (1,.12) and (.75,.12) .. (.7,.5);
  \draw[myblue] (1.015,0) .. controls (1,.12) and (1.25,.12) ..  (1.3,.5);
  \node at (.5,-1.6) {\small $p=0$}; 
}\quad,\quad
\tikz[very thick,scale=.85,baseline={([yshift=1.1ex]current bounding box.center)}]{
  \draw[doublered] (0,-.75)node[below] {\tiny $1$} -- (0,-.5);
  \draw[myred] (-.015,-.5) .. controls (0,-.48) and (-.25,-.48) .. (-.3,.5);
  \draw[myred] (.015,-.5) .. controls (0,.-.48) and (.25,-.48) ..  (.3,.5);
  \node[fill=myred,circle,inner sep=1.5pt] at (-.18,-.25) {};
  \draw[doubleblue] (1,-.75)node[below] {\tiny $2$} -- (1,0);
  \draw[myblue] (.985,0) .. controls (1,.12) and (.75,.12) .. (.7,.5);
  \draw[myblue] (1.015,0) .. controls (1,.12) and (1.25,.12) ..  (1.3,.5);
  \node at (.5,-1.6) {\small $p=1$}; 
}\quad,\quad
\tikz[very thick,scale=.85,baseline={([yshift=1.1ex]current bounding box.center)}]{
  \draw[doublered] (0,-.75)node[below] {\tiny $1$} -- (0,-.5);
  \draw[myred] (-.015,-.5) .. controls (0,-.48) and (-.25,-.48) .. (-.3,.5);
  \draw[myred] (.015,-.5) .. controls (0,.-.48) and (.25,-.48) ..  (.3,.5);
  \draw[doubleblue] (1,-.75)node[below] {\tiny $2$} -- (1,0);
  \draw[myblue] (.985,0) .. controls (1,.12) and (.75,.12) .. (.7,.5);
  \draw[myblue] (1.015,0) .. controls (1,.12) and (1.25,.12) ..  (1.3,.5);
  \node[fill=myblue,circle,inner sep=1.5pt] at (0.775,.25) {};
  \node at (.5,-1.6) {\small $p=1$}; 
}\quad,\quad
\tikz[very thick,scale=.85,baseline={([yshift=1.1ex]current bounding box.center)}]{
  \draw[doublered] (0,-.75)node[below] {\tiny $1$} -- (0,-.5);
  \draw[myred] (-.015,-.5) .. controls (0,-.48) and (-.25,-.48) .. (-.3,.5);
  \draw[myred] (.015,-.5) .. controls (0,.-.48) and (.25,-.48) ..  (.3,.5);
  \node[fill=myred,circle,inner sep=1.5pt] at (-.18,-.25) {};
  \draw[doubleblue] (1,-.75)node[below] {\tiny $2$} -- (1,0);
  \draw[myblue] (.985,0) .. controls (1,.12) and (.75,.12) .. (.7,.5);
  \draw[myblue] (1.015,0) .. controls (1,.12) and (1.25,.12) ..  (1.3,.5);
  \node[fill=myblue,circle,inner sep=1.5pt] at (0.775,.25) {};
  \node at (.5,-1.6) {\small $p=0$}; 
}\ \Biggr\rangle   
\]
of $F_1^2F_2^2(2,2,0)$, and 
\[
\Biggl\langle\
\tikz[very thick,scale=.85,baseline={([yshift=1.1ex]current bounding box.center)}]{
  \draw[doublered] (0,-.75)node[below] {\tiny $1$} -- (0,-.5);
  \draw[myred] (-.015,-.5) .. controls (0,-.48) and (-.25,-.48) .. (-.3,.5);
  \draw[myred] (.015,-.5) .. controls (0,.-.48) and (.25,-.48) ..  (.7,.5);
  \draw[doubleblue] (1,-.75)node[below] {\tiny $2$} -- (1,0);
  \draw[myblue] (.985,0) .. controls (.8,.12) and (.5,.12) .. (.3,.5);
  \draw[myblue] (1.015,0) .. controls (1,.12) and (1.25,.12) ..  (1.3,.5);
  \node at (.5,-1.6) {\small $p=0$}; 
}
\quad,\quad
\tikz[very thick,scale=.85,baseline={([yshift=1.1ex]current bounding box.center)}]{
  \draw[doublered] (0,-.75)node[below] {\tiny $1$} -- (0,-.5);
  \draw[myred] (-.015,-.5) .. controls (0,-.48) and (-.25,-.48) .. (-.3,.5);
  \draw[myred] (.015,-.5) .. controls (0,.-.48) and (.25,-.48) ..  (.7,.5);
  \node[fill=myred,circle,inner sep=1.5pt] at (-.18,-.25) {};
  \draw[doubleblue] (1,-.75)node[below] {\tiny $2$} -- (1,0);
  \draw[myblue] (.985,0) .. controls (.8,.12) and (.5,.12) .. (.3,.5);
  \draw[myblue] (1.015,0) .. controls (1,.12) and (1.25,.12) ..  (1.3,.5);
  \node at (.5,-1.6) {\small $p=1$}; 
}\ \Biggr\rangle
\]
of $F_1F_2F_1F_2(2,2,0)$. 
Then we compute
\begingroup\allowdisplaybreaks
\begin{align*}
\delta \Biggl(\
\tikz[very thick,scale=.85,baseline={([yshift=0ex]current bounding box.center)}]{
  \draw[doublered] (0,-.75)node[below] {\tiny $1$} -- (0,-.5);
  \draw[myred] (-.015,-.5) .. controls (0,-.48) and (-.25,-.48) .. (-.3,.5);
  \draw[myred] (.015,-.5) .. controls (0,.-.48) and (.25,-.48) ..  (.7,.5);
  \draw[doubleblue] (1,-.75)node[below] {\tiny $2$} -- (1,0);
  \draw[myblue] (.985,0) .. controls (.8,.12) and (.5,.12) .. (.3,.5);
  \draw[myblue] (1.015,0) .. controls (1,.12) and (1.25,.12) ..  (1.3,.5);
}\Biggr) &=
- t_{12}\ 
\tikz[very thick,scale=.85,baseline={([yshift=0ex]current bounding box.center)}]{
  \draw[doublered] (0,-.75)node[below] {\tiny $1$} -- (0,-.5);
  \draw[myred] (-.015,-.5) .. controls (0,-.48) and (-.25,-.48) .. (-.3,.5);
  \draw[myred] (.015,-.5) .. controls (0,.-.48) and (.25,-.48) ..  (.3,.5);
  \node[fill=myred,circle,inner sep=1.5pt] at (-.18,-.25) {};
  \draw[doubleblue] (1,-.75)node[below] {\tiny $2$} -- (1,0);
  \draw[myblue] (.985,0) .. controls (1,.12) and (.75,.12) .. (.7,.5);
  \draw[myblue] (1.015,0) .. controls (1,.12) and (1.25,.12) ..  (1.3,.5);
} + t_{21}\ 
\tikz[very thick,scale=.85,baseline={([yshift=0ex]current bounding box.center)}]{
  \draw[doublered] (0,-.75)node[below] {\tiny $1$} -- (0,-.5);
  \draw[myred] (-.015,-.5) .. controls (0,-.48) and (-.25,-.48) .. (-.3,.5);
  \draw[myred] (.015,-.5) .. controls (0,.-.48) and (.25,-.48) ..  (.3,.5);
  \draw[doubleblue] (1,-.75)node[below] {\tiny $2$} -- (1,0);
  \draw[myblue] (.985,0) .. controls (1,.12) and (.75,.12) .. (.7,.5);
  \draw[myblue] (1.015,0) .. controls (1,.12) and (1.25,.12) ..  (1.3,.5);
  \node[fill=myblue,circle,inner sep=1.5pt] at (0.775,.25) {};
}
\\
\delta \Biggl(\
\tikz[very thick,scale=.85,baseline={([yshift=0ex]current bounding box.center)}]{
  \draw[doublered] (0,-.75)node[below] {\tiny $1$} -- (0,-.5);
  \draw[myred] (-.015,-.5) .. controls (0,-.48) and (-.25,-.48) .. (-.3,.5);
  \draw[myred] (.015,-.5) .. controls (0,.-.48) and (.25,-.48) ..  (.7,.5);
  \node[fill=myred,circle,inner sep=1.5pt] at (-.18,-.25) {};
  \draw[doubleblue] (1,-.75)node[below] {\tiny $2$} -- (1,0);
  \draw[myblue] (.985,0) .. controls (.8,.12) and (.5,.12) .. (.3,.5);
  \draw[myblue] (1.015,0) .. controls (1,.12) and (1.25,.12) ..  (1.3,.5);
}
\Biggr) &=
t_{21}\ 
\tikz[very thick,scale=.85,baseline={([yshift=0ex]current bounding box.center)}]{
  \draw[doublered] (0,-.75)node[below] {\tiny $1$} -- (0,-.5);
  \draw[myred] (-.015,-.5) .. controls (0,-.48) and (-.25,-.48) .. (-.3,.5);
  \draw[myred] (.015,-.5) .. controls (0,.-.48) and (.25,-.48) ..  (.3,.5);
  \node[fill=myred,circle,inner sep=1.5pt] at (-.18,-.25) {};
  \draw[doubleblue] (1,-.75)node[below] {\tiny $2$} -- (1,0);
  \draw[myblue] (.985,0) .. controls (1,.12) and (.75,.12) .. (.7,.5);
  \draw[myblue] (1.015,0) .. controls (1,.12) and (1.25,.12) ..  (1.3,.5);
  \node[fill=myblue,circle,inner sep=1.5pt] at (0.775,.25) {};
}
\end{align*}
\endgroup
and
\begingroup\allowdisplaybreaks
\begin{align*}
\mu\Biggl(\ 
\tikz[very thick,scale=.85,baseline={([yshift=0ex]current bounding box.center)}]{
  \draw[doublered] (0,-.75)node[below] {\tiny $1$} -- (0,-.5);
  \draw[myred] (-.015,-.5) .. controls (0,-.48) and (-.25,-.48) .. (-.3,.5);
  \draw[myred] (.015,-.5) .. controls (0,.-.48) and (.25,-.48) ..  (.3,.5);
  \draw[doubleblue] (1,-.75)node[below] {\tiny $2$} -- (1,0);
  \draw[myblue] (.985,0) .. controls (1,.12) and (.75,.12) .. (.7,.5);
  \draw[myblue] (1.015,0) .. controls (1,.12) and (1.25,.12) ..  (1.3,.5);
} \Biggr) =
\tikz[very thick,scale=.85,baseline={([yshift=0ex]current bounding box.center)}]{
  \draw[doublered] (0,-.75)node[below] {\tiny $1$} -- (0,-.5);
  \draw[myred] (-.015,-.5) .. controls (0,-.48) and (-.25,-.48) .. (-.3,.5);
  \draw[myred] (.015,-.5) .. controls (0,.-.48) and (.25,-.48) ..  (.7,.5);
  \draw[doubleblue] (1,-.75)node[below] {\tiny $2$} -- (1,0);
  \draw[myblue] (.985,0) .. controls (.8,.12) and (.5,.12) .. (.3,.5);
  \draw[myblue] (1.015,0) .. controls (1,.12) and (1.25,.12) ..  (1.3,.5);
}
&&\mu\Biggl(\
\tikz[very thick,scale=.85,baseline={([yshift=0ex]current bounding box.center)}]{
  \draw[doublered] (0,-.75)node[below] {\tiny $1$} -- (0,-.5);
  \draw[myred] (-.015,-.5) .. controls (0,-.48) and (-.25,-.48) .. (-.3,.5);
  \draw[myred] (.015,-.5) .. controls (0,.-.48) and (.25,-.48) ..  (.3,.5);
  \node[fill=myred,circle,inner sep=1.5pt] at (-.18,-.25) {};
  \draw[doubleblue] (1,-.75)node[below] {\tiny $2$} -- (1,0);
  \draw[myblue] (.985,0) .. controls (1,.12) and (.75,.12) .. (.7,.5);
  \draw[myblue] (1.015,0) .. controls (1,.12) and (1.25,.12) ..  (1.3,.5);
  \node[fill=myblue,circle,inner sep=1.5pt] at (0.775,.25) {};
}
\Biggr) &= 0
\\
\mu\biggl(\
\tikz[very thick,scale=.85,baseline={([yshift=0ex]current bounding box.center)}]{
  \draw[doublered] (0,-.75)node[below] {\tiny $1$} -- (0,-.5);
  \draw[myred] (-.015,-.5) .. controls (0,-.48) and (-.25,-.48) .. (-.3,.5);
  \draw[myred] (.015,-.5) .. controls (0,.-.48) and (.25,-.48) ..  (.3,.5);
  \node[fill=myred,circle,inner sep=1.5pt] at (-.18,-.25) {};
  \draw[doubleblue] (1,-.75)node[below] {\tiny $2$} -- (1,0);
  \draw[myblue] (.985,0) .. controls (1,.12) and (.75,.12) .. (.7,.5);
  \draw[myblue] (1.015,0) .. controls (1,.12) and (1.25,.12) ..  (1.3,.5);
}
\Biggr) =
\tikz[very thick,scale=.85,baseline={([yshift=0ex]current bounding box.center)}]{
  \draw[doublered] (0,-.75)node[below] {\tiny $1$} -- (0,-.5);
  \draw[myred] (-.015,-.5) .. controls (0,-.48) and (-.25,-.48) .. (-.3,.5);
  \draw[myred] (.015,-.5) .. controls (0,.-.48) and (.25,-.48) ..  (.7,.5);
  \node[fill=myred,circle,inner sep=1.5pt] at (-.18,-.25) {};
  \draw[doubleblue] (1,-.75)node[below] {\tiny $2$} -- (1,0);
  \draw[myblue] (.985,0) .. controls (.8,.12) and (.5,.12) .. (.3,.5);
  \draw[myblue] (1.015,0) .. controls (1,.12) and (1.25,.12) ..  (1.3,.5);
}
&&
\mu\Biggl(\
\tikz[very thick,scale=.85,baseline={([yshift=0ex]current bounding box.center)}]{
  \draw[doublered] (0,-.75)node[below] {\tiny $1$} -- (0,-.5);
  \draw[myred] (-.015,-.5) .. controls (0,-.48) and (-.25,-.48) .. (-.3,.5);
  \draw[myred] (.015,-.5) .. controls (0,.-.48) and (.25,-.48) ..  (.3,.5);
  \draw[doubleblue] (1,-.75)node[below] {\tiny $2$} -- (1,0);
  \draw[myblue] (.985,0) .. controls (1,.12) and (.75,.12) .. (.7,.5);
  \draw[myblue] (1.015,0) .. controls (1,.12) and (1.25,.12) ..  (1.3,.5);
  \node[fill=myblue,circle,inner sep=1.5pt] at (0.775,.25) {};
}
\Biggr) &=t_{12}t_{21}^{-1}\ 
\tikz[very thick,scale=.85,baseline={([yshift=0ex]current bounding box.center)}]{
  \draw[doublered] (0,-.75)node[below] {\tiny $1$} -- (0,-.5);
  \draw[myred] (-.015,-.5) .. controls (0,-.48) and (-.25,-.48) .. (-.3,.5);
  \draw[myred] (.015,-.5) .. controls (0,.-.48) and (.25,-.48) ..  (.7,.5);
  \node[fill=myred,circle,inner sep=1.5pt] at (-.18,-.25) {};
  \draw[doubleblue] (1,-.75)node[below] {\tiny $2$} -- (1,0);
  \draw[myblue] (.985,0) .. controls (.8,.12) and (.5,.12) .. (.3,.5);
  \draw[myblue] (1.015,0) .. controls (1,.12) and (1.25,.12) ..  (1.3,.5);
}
\end{align*}
\endgroup

Using this one sees that easily that $\mu\delta=0$, as in the case of the odd Khovanov homology of~\cite{ORS}.

Setting to 1 all $t_{ij}$ and renaming $\langle 1, a_1, a_2 ,a_1\wedge a_2 \rangle$ the basis
vectors  of $F_1^2F_2^2(2,0,0)$  and 
$\langle 1, a_1=a_2\rangle$ the basis vectors of $F_1F_2F_1F_2(2,0,0)$ one can give the maps
$\delta,\mu,\imath$ and $\varepsilon$
 a form that coincides with the corresponding maps in~\cite[\S1.1]{ORS}. 
 Note though, that while the parities of $\delta$ and $\mu$ coincide with the corresponding maps in~\cite{ORS},
 the parities of $\imath$ and $\varepsilon$ are reversed with respect to~\cite{ORS}.

\begin{equation}\tag{$b$}\label{eq:frobb}
\fF\left(
  \tikz[very thick,scale=.2,baseline={([yshift=0ex]current bounding box.center)}]{
\draw [doubleblack] (-2,-4.5)node[below] {\tiny $2$} to (-2,-3);\draw (-2,-3) to (-2,0);\draw[dotted] (-2,0) to (-2,6.5); 
\draw [dotted] (2,-4.5)node[below] {\tiny $0$} to (2,-3);\draw (2,-3) to (2,0);\draw[doubleblack] (2,0) to (2,2);
  \draw (2,2) to (2,5); \draw[dotted] (2,5) to (2,6.5);   
\draw [dotted] (6,-4.5)node[below] {\tiny $0$} to (6,2);\draw (6,2) to (6,5);\draw[doubleblack] (6,5) to (6,6.5);  
	\draw (-2,-3) to (2,-3);\draw (-2,0) to (2,0);
   \draw (2,2) to (6,2);\draw (2,5) to (6,5);
}\right)\mspace{10mu}
\tikz[semithick,scale=.2,baseline={([yshift=-.8ex]current bounding box.center)}]{
\draw[->] (0,1) to (5,1);\draw[->] (5,-1) to (0,-1);
\node at (2.5,2.1) {\tiny $\mu'$};\node at (2.5,-2.1) {\tiny $\delta'$};
}\mspace{10mu}
\fF\left(\tikz[very thick,scale=.2,baseline={([yshift=0ex]current bounding box.center)}]{
\draw [doubleblack] (-2,-4.5)node[below] {\tiny $2$} to (-2,-3);\draw (-2,-3) to (-2,2);\draw[dotted] (-2,2) to (-2,6.5); 
\draw [dotted] (2,-4.5)node[below] {\tiny $0$} to (2,-3);\draw (2,-3) to (2,0);\draw[dotted] (2,0) to (2,2);
  \draw (2,2) to (2,5); \draw[dotted] (2,5) to (2,6.5);   
\draw [dotted] (6,-4.5)node[below] {\tiny $0$} to (6,0);\draw (6,0) to (6,5);\draw[doubleblack] (6,5) to (6,6.5);  
   \draw (-2,-3) to (2,-3);\draw (-2,2) to (2,2);
   \draw (2,0) to (6,0);\draw (2,5) to (6,5);
}\right)
\end{equation}

The maps $\mu'$ and $\delta'$ are given by 
\[
\mu'\colon F_2^2F_1^2(2,0,0)
\xra{
   \tikz[very thick,scale=0.9,baseline={([yshift=.9ex]current bounding box.center)}]{
     \draw[myred] (1.75,-.5) -- (1.75,.5);\node at (1.9,-.425) {\tiny $1$};
     \draw[myblue] (-.75,-.5) -- (-.75,.5);\node at (-.90,-.425) {\tiny $2$};
     \draw[myblue] (0,-.5) .. controls (0,0) and (1,0) .. (1,.5); \node at (-.15,-.425) {\tiny $2$};
     \draw[myred] (1,-.5) .. controls (1,0) and (0,0) .. (0,.5);  \node at (1.15,-.425) {\tiny $1$};}
}
F_2F_1F_2F_1(2,0,0) , 
\]
and
\[
\delta'\colon F_2F_1F_2F_1(2,0,0)
\xra{
   \tikz[very thick,scale=0.9,baseline={([yshift=.9ex]current bounding box.center)}]{
     \draw[myred] (1.75,-.5) -- (1.75,.5);\node at (1.9,-.425) {\tiny $1$};
     \draw[myblue] (-.75,-.5) -- (-.75,.5);\node at (-.90,-.425) {\tiny $2$};
     \draw[myred] (0,-.5) .. controls (0,0) and (1,0) .. (1,.5); \node at (-.15,-.425) {\tiny $1$};
     \draw[myblue] (1,-.5) .. controls (1,0) and (0,0) .. (0,.5);  \node at (1.15,-.425) {\tiny $2$};}
}
F_2^2F_1^2(2,0,0) . 
\]
Proceeding as above we fix a basis
\[
\Biggl\langle\
\tikz[very thick,scale=.85,baseline={([yshift=1.1ex]current bounding box.center)}]{
  \draw[doubleblue] (0,-.75)node[below] {\tiny $2$} -- (0,-.5);
  \draw[myblue] (-.015,-.5) .. controls (0,-.48) and (-.25,-.48) .. (-.3,.5);
  \draw[myblue] (.015,-.5) .. controls (0,.-.48) and (.25,-.48) ..  (.7,.5);
  \draw[doublered] (1,-.75)node[below] {\tiny $1$} -- (1,0);
  \draw[myred] (.985,0) .. controls (.8,.12) and (.5,.12) .. (.3,.5);
  \draw[myred] (1.015,0) .. controls (1,.12) and (1.25,.12) ..  (1.3,.5);
  \node at (.5,-1.6) {\small $p=1$}; 
}
\quad,\quad
\tikz[very thick,scale=.85,baseline={([yshift=1.1ex]current bounding box.center)}]{
  \draw[doubleblue] (0,-.75)node[below] {\tiny $2$} -- (0,-.5);
  \draw[myblue] (-.015,-.5) .. controls (0,-.48) and (-.25,-.48) .. (-.3,.5);
  \draw[myblue] (.015,-.5) .. controls (0,.-.48) and (.25,-.48) ..  (.7,.5);
  \node[fill=myblue,circle,inner sep=1.5pt] at (-.18,-.25) {};
  \draw[doublered] (1,-.75)node[below] {\tiny $1$} -- (1,0);
  \draw[myred] (.985,0) .. controls (.8,.12) and (.5,.12) .. (.3,.5);
  \draw[myred] (1.015,0) .. controls (1,.12) and (1.25,.12) ..  (1.3,.5);
  \node at (.5,-1.6) {\small $p=0$}; 
}\ \Biggr\rangle
\]
of $F_2F_1F_2F_1(2,0,0)$ and
\[
\Biggl\langle\
\tikz[very thick,scale=.85,baseline={([yshift=1.1ex]current bounding box.center)}]{
  \draw[doubleblue] (0,-.75)node[below] {\tiny $2$} -- (0,-.5);
  \draw[myblue] (-.015,-.5) .. controls (0,-.48) and (-.25,-.48) .. (-.3,.5);
  \draw[myblue] (.015,-.5) .. controls (0,.-.48) and (.25,-.48) ..  (.3,.5);
  \draw[doublered] (1,-.75)node[below] {\tiny $1$} -- (1,0);
  \draw[myred] (.985,0) .. controls (1,.12) and (.75,.12) .. (.7,.5);
  \draw[myred] (1.015,0) .. controls (1,.12) and (1.25,.12) ..  (1.3,.5);
  \node at (.5,-1.6) {\small $p=0$}; 
}\quad,\quad
\tikz[very thick,scale=.85,baseline={([yshift=1.1ex]current bounding box.center)}]{
  \draw[doubleblue] (0,-.75)node[below] {\tiny $2$} -- (0,-.5);
  \draw[myblue] (-.015,-.5) .. controls (0,-.48) and (-.25,-.48) .. (-.3,.5);
  \draw[myblue] (.015,-.5) .. controls (0,.-.48) and (.25,-.48) ..  (.3,.5);
  \node[fill=myblue,circle,inner sep=1.5pt] at (-.18,-.25) {};
  \draw[doublered] (1,-.75)node[below] {\tiny $1$} -- (1,0);
  \draw[myred] (.985,0) .. controls (1,.12) and (.75,.12) .. (.7,.5);
  \draw[myred] (1.015,0) .. controls (1,.12) and (1.25,.12) ..  (1.3,.5);
  \node at (.5,-1.6) {\small $p=1$}; 
}\quad,\quad
\tikz[very thick,scale=.85,baseline={([yshift=1.1ex]current bounding box.center)}]{
  \draw[doubleblue] (0,-.75)node[below] {\tiny $2$} -- (0,-.5);
  \draw[myblue] (-.015,-.5) .. controls (0,-.48) and (-.25,-.48) .. (-.3,.5);
  \draw[myblue] (.015,-.5) .. controls (0,.-.48) and (.25,-.48) ..  (.3,.5);
  \draw[doublered] (1,-.75)node[below] {\tiny $1$} -- (1,0);
  \draw[myred] (.985,0) .. controls (1,.12) and (.75,.12) .. (.7,.5);
  \draw[myred] (1.015,0) .. controls (1,.12) and (1.25,.12) ..  (1.3,.5);
  \node[fill=myred,circle,inner sep=1.5pt] at (0.775,.25) {};
  \node at (.5,-1.6) {\small $p=1$}; 
}\quad,\quad
\tikz[very thick,scale=.85,baseline={([yshift=1.1ex]current bounding box.center)}]{
  \draw[doubleblue] (0,-.75)node[below] {\tiny $2$} -- (0,-.5);
  \draw[myblue] (-.015,-.5) .. controls (0,-.48) and (-.25,-.48) .. (-.3,.5);
  \draw[myblue] (.015,-.5) .. controls (0,.-.48) and (.25,-.48) ..  (.3,.5);
  \node[fill=myblue,circle,inner sep=1.5pt] at (-.18,-.25) {};
  \draw[doublered] (1,-.75)node[below] {\tiny $1$} -- (1,0);
  \draw[myred] (.985,0) .. controls (1,.12) and (.75,.12) .. (.7,.5);
  \draw[myred] (1.015,0) .. controls (1,.12) and (1.25,.12) ..  (1.3,.5);
  \node[fill=myred,circle,inner sep=1.5pt] at (0.775,.25) {};
  \node at (.5,-1.6) {\small $p=0$}; 
}\ \Biggr\rangle
\]
of $F_2^2F_1^2(2,2,0)$, to get
\begingroup\allowdisplaybreaks
\begin{align*}
\delta' \Biggl(\
\tikz[very thick,scale=.85,baseline={([yshift=0ex]current bounding box.center)}]{
  \draw[doubleblue] (0,-.75)node[below] {\tiny $2$} -- (0,-.5);
  \draw[myblue] (-.015,-.5) .. controls (0,-.48) and (-.25,-.48) .. (-.3,.5);
  \draw[myblue] (.015,-.5) .. controls (0,.-.48) and (.25,-.48) ..  (.7,.5);
  \draw[doublered] (1,-.75)node[below] {\tiny $1$} -- (1,0);
  \draw[myred] (.985,0) .. controls (.8,.12) and (.5,.12) .. (.3,.5);
  \draw[myred] (1.015,0) .. controls (1,.12) and (1.25,.12) ..  (1.3,.5);
}\Biggr) &= 
- t_{21}\ 
\tikz[very thick,scale=.85,baseline={([yshift=0ex]current bounding box.center)}]{
  \draw[doubleblue] (0,-.75)node[below] {\tiny $2$} -- (0,-.5);
  \draw[myblue] (-.015,-.5) .. controls (0,-.48) and (-.25,-.48) .. (-.3,.5);
  \draw[myblue] (.015,-.5) .. controls (0,.-.48) and (.25,-.48) ..  (.3,.5);
  \node[fill=myblue,circle,inner sep=1.5pt] at (-.18,-.25) {};
  \draw[doublered] (1,-.75)node[below] {\tiny $1$} -- (1,0);
  \draw[myred] (.985,0) .. controls (1,.12) and (.75,.12) .. (.7,.5);
  \draw[myred] (1.015,0) .. controls (1,.12) and (1.25,.12) ..  (1.3,.5);
} + t_{12}\ 
\tikz[very thick,scale=.85,baseline={([yshift=0ex]current bounding box.center)}]{
  \draw[doubleblue] (0,-.75)node[below] {\tiny $2$} -- (0,-.5);
  \draw[myblue] (-.015,-.5) .. controls (0,-.48) and (-.25,-.48) .. (-.3,.5);
  \draw[myblue] (.015,-.5) .. controls (0,.-.48) and (.25,-.48) ..  (.3,.5);
  \draw[doublered] (1,-.75)node[below] {\tiny $1$} -- (1,0);
  \draw[myred] (.985,0) .. controls (1,.12) and (.75,.12) .. (.7,.5);
  \draw[myred] (1.015,0) .. controls (1,.12) and (1.25,.12) ..  (1.3,.5);
  \node[fill=myred,circle,inner sep=1.5pt] at (0.775,.25) {};
}
\\
\delta' \Biggl(\
\tikz[very thick,scale=.85,baseline={([yshift=0ex]current bounding box.center)}]{
  \draw[doubleblue] (0,-.75)node[below] {\tiny $2$} -- (0,-.5);
  \draw[myblue] (-.015,-.5) .. controls (0,-.48) and (-.25,-.48) .. (-.3,.5);
  \draw[myblue] (.015,-.5) .. controls (0,.-.48) and (.25,-.48) ..  (.7,.5);
  \node[fill=myblue,circle,inner sep=1.5pt] at (-.18,-.25) {};
  \draw[doublered] (1,-.75)node[below] {\tiny $1$} -- (1,0);
  \draw[myred] (.985,0) .. controls (.8,.12) and (.5,.12) .. (.3,.5);
  \draw[myred] (1.015,0) .. controls (1,.12) and (1.25,.12) ..  (1.3,.5);
}
\Biggr) &=
t_{12}\ 
\tikz[very thick,scale=.85,baseline={([yshift=0ex]current bounding box.center)}]{
  \draw[doubleblue] (0,-.75)node[below] {\tiny $2$} -- (0,-.5);
  \draw[myblue] (-.015,-.5) .. controls (0,-.48) and (-.25,-.48) .. (-.3,.5);
  \draw[myblue] (.015,-.5) .. controls (0,.-.48) and (.25,-.48) ..  (.3,.5);
  \node[fill=myblue,circle,inner sep=1.5pt] at (-.18,-.25) {};
  \draw[doublered] (1,-.75)node[below] {\tiny $1$} -- (1,0);
  \draw[myred] (.985,0) .. controls (1,.12) and (.75,.12) .. (.7,.5);
  \draw[myred] (1.015,0) .. controls (1,.12) and (1.25,.12) ..  (1.3,.5);
  \node[fill=myred,circle,inner sep=1.5pt] at (0.775,.25) {};
}
\end{align*}
\endgroup
and
\begingroup\allowdisplaybreaks
\begin{align*}
\mu'\Biggl(\ 
\tikz[very thick,scale=.85,baseline={([yshift=0ex]current bounding box.center)}]{
  \draw[doubleblue] (0,-.75)node[below] {\tiny $2$} -- (0,-.5);
  \draw[myblue] (-.015,-.5) .. controls (0,-.48) and (-.25,-.48) .. (-.3,.5);
  \draw[myblue] (.015,-.5) .. controls (0,.-.48) and (.25,-.48) ..  (.3,.5);
  \draw[doublered] (1,-.75)node[below] {\tiny $1$} -- (1,0);
  \draw[myred] (.985,0) .. controls (1,.12) and (.75,.12) .. (.7,.5);
  \draw[myred] (1.015,0) .. controls (1,.12) and (1.25,.12) ..  (1.3,.5);
} \Biggr) =
\tikz[very thick,scale=.85,baseline={([yshift=0ex]current bounding box.center)}]{
  \draw[doubleblue] (0,-.75)node[below] {\tiny $2$} -- (0,-.5);
  \draw[myblue] (-.015,-.5) .. controls (0,-.48) and (-.25,-.48) .. (-.3,.5);
  \draw[myblue] (.015,-.5) .. controls (0,.-.48) and (.25,-.48) ..  (.7,.5);
  \draw[doublered] (1,-.75)node[below] {\tiny $1$} -- (1,0);
  \draw[myred] (.985,0) .. controls (.8,.12) and (.5,.12) .. (.3,.5);
  \draw[myred] (1.015,0) .. controls (1,.12) and (1.25,.12) ..  (1.3,.5);
}
&&\mu'\Biggl(\
\tikz[very thick,scale=.85,baseline={([yshift=0ex]current bounding box.center)}]{
  \draw[doubleblue] (0,-.75)node[below] {\tiny $2$} -- (0,-.5);
  \draw[myblue] (-.015,-.5) .. controls (0,-.48) and (-.25,-.48) .. (-.3,.5);
  \draw[myblue] (.015,-.5) .. controls (0,.-.48) and (.25,-.48) ..  (.3,.5);
  \node[fill=myblue,circle,inner sep=1.5pt] at (-.18,-.25) {};
  \draw[doublered] (1,-.75)node[below] {\tiny $1$} -- (1,0);
  \draw[myred] (.985,0) .. controls (1,.12) and (.75,.12) .. (.7,.5);
  \draw[myred] (1.015,0) .. controls (1,.12) and (1.25,.12) ..  (1.3,.5);
  \node[fill=myred,circle,inner sep=1.5pt] at (0.775,.25) {};
}
\Biggr) &= 0
\\
\mu'\biggl(\
\tikz[very thick,scale=.85,baseline={([yshift=0ex]current bounding box.center)}]{
  \draw[doubleblue] (0,-.75)node[below] {\tiny $2$} -- (0,-.5);
  \draw[myblue] (-.015,-.5) .. controls (0,-.48) and (-.25,-.48) .. (-.3,.5);
  \draw[myblue] (.015,-.5) .. controls (0,.-.48) and (.25,-.48) ..  (.3,.5);
  \node[fill=myblue,circle,inner sep=1.5pt] at (-.18,-.25) {};
  \draw[doublered] (1,-.75)node[below] {\tiny $1$} -- (1,0);
  \draw[myred] (.985,0) .. controls (1,.12) and (.75,.12) .. (.7,.5);
  \draw[myred] (1.015,0) .. controls (1,.12) and (1.25,.12) ..  (1.3,.5);
}
\Biggr) =
\tikz[very thick,scale=.85,baseline={([yshift=0ex]current bounding box.center)}]{
  \draw[doubleblue] (0,-.75)node[below] {\tiny $2$} -- (0,-.5);
  \draw[myblue] (-.015,-.5) .. controls (0,-.48) and (-.25,-.48) .. (-.3,.5);
  \draw[myblue] (.015,-.5) .. controls (0,.-.48) and (.25,-.48) ..  (.7,.5);
  \node[fill=myblue,circle,inner sep=1.5pt] at (-.18,-.25) {};
  \draw[doublered] (1,-.75)node[below] {\tiny $1$} -- (1,0);
  \draw[myred] (.985,0) .. controls (.8,.12) and (.5,.12) .. (.3,.5);
  \draw[myred] (1.015,0) .. controls (1,.12) and (1.25,.12) ..  (1.3,.5);
}
&&
\mu'\Biggl(\
\tikz[very thick,scale=.85,baseline={([yshift=0ex]current bounding box.center)}]{
  \draw[doubleblue] (0,-.75)node[below] {\tiny $2$} -- (0,-.5);
  \draw[myblue] (-.015,-.5) .. controls (0,-.48) and (-.25,-.48) .. (-.3,.5);
  \draw[myblue] (.015,-.5) .. controls (0,.-.48) and (.25,-.48) ..  (.3,.5);
  \draw[doublered] (1,-.75)node[below] {\tiny $1$} -- (1,0);
  \draw[myred] (.985,0) .. controls (1,.12) and (.75,.12) .. (.7,.5);
  \draw[myred] (1.015,0) .. controls (1,.12) and (1.25,.12) ..  (1.3,.5);
  \node[fill=myred,circle,inner sep=1.5pt] at (0.775,.25) {};
}
\Biggr) &=t_{21}t_{12}^{-1}\ 
\tikz[very thick,scale=.85,baseline={([yshift=0ex]current bounding box.center)}]{
  \draw[doubleblue] (0,-.75)node[below] {\tiny $2$} -- (0,-.5);
  \draw[myblue] (-.015,-.5) .. controls (0,-.48) and (-.25,-.48) .. (-.3,.5);
  \draw[myblue] (.015,-.5) .. controls (0,.-.48) and (.25,-.48) ..  (.7,.5);
  \node[fill=myblue,circle,inner sep=1.5pt] at (-.18,-.25) {};
  \draw[doublered] (1,-.75)node[below] {\tiny $1$} -- (1,0);
  \draw[myred] (.985,0) .. controls (.8,.12) and (.5,.12) .. (.3,.5);
  \draw[myred] (1.015,0) .. controls (1,.12) and (1.25,.12) ..  (1.3,.5);
}
\end{align*}
\endgroup
In this case we also have $\mu'\delta'=0$. 

Contrary to the previous case, we have $p(\mu')=1$ and $p(\delta')=0$.
The maps $\mu'$ and $\delta'$ can also be made to agree with~\cite{ORS}, but the parity is reversed (as with $\imath$
and $\varepsilon$ above).

%
%
\subsubsection{A sample computation}

We now compute the homology of the left-handed trefoil $T$ in its lowest and highest homological degrees.  
Consider the following presentation of $T$,

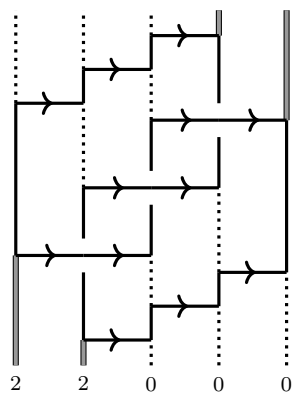
\[
\tikz[very thick,scale=.225,baseline={([yshift=0ex]current bounding box.center)}]{
\draw[doubleblack] (-2,-9)node[below] {\tiny $2$} to (-2,-2.5);\draw (-2,-2.5) to (-2,6.5);
     \draw[dotted] (-2,6.5) to (-2,12); 
\draw[doubleblack] ( 2,-9)node[below] {\tiny $2$} to (2,-7.5);\draw (2,-7.5) to (2,-3.5);
     \draw (2,-1.5) to (2,1.5);\draw [dotted] (2,1.5) to (2,6.5);\draw (2,6.5) to (2,8.5);\draw[dotted] (2,8.5) to (2,12);   
\draw[dotted] (6,-9)node[below] {\tiny $0$} to (6,-7.5);\draw (6,-7.5) to (6,-5.5);\draw[dotted] (6,-5.5) to (6,-2.5);
   \draw (6,-2.5) to (6,.5);\draw (6,2.5) to (6,5.5);\draw[dotted] (6,5.5) to (6,8.5);\draw (6,8.5) to (6,10.5);
   \draw[dotted] (6,10.5) to (6,12);
   \draw [dotted] (10,-9)node[below] {\tiny $0$} to (10,-5.5);\draw (10,-5.5) to (10,-3.5);
   \draw[dotted] (10,-3.5) to (10,1.5);\draw (10,1.5) to (10,4.5);\draw (10,6.5) to (10,10.5);
   \draw[doubleblack] (10,10.5) to (10,12);   
\draw [dotted] (14,-9)node[below] {\tiny $0$} to (14,-3.5);\draw (14,-3.5) to (14,5.5);
   \draw[doubleblack] (14,5.5) to (14,12);
\draw [directed=.55] (2,-7.5) to (6,-7.5);
\draw [directed=.55] (6,-5.5) to (10,-5.5);
\draw [directed=.55] (10,-3.5) to (14,-3.5);
\draw [directed=.55] (-2,6.5) to (2,6.5);
\draw [directed=.55] (2,8.5) to (6,8.5);
\draw [directed=.55] (6,10.5) to (10,10.5);
\draw [directed=.6] (-2,-2.5) to (2,-2.5);\draw [directed=.6] (2,-2.5) to (6,-2.5);
\draw [directed=.6] (2,1.5) to (6,1.5);\draw [directed=.6] (6,1.5) to (10,1.5);
\draw [directed=.6] (6,5.5) to (10,5.5);\draw [directed=.6] (10,5.5) to (14,5.5);
}
\]

The computation of $H_0(T)$ is fairly simple: up to an overall degree shift it is the homology in degree 1 of the complex 
\begin{equation}\label{eq:cplx-Tzero}
 \xy
 (0,-1.5)*{
 \begin{tikzpicture}
  \node at (7, 0) {$q^3F_tF_{342312}F_b\mu$};
  \node at (- 2, 2) {$q^2F_tF_{432312}F_b\mu$};
   \node at (-2,1) {$\bigoplus$};  
  \node at (-2, 0) {$q^2F_tF_{343212}F_b\mu$};
   \node at (-2,-1) {$\bigoplus$};  
  \node at (-2, -2) {$q^2F_tF_{342321}F_b\mu$};
  \draw[semithick,rdirected=0.01] (5.5, .75) -- (-.5, 1.9);
  \draw[semithick,rdirected=0.01] (5.5, 0) -- (-.5,  0);
  \draw[semithick,rdirected=0.01] (5.5,-.75) -- (-.5,-1.9);
  \node at (3.75,1.75) {
\tikz[very thick,scale=0.6,baseline={([yshift=.9ex]current bounding box.center)}]{
     \draw[black] (0,-.5)node[below] {\tiny $4$} .. controls (0,0) and (1,0) .. (1,.5); 
     \draw[mygreen] (1,-.5)node[below] {\textcolor{black}{\tiny $3$}} .. controls (1,0) and (0,0) .. (0,.5); 
     \draw[myblue] (1.5,-.5)node[below] {\textcolor{black}{\tiny $2$}} -- (1.5,.5);
     \draw[mygreen] (2,-.5)node[below] {\textcolor{black}{\tiny $3$}} -- (2,.5);
     \draw[myred] (2.5,-.5)node[below] {\textcolor{black}{\tiny $1$}} -- (2.5,.5);
     \draw[myblue] (3,-.5)node[below] {\textcolor{black}{\tiny $2$}} -- (3,.5);}
  };
\node at (1.6,.5) {
\tikz[very thick,scale=0.6,baseline={([yshift=.9ex]current bounding box.center)}]{
    \draw[mygreen] (-1,-.5)node[below] {\textcolor{black}{\tiny $3$}} -- (-1,.5);
     \draw[black] (-.5,-.5)node[below] {\tiny $4$} -- (-.5,.5);
    \draw[mygreen] (0,-.5)node[below] {\textcolor{black}{\tiny $3$}} .. controls (0,0) and (1,0) .. (1,.5); 
     \draw[myblue] (1,-.5)node[below] {\textcolor{black}{\tiny $2$}} .. controls (1,0) and (0,0) .. (0,.5); 
      \draw[myred] (1.5,-.5)node[below] {\textcolor{black}{\tiny $1$}} -- (1.5,.5);
     \draw[myblue] (2,-.5)node[below] {\textcolor{black}{\tiny $2$}} -- (2,.5);}
};  
\node at (3.8,-2) {
\tikz[very thick,scale=0.6,baseline={([yshift=.9ex]current bounding box.center)}]{
     \draw[mygreen] (-2,-.5)node[below] {\textcolor{black}{\tiny $3$}} -- (-2,.5);
     \draw[black] (-1.5,-.5)node[below] {\tiny $4$} -- (-1.5,.5);
     \draw[myblue] (-1,-.5)node[below] {\textcolor{black}{\tiny $2$}} -- (-1,.5);
     \draw[mygreen] (-.5,-.5)node[below] {\textcolor{black}{\tiny $3$}} -- (-.5,.5);
     \draw[myblue] (0,-.5)node[below] {\textcolor{black}{\tiny $2$}} .. controls (0,0) and (1,0) .. (1,.5); 
     \draw[myred] (1,-.5)node[below] {\textcolor{black}{\tiny $1$}} .. controls (1,0) and (0,0) .. (0,.5);} 
};
 \end{tikzpicture}
 }; \endxy 
\end{equation}

The three terms in homological degree zero are isomorphic to $F_{43^{(2)}2^{(2)}1}$. 
Composing the isomorphisms from $F_{43^{(2)}2^{(2)}1}$ to $F_{432312}$, $F_{343212}$ and to $F_{342321}$ with the corresponding
maps above gives three maps that differ by a sign. 

By inspection, one sees that up to a sign, these theree maps are equal to the map $\delta$ from the case~~\eqref{eq:frob} in the previous
subsection.
The cokernel map in~\eqref{eq:cplx-Tzero} is therefore two-dimensional. 
Adding the degree shifts one obtains
\[
H_0(T) = q^{-1}\Bbbk \oplus q^{-3}\Bbbk . 
\]

We now compute $H_{-3}(H)$.
Up to an overall degree shift it is computed as the homology in
degree zero of the complex 
\[
  \begin{tikzpicture}
  \node at (-7, 0) {$F_{321}F_{433221}F_{432}\mu$};
  \node at ( 2, 2) {$qF_{321}F_{343221}F_{432}\mu$};
   \node at (2,1) {$\bigoplus$};  
  \node at ( 2, 0) {$qF_{321}F_{432321}F_{432}\mu$};
   \node at (2,-1) {$\bigoplus$};  
  \node at ( 2, -2) {$qF_{321}F_{433212}F_{432}\mu$};
  \draw[semithick,directed=1] (-5, .75) -- (0, 1.9);
  \draw[semithick,directed=1] (-5, 0) -- (0,  0);
  \draw[semithick,directed=1] (-5,-.75) -- (0,-1.9);
  \node at (-3.75,1.75) {
    $\dotsm$\! \tikz[very thick,scale=0.6,baseline={([yshift=.9ex]current bounding box.center)}]{
     \draw[black] (0,-.5)node[below] {\tiny $4$} .. controls (0,0) and (1,0) .. (1,.5); 
     \draw[mygreen] (1,-.5)node[below] {\textcolor{black}{\tiny $3$}} .. controls (1,0) and (0,0) .. (0,.5); 
     \draw[mygreen] (1.5,-.5)node[below] {\textcolor{black}{\tiny $3$}} -- (1.5,.5);
     \draw[myblue] (2,-.5)node[below] {\textcolor{black}{\tiny $2$}} -- (2,.5);
     \draw[myblue] (2.5,-.5)node[below] {\textcolor{black}{\tiny $2$}} -- (2.5,.5);
     \draw[myred] (3,-.5)node[below] {\textcolor{black}{\tiny $1$}} -- (3,.5);}
\!$\dotsm$
  };
\node at (-1.6,.5) {
$\dotsm$\! \tikz[very thick,scale=0.6,baseline={([yshift=.9ex]current bounding box.center)}]{
    \draw[black] (-1,-.5)node[below] {\tiny $4$} -- (-1,.5);
     \draw[mygreen] (-.5,-.5)node[below] {\textcolor{black}{\tiny $3$}} -- (-.5,.5);
    \draw[mygreen] (0,-.5)node[below] {\textcolor{black}{\tiny $3$}} .. controls (0,0) and (1,0) .. (1,.5); 
     \draw[myblue] (1,-.5)node[below] {\textcolor{black}{\tiny $2$}} .. controls (1,0) and (0,0) .. (0,.5); 
      \draw[myblue] (1.5,-.5)node[below] {\textcolor{black}{\tiny $2$}} -- (1.5,.5);
     \draw[myred] (2,-.5)node[below] {\textcolor{black}{\tiny $1$}} -- (2,.5);}
\!$\dotsm$
};  
\node at (-3.8,-2) {
$\dotsm$\! \tikz[very thick,scale=0.6,baseline={([yshift=.9ex]current bounding box.center)}]{
     \draw[black] (-2,-.5)node[below] {\tiny $4$} -- (-2,.5);
     \draw[mygreen] (-1.5,-.5)node[below] {\textcolor{black}{\tiny $3$}} -- (-1.5,.5);
     \draw[mygreen] (-1,-.5)node[below] {\textcolor{black}{\tiny $3$}} -- (-1,.5);
     \draw[myblue] (-.5,-.5)node[below] {\textcolor{black}{\tiny $2$}} -- (-.5,.5);
     \draw[myblue] (0,-.5)node[below] {\textcolor{black}{\tiny $2$}} .. controls (0,0) and (1,0) .. (1,.5); 
     \draw[myred] (1,-.5)node[below] {\textcolor{black}{\tiny $1$}} .. controls (1,0) and (0,0) .. (0,.5);} 
\!$\dotsm$};
\end{tikzpicture} 
\]
Here $\mu=(2,2,0,0,0)$ and
the factors $F_{321}$ and $F_{432}$ are the upper and lower closures of the diagram.
We write $F_t$ for $F_{321}$ and $F_b$ for $F_{432}$ and sometimes we write
$F_tF_{433221}F_b\mu$ instead of  $F_{321}F_{433221}F_{432}\mu$, etc...,
and we only depict the pertinent part of the morphisms. 

\smallskip

In the following we will use the identities 
\begin{equation}\label{eq:dotjumpsover}
\dotsm 
  \tikz[very thick,scale=0.8,baseline={([yshift=.6ex]current bounding box.center)}]{
     \draw[black] (1,-.5)node[below] {\tiny $4$} -- (1,.5)node[midway,fill=black,circle,inner sep=1.3pt]{}; 
     \draw[doublegreen] (1.5,-.5)node[below] {\tiny $3$} -- (1.5,.5);
     \draw[doubleblue] (2,-.5)node[below] {\tiny $2$} -- (2,.5);
     \draw[myred] (2.5,-.5)node[below] {\textcolor{black}{\tiny $1$}} -- (2.5,.5);  
     \draw[black] (3,-.5)node[below] {\textcolor{black}{\tiny $4$}} -- (3,.5);
     \draw[mygreen] (3.5,-.5)node[below] {\textcolor{black}{\tiny $3$}} -- (3.5,.5);
     \draw[myblue] (4,-.5)node[below] {\textcolor{black}{\tiny $2$}} -- (4,.5);
\node at (4.5,0) {\small $\mu$};}  
  =
\dotsm 
  \tikz[very thick,scale=0.8,baseline={([yshift=.6ex]current bounding box.center)}]{
     \draw[black] (1,-.5)node[below] {\tiny $4$} -- (1,.5); 
     \draw[doublegreen] (1.5,-.5)node[below] {\tiny $3$} -- (1.5,.5);
     \draw[doubleblue] (2,-.5)node[below] {\tiny $2$} -- (2,.5);
     \draw[myred] (2.5,-.5)node[below] {\textcolor{black}{\tiny $1$}} -- (2.5,.5);  
     \draw[black] (3,-.5)node[below] {\textcolor{black}{\tiny $4$}} -- (3,.5)node[midway,fill=black,circle,inner sep=1.3pt]{};
     \draw[mygreen] (3.5,-.5)node[below] {\textcolor{black}{\tiny $3$}} -- (3.5,.5);
     \draw[myblue] (4,-.5)node[below] {\textcolor{black}{\tiny $2$}} -- (4,.5);
     \node at (4.5,0) {\small $\mu$};}
  = -\tfrac{t_{12}}{t_{21}}\tfrac{t_{23}}{t_{32}}\tfrac{t_{34}}{t_{43}}
 \dotsm 
  \tikz[very thick,scale=0.8,baseline={([yshift=.6ex]current bounding box.center)}]{
     \draw[black] (1,-.5)node[below] {\tiny $4$} -- (1,.5); 
     \draw[doublegreen] (1.5,-.5)node[below] {\tiny $3$} -- (1.5,.5);
     \draw[doubleblue] (2,-.5)node[below] {\tiny $2$} -- (2,.5);
     \draw[myred] (2.5,-.5)node[below] {\textcolor{black}{\tiny $1$}} -- (2.5,.5)node[midway,fill=myred,circle,inner sep=1.3pt]{};  
     \draw[black] (3,-.5)node[below] {\textcolor{black}{\tiny $4$}} -- (3,.5);
     \draw[mygreen] (3.5,-.5)node[below] {\textcolor{black}{\tiny $3$}} -- (3.5,.5);
     \draw[myblue] (4,-.5)node[below] {\textcolor{black}{\tiny $2$}} -- (4,.5);
     \node at (4.5,0) {\small $\mu$};} 
\end{equation}
The first equality follows from~\fullref{lem:dotjumpdlb} after using~\eqref{eq:klrR2} on the second strand labeled 4 to
pull it to the left. 
The second equality can be checked by a applying~\eqref{eq:klrR2} three times.

\smallskip

Coming back to $H_{-3}(T)$ we apply the isomorphisms
\begingroup\allowdisplaybreaks
\begin{align*}
F_{433221} &\simeq qF_{4332^{(2)}1}\oplus q^{-1}F_{4332^{(2)}1} , 
\\
F_{343221} &\simeq qF_{3432^{(2)}1}\oplus q^{-1}F_{3432^{(2)}1} ,
\\ 
F_{433212} &\simeq F_{4332^{(2)}1} ,
\end{align*}
\endgroup
to obtain the isomorphic complex 
\[
\begin{pmatrix}
  qF_tF_{4332^{(2)}1}F_b\mu
  \\[1ex]
  q^{-1}F_tF_{4332^{(2)}1}F_b\mu
\end{pmatrix}
\xra{
  \begin{pmatrix}
\tikz[very thick,scale=0.6,baseline={([yshift=0ex]current bounding box.center)}]{
     \draw[black] (0,-.5)node[below] {\tiny $4$} .. controls (0,0) and (1,0) .. (1,.5); 
     \draw[mygreen] (1,-.5)node[below] {\textcolor{black}{\tiny $3$}} .. controls (1,0) and (0,0) .. (0,.5); 
     \draw[mygreen] (1.5,-.5)node[below] {\textcolor{black}{\tiny $3$}} -- (1.5,.5);
     \draw[doubleblue] (2,-.5)node[below] {\tiny $2$} -- (2,.5);
     \draw[myred] (2.5,-.5)node[below] {\textcolor{black}{\tiny $1$}} -- (2.5,.5);}  
& 0
\\[1.5ex]
0 &
-\tikz[very thick,scale=0.6,baseline={([yshift=0ex]current bounding box.center)}]{
     \draw[black] (0,-.5)node[below] {\tiny $4$} .. controls (0,0) and (1,0) .. (1,.5); 
     \draw[mygreen] (1,-.5)node[below] {\textcolor{black}{\tiny $3$}} .. controls (1,0) and (0,0) .. (0,.5); 
     \draw[mygreen] (1.5,-.5)node[below] {\textcolor{black}{\tiny $3$}} -- (1.5,.5);
     \draw[doubleblue] (2,-.5)node[below] {\tiny $2$} -- (2,.5);
     \draw[myred] (2.5,-.5)node[below] {\textcolor{black}{\tiny $1$}} -- (2.5,.5);}  
\\[1.5ex]
\tikz[very thick,scale=0.6,baseline={([yshift=.8ex]current bounding box.center)}]{
  \draw[black] (-1.75,.5) -- (-1.75,-.5) node[below] {\tiny $4$};  
  \draw[mygreen] (-1.25,.5) -- (-1.25,-.5) node[below] {\textcolor{black}{\tiny $3$}};  
  \draw[mygreen] (-.225,.5)  .. controls (-.35,.1) and (-.8,.1) ..  (-.75,-.5)node[below] {\textcolor{black}{\tiny $3$}};
  \draw[doubleblue] (0,-.5)node[below] {\tiny $2$} -- (0,0);
  \draw[myblue] (-.015,0) .. controls (0,.12) and (-.25,-.06) .. (-.75,.5);
  \draw[myblue] (.015,0) .. controls (0,.12) and (.25,.12) ..  (.3,.5);
  \draw[myred] (.75,.5) -- (.75,-.5) node[below] {\textcolor{black}{\tiny $1$}};}  
&
\tikz[very thick,scale=0.6,baseline={([yshift=.8ex]current bounding box.center)}]{
  \draw[black] (-1.75,.5) -- (-1.75,-.5) node[below] {\tiny $4$};  
  \draw[mygreen] (-1.25,.5) -- (-1.25,-.5) node[below] {\textcolor{black}{\tiny $3$}};  
  \draw[mygreen] (-.225,.5)  .. controls (-.35,.3) and (-.8,.2) ..  (-.75,-.5)node[below] {\textcolor{black}{\tiny $3$}};
  \draw[doubleblue] (0,-.5)node[below] {\tiny $2$} -- (0,0);
  \draw[myblue] (-.015,0) .. controls (0,.12) and (-.25,-.06) .. (-.75,.5);
  \draw[myblue] (.015,0) .. controls (0,.12) and (.25,.12) ..  (.3,.5);
  \draw[myred] (.75,.5) -- (.75,-.5) node[below] {\textcolor{black}{\tiny $1$}};
  \node[fill=myblue,circle,inner sep=1.3pt] at (-.22,.1) {};}
\\[1.5ex]
t_{21}
\tikz[very thick,scale=0.6,baseline={([yshift=0.6ex]current bounding box.center)}]{
     \draw[black] (.5,-.5)node[below] {\tiny $4$} -- (.5,.5); 
     \draw[mygreen] (1,-.5)node[below] {\textcolor{black}{\tiny $3$}} -- (1,.5); 
     \draw[mygreen] (1.5,-.5)node[below] {\textcolor{black}{\tiny $3$}} -- (1.5,.5);
     \draw[doubleblue] (2,-.5)node[below] {\tiny $2$} -- (2,.5);
     \draw[myred] (2.5,-.5)node[below] {\textcolor{black}{\tiny $1$}} -- (2.5,.5);}  
&
t_{12}
\tikz[very thick,scale=0.6,baseline={([yshift=0.6ex]current bounding box.center)}]{
     \draw[black] (.5,-.5)node[below] {\tiny $4$} -- (.5,.5); 
     \draw[mygreen] (1,-.5)node[below] {\textcolor{black}{\tiny $3$}} -- (1,.5); 
     \draw[mygreen] (1.5,-.5)node[below] {\textcolor{black}{\tiny $3$}} -- (1.5,.5);
     \draw[doubleblue] (2,-.5)node[below] {\tiny $2$} -- (2,.5);
     \draw[myred] (2.5,-.5)node[below] {\textcolor{black}{\tiny $1$}} -- (2.5,.5)node[midway,fill=myred,circle,inner sep=1.3pt]{};}  
  \end{pmatrix}
}
\begin{pmatrix}
q^2F_tF_{3432^{(2)}1}F_b\mu
\\[1ex]
F_tF_{3432^{(2)}1}F_b\mu
\\[1ex]
qF_tF_{432321}F_b\mu
\\[1ex]
qF_tF_{4332^{(2)}1}F_b\mu
\end{pmatrix} . 
\]

By Gaussian elimination of the acyclic complex 
\[
qF_tF_{4332^{(2)}1}F_b\mu 
\xra{t_{21}
\tikz[very thick,scale=0.65,baseline={([yshift=0.6ex]current bounding box.center)}]{
     \draw[black] (.5,-.5)node[below] {\tiny $4$} -- (.5,.5); 
     \draw[mygreen] (1,-.5)node[below] {\textcolor{black}{\tiny $3$}} -- (1,.5); 
     \draw[mygreen] (1.5,-.5)node[below] {\textcolor{black}{\tiny $3$}} -- (1.5,.5);
     \draw[doubleblue] (2,-.5)node[below] {\tiny $2$} -- (2,.5);
     \draw[myred] (2.5,-.5)node[below] {\textcolor{black}{\tiny $1$}} -- (2.5,.5);}  }
qF_tF_{4332^{(2)}1}F_b\mu . 
\]
we obtain the homotopy equivalent complex
\[
  q^{-1}F_tF_{4332^{(2)}1}F_b\mu
\xra{
  \begin{pmatrix}
-\tfrac{t_{12}}{t_{21}}
    \tikz[very thick,scale=0.6,baseline={([yshift=0.6ex]current bounding box.center)}]{
     \draw[black] (0,-.5)node[below] {\tiny $4$} .. controls (0,0) and (1,0) .. (1,.5); 
     \draw[mygreen] (1,-.5)node[below] {\textcolor{black}{\tiny $3$}} .. controls (1,0) and (0,0) .. (0,.5); 
     \draw[mygreen] (1.5,-.5)node[below] {\textcolor{black}{\tiny $3$}} -- (1.5,.5);
     \draw[doubleblue] (2,-.5)node[below] {\tiny $2$} -- (2,.5);
     \draw[myred] (2.5,-.5)node[below] {\textcolor{black}{\tiny $1$}} -- (2.5,.5)node[near start,fill=myred,circle,inner sep=1.3pt]{};}  
\\[1.5ex]
-\tikz[very thick,scale=0.6,baseline={([yshift=0ex]current bounding box.center)}]{
     \draw[black] (0,-.5)node[below] {\tiny $4$} .. controls (0,0) and (1,0) .. (1,.5); 
     \draw[mygreen] (1,-.5)node[below] {\textcolor{black}{\tiny $3$}} .. controls (1,0) and (0,0) .. (0,.5); 
     \draw[mygreen] (1.5,-.5)node[below] {\textcolor{black}{\tiny $3$}} -- (1.5,.5);
     \draw[doubleblue] (2,-.5)node[below] {\tiny $2$} -- (2,.5);
     \draw[myred] (2.5,-.5)node[below] {\textcolor{black}{\tiny $1$}} -- (2.5,.5);}  
\\[1.5ex]
\tikz[very thick,scale=0.6,baseline={([yshift=.8ex]current bounding box.center)}]{
  \draw[black] (-1.75,.5) -- (-1.75,-.5) node[below] {\tiny $4$};  
  \draw[mygreen] (-1.25,.5) -- (-1.25,-.5) node[below] {\textcolor{black}{\tiny $3$}};  
  \draw[mygreen] (-.225,.5)  .. controls (-.35,.3) and (-.8,.2) ..  (-.75,-.5)node[below] {\textcolor{black}{\tiny $3$}};
  \draw[doubleblue] (0,-.5)node[below] {\tiny $2$} -- (0,0);
  \draw[myblue] (-.015,0) .. controls (0,.12) and (-.25,-.06) .. (-.75,.5);
  \draw[myblue] (.015,0) .. controls (0,.12) and (.25,.12) ..  (.3,.5);
  \draw[myred] (.75,.5) -- (.75,-.5) node[below] {\textcolor{black}{\tiny $1$}};
  \node[fill=myblue,circle,inner sep=1.3pt] at (-.22,.1) {};}
-\tfrac{t_{12}}{t_{21}}
\tikz[very thick,scale=0.6,baseline={([yshift=.8ex]current bounding box.center)}]{
  \draw[black] (-1.75,.5) -- (-1.75,-.5) node[below] {\tiny $4$};  
  \draw[mygreen] (-1.25,.5) -- (-1.25,-.5) node[below] {\textcolor{black}{\tiny $3$}};  
  \draw[mygreen] (-.225,.5)  .. controls (-.35,.1) and (-.8,.1) ..  (-.75,-.5)node[below] {\textcolor{black}{\tiny $3$}};
  \draw[doubleblue] (0,-.5)node[below] {\tiny $2$} -- (0,0);
  \draw[myblue] (-.015,0) .. controls (0,.12) and (-.25,-.06) .. (-.75,.5);
  \draw[myblue] (.015,0) .. controls (0,.12) and (.25,.12) ..  (.3,.5);
  \draw[myred] (.75,.5) -- (.75,-.5) node[below] {\textcolor{black}{\tiny $1$}}node[near end,fill=myred,circle,inner sep=1.3pt]{};}  
  \end{pmatrix}
}
\begin{pmatrix}
q^2F_tF_{3432^{(2)}1}F_b\mu
\\[1ex]
F_tF_{3432^{(2)}1}F_b\mu
\\[1ex]
qF_tF_{432321}F_b\mu
\end{pmatrix} . 
\]
Applying the isomorphisms
\begin{equation}\label{eq:isoFtrefoil}
  F_{4332^{(2)}1}\simeq qF_{43^{(2)}2^{(2)}1}\oplus q^{-1}F_{43^{(2)}2^{(2)}1}  
\end{equation}
  and 
$F_{3432^{(2)}1}\simeq F_{43^{(2)}2^{(2)}1}$ gives the isomorphic complex 
\[
\begin{pmatrix}
 F_tF_{43^{(2)}2^{(2)}1}F_b\mu
  \\[1ex]
 q^{-2}F_tF_{43^{(2)}2^{(2)}1}F_b\mu
 \end{pmatrix}
\xra{
  \begin{pmatrix}
\tfrac{t_{12}t_{34}}{t_{21}}
\tikz[very thick,scale=0.6,baseline={([yshift=0.6ex]current bounding box.center)}]{
     \draw[black] (1,-.5)node[below] {\tiny $4$} -- (1,.5); 
     \draw[doublegreen] (1.5,-.5)node[below] {\tiny $3$} -- (1.5,.5);
     \draw[doubleblue] (2,-.5)node[below] {\tiny $2$} -- (2,.5);
     \draw[myred] (2.5,-.5)node[below] {\textcolor{black}{\tiny $1$}} -- (2.5,.5)node[midway,fill=myred,circle,inner sep=1.3pt]{};}  
& 0
\\[1.5ex]
-t_{34} 
\tikz[very thick,scale=0.6,baseline={([yshift=0.6ex]current bounding box.center)}]{
     \draw[black] (1,-.5)node[below] {\tiny $4$} -- (1,.5); 
     \draw[doublegreen] (1.5,-.5)node[below] {\tiny $3$} -- (1.5,.5);
     \draw[doubleblue] (2,-.5)node[below] {\tiny $2$} -- (2,.5);
     \draw[myred] (2.5,-.5)node[below] {\textcolor{black}{\tiny $1$}} -- (2.5,.5);}  
&
-t_{43}
\tikz[very thick,scale=0.6,baseline={([yshift=0.6ex]current bounding box.center)}]{
     \draw[black] (1,-.5)node[below] {\tiny $4$} -- (1,.5)node[midway,fill=black,circle,inner sep=1.3pt]{}; 
     \draw[doublegreen] (1.5,-.5)node[below] {\tiny $3$} -- (1.5,.5);
     \draw[doubleblue] (2,-.5)node[below] {\tiny $2$} -- (2,.5);
     \draw[myred] (2.5,-.5)node[below] {\textcolor{black}{\tiny $1$}} -- (2.5,.5);}  
\\[1.5ex]
f & g
  \end{pmatrix}
}
\begin{pmatrix}
 q^2F_tF_{43^{(2)}2^{(2)}1}F_b\mu
\\[1ex]
 F_tF_{43^{(2)}2^{(2)}1}F_b\mu
\\[1ex]
qF_tF_{432321}F_b\mu
\end{pmatrix} , 
\]
or
\[
\begin{pmatrix}
 F_tF_{43^{(2)}2^{(2)}1}F_b\mu
  \\[1ex]
 q^{-2}F_tF_{43^{(2)}2^{(2)}1}F_b\mu
 \end{pmatrix}
\xra{
  \begin{pmatrix}
\tfrac{t_{12}t_{34}}{t_{21}}
\tikz[very thick,scale=0.6,baseline={([yshift=0.6ex]current bounding box.center)}]{
     \draw[black] (1,-.5)node[below] {\tiny $4$} -- (1,.5); 
     \draw[doublegreen] (1.5,-.5)node[below] {\tiny $3$} -- (1.5,.5);
     \draw[doubleblue] (2,-.5)node[below] {\tiny $2$} -- (2,.5);
     \draw[myred] (2.5,-.5)node[below] {\textcolor{black}{\tiny $1$}} -- (2.5,.5)node[midway,fill=myred,circle,inner sep=1.3pt]{};}  
& 0
\\[1.5ex]
-t_{34} 
\tikz[very thick,scale=0.6,baseline={([yshift=0.6ex]current bounding box.center)}]{
     \draw[black] (1,-.5)node[below] {\tiny $4$} -- (1,.5); 
     \draw[doublegreen] (1.5,-.5)node[below] {\tiny $3$} -- (1.5,.5);
     \draw[doubleblue] (2,-.5)node[below] {\tiny $2$} -- (2,.5);
     \draw[myred] (2.5,-.5)node[below] {\textcolor{black}{\tiny $1$}} -- (2.5,.5);}  
&
t_{34}\tfrac{t_{12}}{t_{21}}\tfrac{t_{23}}{t_{32}}
\tikz[very thick,scale=0.6,baseline={([yshift=0.6ex]current bounding box.center)}]{
     \draw[black] (1,-.5)node[below] {\tiny $4$} -- (1,.5); 
     \draw[doublegreen] (1.5,-.5)node[below] {\tiny $3$} -- (1.5,.5);
     \draw[doubleblue] (2,-.5)node[below] {\tiny $2$} -- (2,.5);
     \draw[myred] (2.5,-.5)node[below] {\textcolor{black}{\tiny $1$}} -- (2.5,.5)node[midway,fill=myred,circle,inner sep=1.3pt]{};}  
\\[1.5ex]
f & g
  \end{pmatrix}
}
\begin{pmatrix}
 q^2F_tF_{43^{(2)}2^{(2)}1}F_b\mu
\\[1ex]
 F_tF_{43^{(2)}2^{(2)}1}F_b\mu
\\[1ex]
qF_tF_{432321}F_b\mu
\end{pmatrix} , 
\]
where $f$ (resp. $g$) is the composite of the map from $F_{43^{(2)}2^{(2)}1}$
(resp. $q^{-2}F_{43^{(2)}2^{(2)}1}$) to $q^{-1}F_{4332^{(2)}1}$ in~\eqref{eq:isoFtrefoil} and 
\[
\tikz[very thick,scale=0.8,baseline={([yshift=.7ex]current bounding box.center)}]{
  \draw[black] (-1.75,.5) -- (-1.75,-.5) node[below] {\tiny $4$};  
  \draw[mygreen] (-1.25,.5) -- (-1.25,-.5) node[below] {\textcolor{black}{\tiny $3$}};  
  \draw[mygreen] (-.225,.5)  .. controls (-.35,.3) and (-.8,.2) ..  (-.75,-.5)node[below] {\textcolor{black}{\tiny $3$}};
  \draw[doubleblue] (0,-.5)node[below] {\tiny $2$} -- (0,0);
  \draw[myblue] (-.015,0) .. controls (0,.12) and (-.25,-.06) .. (-.75,.5);
  \draw[myblue] (.015,0) .. controls (0,.12) and (.25,.12) ..  (.3,.5);
  \draw[myred] (.75,.5) -- (.75,-.5) node[below] {\textcolor{black}{\tiny $1$}};
  \node[fill=myblue,circle,inner sep=1.3pt] at (-.22,.1) {};}
-\tfrac{t_{12}}{t_{21}}
\tikz[very thick,scale=0.8,baseline={([yshift=.7ex]current bounding box.center)}]{
  \draw[black] (-1.75,.5) -- (-1.75,-.5) node[below] {\tiny $4$};  
  \draw[mygreen] (-1.25,.5) -- (-1.25,-.5) node[below] {\textcolor{black}{\tiny $3$}};  
  \draw[mygreen] (-.225,.5)  .. controls (-.35,.1) and (-.8,.1) ..  (-.75,-.5)node[below] {\textcolor{black}{\tiny $3$}};
  \draw[doubleblue] (0,-.5)node[below] {\tiny $2$} -- (0,0);
  \draw[myblue] (-.015,0) .. controls (0,.12) and (-.25,-.06) .. (-.75,.5);
  \draw[myblue] (.015,0) .. controls (0,.12) and (.25,.12) ..  (.3,.5);
  \draw[myred] (.75,.5) -- (.75,-.5) node[below] {\textcolor{black}{\tiny $1$}}node[near end,fill=myred,circle,inner sep=1.3pt]{};}  
\]

Gaussian elimination of the acyclic complex
\[
F_tF_{43^{(2)}2^{(2)}1}F_b\mu 
\xra{
-t_{34} 
\tikz[very thick,scale=0.6,baseline={([yshift=.8ex]current bounding box.center)}]{
     \draw[black] (1,-.5)node[below] {\tiny $4$} -- (1,.5); 
     \draw[doublegreen] (1.5,-.5)node[below] {\tiny $3$} -- (1.5,.5);
     \draw[doubleblue] (2,-.5)node[below] {\tiny $2$} -- (2,.5);
     \draw[myred] (2.5,-.5)node[below] {\textcolor{black}{\tiny $1$}} -- (2.5,.5);} 
}
F_tF_{43^{(2)}2^{(2)}1}F_b\mu , 
\]
yields the homotopy equivalent complex 
\[
 q^{-2}F_tF_{43^{(2)}2^{(2)}1}F_b\mu
\xra{
  \begin{pmatrix}
    0
    \\[1.5ex]
h
    \end{pmatrix}
}
\begin{pmatrix}
 q^2F_tF_{43^{(2)}2^{(2)}1}F_b\mu
\\[1ex]
qF_tF_{432321}F_b\mu 
\end{pmatrix} ,  
\]
where 
\[
h =
  \tikz[very thick,scale=.8,baseline={([yshift=0.6ex]current bounding box.center)}]{
  \draw[doublegreen] (0,-.75)node[below] {\tiny $3$} -- (0,-.5);
  \draw[mygreen] (-.015,-.5) .. controls (0,-.48) and (-.25,-.48) .. (-.3,.5);
  \draw[mygreen] (.015,-.5) .. controls (0,.-.48) and (.25,-.48) ..  (.7,.5);
  \node[fill=mygreen,circle,inner sep=1.5pt] at (-.15,-.35) {};
  \draw[doubleblue] (1,-.75)node[below] {\tiny $2$} -- (1,0);
  \draw[myblue] (.985,0) .. controls (.8,.12) and (.5,.12) .. (.3,.5);
  \draw[myblue] (1.015,0) .. controls (1,.12) and (1.25,.12) ..  (1.3,.5);
  \node[fill=myblue,circle,inner sep=1.5pt] at (.8,.09) {};  
  \draw[black] (-.75,-.75)node[below] {\tiny $4$} -- (-.75,.5);
  \draw[myred] (1.75,-.75)node[below] {\textcolor{black}{\tiny $1$}} -- (1.75,.5);}
  -\tfrac{t_{12}}{t_{21}}
\tikz[very thick,scale=.8,baseline={([yshift=0.6ex]current bounding box.center)}]{
  \draw[doublegreen] (0,-.75)node[below] {\tiny $3$} -- (0,-.5);
  \draw[mygreen] (-.015,-.5) .. controls (0,-.48) and (-.25,-.48) .. (-.3,.5);
  \draw[mygreen] (.015,-.5) .. controls (0,.-.48) and (.25,-.48) ..  (.7,.5);
  \node[fill=mygreen,circle,inner sep=1.5pt] at (-.15,-.35) {};
  \draw[doubleblue] (1,-.75)node[below] {\tiny $2$} -- (1,0);
  \draw[myblue] (.985,0) .. controls (.8,.12) and (.5,.12) .. (.3,.5);
  \draw[myblue] (1.015,0) .. controls (1,.12) and (1.25,.12) ..  (1.3,.5);
  \draw[black] (-.75,-.75)node[below] {\tiny $4$} -- (-.75,.5);
  \draw[myred] (1.75,-.75)node[below] {\textcolor{black}{\tiny $1$}} -- (1.75,.5)node[midway,fill=myred,circle,inner sep=1.5pt] at (-.18,-.25) {};}
+\tfrac{t_{12}}{t_{21}}\tfrac{t_{23}}{t_{32}}
  \tikz[very thick,scale=.8,baseline={([yshift=0.6ex]current bounding box.center)}]{
  \draw[doublegreen] (0,-.75)node[below] {\tiny $3$} -- (0,-.5);
  \draw[mygreen] (-.015,-.5) .. controls (0,-.48) and (-.25,-.48) .. (-.3,.5);
  \draw[mygreen] (.015,-.5) .. controls (0,.-.48) and (.25,-.48) ..  (.7,.5);
  \draw[doubleblue] (1,-.75)node[below] {\tiny $2$} -- (1,0);
  \draw[myblue] (.985,0) .. controls (.8,.12) and (.5,.12) .. (.3,.5);
  \draw[myblue] (1.015,0) .. controls (1,.12) and (1.25,.12) ..  (1.3,.5);
  \node[fill=myblue,circle,inner sep=1.5pt] at (.8,.09) {};  
  \draw[black] (-.75,-.75)node[below] {\tiny $4$} -- (-.75,.5);
  \draw[myred] (1.75,-.75)node[below] {\textcolor{black}{\tiny $1$}} -- (1.75,.5);
  \node[fill=myred,circle,inner sep=1.5pt] at (1.75,-.575) {};}
\]
Since we are only interested in the lowest homological degree we restrict to considering the complex 
\[
 q^{-2}F_tF_{43^{(2)}2^{(2)}1}F_b\mu
\xra{\ \ h\ \ }
qF_tF_{432321}F_b\mu  .  
\]
Finally, applying the isomorphism  $F_tF_{432321}F_b\simeq F_tF_{4332^{(2)}1}F_b$ results in the isomorphic complex
\[
 q^{-2}F_tF_{43^{(2)}2^{(2)}1}F_b\mu
\xra{\ \ 0\ \ }
qF_tF_{432321}F_b\mu  .  
\]

Adding the shift  corresponding to the normalization~\eqref{eq:nrmlzcplx}, and using the fact that 
$F_tF_{43^{(2)}2^{(2)}1}F_b\mu$ is a $\Bbbk$-supervector space of graded dimension $q+q^{-1}$,  yields 
\[
H_{-3}(T) =  q^{-7}\Bbbk \oplus q^{-9}\Bbbk , 
\]
which agrees with the odd Khovanov homology of $T$.

%
%
\section{Further properties of $\fR$}\label{sec:furtherR}

In this section we sketch several of its higher representation theory properties of $\fR$,  
some of them we have used in the previous section.

%
%

\subsection{Supercategorical action on $\cR^\Lambda(k,d)$}\label{sec:CAction}

Given a $\gln$-weight $\Lambda=(\Lambda_1,\dotsc,\Lambda_n)$ we write
$\overline{\Lambda}=(\Lambda_1-\Lambda_2,\dotsc ,\Lambda_{n-1}-\Lambda_{n})$ for the corresponding
$\sln$-weight.
The super algebra  $\overline{\fR}^{\Lambda}(\nu)$ for $\glk$ is defined to be the same as
the superalgebra $\overline{\fR}^{\overline{\Lambda}}(\nu)$ for $\slk$.

We now explain how the bifunctor 
$\Phi\colon \fR \times \fR^\Lambda \to  \fR^\Lambda$ in~\eqref{eq:bifunctor}.
gives rise to an action of $\glk$ on $\cR^\Lambda(k,d)$
for $\Lambda$ a dominant integrable $\glk$-weight of level 2 with
$\Lambda_1+\dotsm \Lambda_n=d$. 
A diagram $D$ in $\cR^\Lambda(k,d)$ with leftmost region labeled $\mu$
defines a web $W_D$ with bottom boundary labeled $\Lambda$ and with top boundary labeled $\mu$. 
We denote $f_i$,  $e_i\in U_q(\glk)$ the Chevalley generators. 

\smallskip

Behind Tubbenhauer's construction in~\cite{tubbenhauer} there is the observation that the transformation  
\begin{equation}\label{eq:daniel-trick}
\tikz[very thick,scale=.25,baseline={([yshift=-.8ex]current bounding box.center)}]{
	\draw (-2,-4)node[below] {\tiny $\textcolor{white}{b}a\textcolor{white}{b}$} to (-2,0);
	\draw [directed=1] (-2,0) to (-2,4)node[above] {\tiny $a+1$};
	\draw (2,-4)node[below] {\tiny $b$} to (2,-2);	\draw (2,-2) to (2,0);
	\draw [directed=1] (2,2.5) to (2,4)node[above] {\tiny $b-1$};
      	\draw (2,0) to (2,2.5);
	\draw [directed=.6] (2,0) to (-2,0);
	\draw [semithick,|->] (5,0) to (7,0);}\mspace{18mu} 
\tikz[very thick,scale=.25,baseline={([yshift=-.8ex]current bounding box.center)}]{
	\draw (-2,-4)node[below] {\tiny $\textcolor{white}{b}a\textcolor{white}{b}$} to (-2,1.5);
	\draw [dotted,directed=1] (-2,1.5) to (-2,4)node[above] {\tiny $0$};
	\draw (2,-4)node[below] {\tiny $b$} to (2,-2);	\draw[->] (2,-2) to (2,0);
	\draw [directed=1] (2,2.5) to (2,4)node[above] {\tiny $a+1$};
      	\draw (2,0) to (2,2.5);
	\draw [rdirected=.55] (2,1.5) to (-2,1.5);	\draw [rdirected=.55] (6,-2) to (2,-2);
	\draw [dotted] (6,-4)node[below] {\tiny $0$} to (6,-2);	\draw (6,-2) to (6,2.5);
       	\draw [->] (6,2.5) to (6,4)node[above] {\tiny $b-1$}; }
\end{equation}
turns any web into a web with all horizontal edges pointing to the right. This goes through the obvious embedding
of $\glk$ into $\mathfrak{gl}_{k+1}$.

\medskip

\n$\bullet$ The generator $f_i$ acts by stacking the web
\begin{equation}\label{eq:fi-action}
\dotsm \tikz[very thick,scale=.25,baseline={([yshift=-0ex]current bounding box.center)}]{
	\draw (-2,-4)node[below] {\tiny $\mu_i$} to (-2,0);
	\draw [directed=1] (-2,0) to (-2,4);
	\draw (2,-4)node[below] {\tiny $\mu_{i+1}$} to (2,-2);	\draw (2,-2) to (2,0);
	\draw [directed=1] (2,2.5) to (2,4);
      	\draw (2,0) to (2,2.5);
	\draw [rdirected=.55] (2,0) to (-2,0);}
\dotsm
\end{equation}
on the top of $W_D$. This means that $f_i$ acts on $\cR^\Lambda(n,d)$ as the functor that adds a strand labeled $i$
to the left of $D$.

\n$\bullet$  To define the action of $e_i$ we stack the web 
\[
\dotsm \tikz[very thick,scale=.25,baseline={([yshift=-0ex]current bounding box.center)}]{
	\draw (-2,-4)node[below] {\tiny $\mu_i$} to (-2,0);
	\draw [directed=1] (-2,0) to (-2,4);
	\draw (2,-4)node[below] {\tiny $\mu_{i+1}$} to (2,-2);	\draw (2,-2) to (2,0);
	\draw [directed=1] (2,2.5) to (2,4);
      	\draw (2,0) to (2,2.5);
	\draw [directed=.6] (2,0) to (-2,0);}
\dotsm
\]
on the top of $W_D$, then we use Tubbenhauer's trick~\eqref{eq:daniel-trick} to put in a form that uses only $F$'s.
The transformation in~\eqref{eq:daniel-trick} is not local and in order to be well defined one needs to keep trace
of the indices before and after acting with an $e_i$. Tubbenhauer's trick gives 
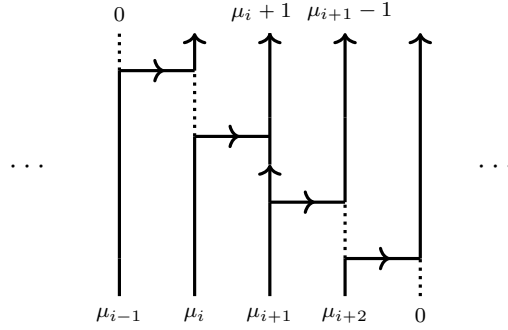
\[
\dotsm\mspace{20mu}
\tikz[very thick,scale=.25,baseline={([yshift=-.8ex]current bounding box.center)}]{
  \draw (-6,-7)node[below] {\tiny $\mu_{i-1}$} to (-6,-5);\draw (-6,-5) to (-6,5);\draw[dotted] (-6,5) to (-6,7)node[above] {\tiny $0$};
    \draw [rdirected=.55] (-2,5) to (-6,5);
  \draw (-2,-7)node[below] {\tiny $\mu_i$} to (-2,1.5);\draw [dotted] (-2,1.5) to (-2,5);\draw [->] (-2,5) to (-2,7);
	\draw (2,-7)node[below] {\tiny $\mu_{i+1}$} to (2,-2);	\draw[->] (2,-2) to (2,0);
	\draw [directed=1] (2,2.5) to (2,7);
      	\draw (2,0) to (2,2.5);
	\draw [rdirected=.55] (2,1.5) to (-2,1.5);\draw [rdirected=.55] (6,-2) to (2,-2);
	\draw (6,-7)node[below] {\tiny $\mu_{i+2}$} to (6,-5);\draw[dotted] (6,-5) to (6,-2);\draw (6,-2) to (6,2.5);
       	\draw [->] (6,2.5) to (6,7);
        \node at (1.55,8.12) {\tiny $\mu_{i}+1$};
        \node at (6.25,8.12) {\tiny $\mu_{i+1}-1$};
        \draw [rdirected=.55] (10,-5) to (6,-5);
\draw[dotted] (10,-7)node[below] {\tiny $0$} to (10,-5);\draw[->] (10,-5) to (10,7);}
\mspace{20mu}\dotsm
\]
Every time we act with an $e_i$ we embed $U_q(\mathfrak{gl}_k)\hookrightarrow U_q(\mathfrak{gl}_{k+1})$
and set
\[
e_i(W_D)=f_{ {1}^{(\mu_1)}\dotsm{i-1}^{(\mu_{i-1})} }
f_i^{(\mu_{i})}f_{i+1}^{(\mu_{i+1}-1)}
f_{ {i+2}^{(\mu_{i+2})} \dotsm {k}^{(\mu_k)}}
  (\mu,0)(W_D) . 
\] 
After being  acted with an $e_j$, $f_i$ acts on $W_D$ through the web corresponding to $f_{i+1}(\mu,0)$.

We define the action of $e_i$ on $\cR^\Lambda(k,d)$ as the superfunctor that adds 
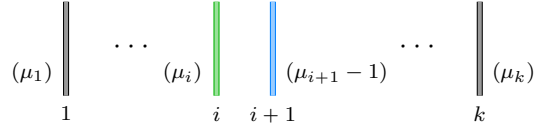
\[
  \tikz[very thick,scale=1,baseline={([yshift=0ex]current bounding box.center)}]{
\draw[doubleblack] (-2.5,-.75)node[below] {\tiny $1$}  --  (-2.5,.5);\node at (-2.95,-.45) {\tiny $(\mu_{1})$};
\node at (-1.6,-.1) {$\dotsm$};
\draw[doublegreen] (-.5,-.75)node[below] {\tiny $i$}  --  (-.5,.5);\node at (-.95,-.45) {\tiny $(\mu_{i})$};
\draw[doubleblue] (.25,-.75)node[below] {\tiny $i+1$}  -- (.25,.5);\node at (1.1,-.45) {\tiny $(\mu_{i+1}-1)$};
\node at (2.2,-.1) {$\dotsm$};
\draw[doubleblack] (3,-.75)node[below] {\tiny $k$}  --  (3,.5);\node at (3.45,-.45) {\tiny $(\mu_{k})$};}
  \]
  to the left of $D$ (here $(\mu_{1})$, etc..., are the thicknesses) that is,  
 we act with the identity 2-morphism of
  $F_{ {1}^{(\mu_1)}\dotsm{i-1}^{(\mu_{i-1})} }F_i^{(\mu_{i})}F_{i+1}^{(\mu_{i+1}-1)}F_{ {i+2}^{(\mu_{i+2})} \dotsm {k}^{(\mu_k)}}(\mu,0)$.

Denote $\Phi(e_i)$ and $\Phi(f_i)$ the morphisms in $\fR^\Lambda$ that act as endofunctors of $\cR^\Lambda(n,d)$
through the action above.  
It is clear that $\Phi(uv)=\Phi(u)\Phi(v)$ for $u$, $v\in U_q(\glk)$. 
Note that $\Phi(1)(\mu)$ is a canonical element $F_{can}(\mu)$ as introduced in~\eqref{eq:canF}.

\begin{lem}\label{lem:isoPhi}
We have natural isomorphisms 
\begin{align*}
  \Phi(e_i)\Phi(f_i)(\lambda) &\simeq   \Phi(f_i)\Phi(e_i)(\lambda)\oplus \Phi(1)^{\oplus[\overline{\lambda}_i]}(\lambda)
& \text{if }\ \overline{\lambda}_i\geq 0, 
\\[1ex]
\Phi(f_i)\Phi(e_i)(\lambda) &\simeq
\Phi(e_i)\Phi(f_i)(\lambda)\oplus \Phi(1)^{\oplus[-\overline{\lambda}_i]}(\lambda)
& \text{if }\ \overline{\lambda}_i\leq 0.  
\end{align*}
\end{lem}
\begin{proof}
These are instances of the categorified higher Serre relations. 
Denote $F_u=F_{ {1}^{(\lambda_1)}\dotsm{i-1}^{(\lambda_{i-1})} }$ and $F_d=F_{ {i+2}^{(\lambda_{i+2})} \dotsm {k}^{(\lambda_k)}}$. 
  We have
\begin{align*}
    \Phi(e_i)\Phi(f_i)(\lambda) &=
F_u
F_i^{(\lambda_{i}-1)}F_{i+1}^{(\lambda_{i+1})}
F_d
    F_i(\lambda,0)
    \\ &\simeq
F_u
F_i^{(\lambda_{i}-1)}F_{i+1}^{(\lambda_{i+1})}
F_i(\dotsc,\lambda_i,\lambda_{i+1},0,\lambda_{i+2},\dotsc)
F_d ,
(\lambda,0) , 
\intertext{and} 
  \Phi(f_i)\Phi(e_i)(\lambda) &=
 F_tF_{i+1}F_i^{(\lambda_i)}F_{i+1}^{(\lambda_{i+1}-1)}F_b(\lambda,0) , 
\end{align*}  
and therefore, it is enough to check that the relations above are satisfied by the superfunctors 
$F_i^{(\lambda_{i}-1)}F_{i+1}^{(\lambda_{i+1})}F_i(\lambda_i,\lambda_{i+1},0)$ 
and $F_{i+1}F_i^{(\lambda_i)}F_{i+1}^{(\lambda_{i+1}-1)}(\lambda_i,\lambda_{i+1},0)$.
Suppose $\lambda_i\geq \lambda_{i+1}$. Then we have $\lambda_i\in\{1,2\}$ and $\lambda_{i+1}\in\{0,1\}$.   
The computations involved are rather simple and we can check the four cases separately.
\begin{enumerate}
\item $(\lambda_i,\lambda_{i+1})=(1,0)$:
  \begin{align*}
\Phi(e_i)\Phi(f_i)(\lambda) &=
  F_i^{(\lambda_{i}-1)}F_{i+1}^{(\lambda_{i+1})}F_i(\lambda_i,\lambda_{i+1})= F_i(1,0)=0\oplus F_{can}(1,0),
  \\ &=
   \Phi(f_i)\Phi(e_i)(\lambda) \oplus \Phi(1)(\lambda). 
  \end{align*}
  
\item  $(\lambda_i,\lambda_{i+1})=(1,1)$: 
  \begin{align*}
\Phi(e_i)\Phi(f_i)(\lambda) &=
  F_{i}F_{i+1}(1,1,0)= \Phi(f_i)\Phi(e_i)(\lambda). 
  \end{align*}

\item  $(\lambda_i,\lambda_{i+1})=(2,0)$: 
  \begin{align*}
\Phi(e_i)\Phi(f_i)(\lambda) &=
F_iF_i(2,0,0) \simeq q F_i^{(2)}(2,0,0) + q^{-1}F_i^{(2)}(2,0,0) = \Phi(1)^{\oplus [2]}(\lambda). 
  \end{align*}

\item  $(\lambda_i,\lambda_{i+1})=(2,1)$: 
  \begin{align*}
\Phi(e_i)\Phi(f_i)(\lambda) &=
 F_iF_{i+1}F_i(2,1,0) \simeq 0 \oplus F_i^{(2)}F_{i+1}(2,1,0)
=  \Phi(f_i)\Phi(e_i)(\lambda)\oplus \Phi(1)(\lambda). 
  \end{align*}  
  \end{enumerate}
An this proves the first isomorphism in the statement. The second isomorphism can be checked using the same
method. 
\end{proof}

The proof of~\fullref{lem:isoPhi} uses several supernatural transformations between the various compositions of
$\Phi(f_i)(\lambda)$ and $\Phi(e_i)(\lambda)$ and $\Phi(1)(\lambda)$ that can be given a presentation in terms of the
diagrams from $\fR$. We act with such diagrams by stacking them on the top of the diagrams for the image
of $\Phi$. 
On the weight space $(1,1)$ these maps coincide with the maps used to define
the chain complex for a tangle diagram in the previous section.
In the general case these maps are units and co-units of adjunctions in the following.
\begin{lem}
Up to degree shifts, the functor $\Phi(e_i)$ is left and right adjoint to $\Phi(f_i)$.  
\end{lem}

\begin{lem}
  We have the following natural isomorphisms:
\begin{align*}
  \Phi(e_j)\Phi(f_i)(\lambda) &\simeq \Phi(f_i)\Phi(e_j)(\lambda) &\text{for }\ i\neq j ,
  \\
  \Phi(f_i)\Phi(f_{i\pm 1}) \Phi(f_i)(\lambda) &\simeq \Phi(f_{i}^{(2)})\Phi(f_{i\pm 1})(\lambda) \oplus \Phi(f_{i\pm 1})\Phi(f_{i}^{(2)})(\lambda) , 
  \\
  \Phi(e_i)\Phi(e_{i\pm 1})(\lambda) \Phi(e_i) &\simeq \Phi(e_{i}^{(2)})\Phi(e_{i\pm 1})(\lambda)\oplus\Phi(e_{i\pm 1})\Phi(e_{i}^{(2)})(\lambda) . 
\end{align*}
\end{lem}

\begin{proof}
The proof consists of a case-by-case computation.
We illustrate the proof with the case of $\Phi(e_i)\Phi(f_{i+1})(\lambda)\simeq \Phi(f_{i+1}) \Phi(e_i)(\lambda)$ and leave the
rest to the reader.
We have
\begin{align*}
  \Phi(e_i)\Phi(f_{i+1})(\lambda) &= F_i^{(\lambda_i)}F_{i+1}^{(\lambda_{i+1}-2)}F_{i+2}^{(\lambda_{i+2}+1)}F_{i+1}(\lambda) ,
  \intertext{and}
  \Phi(f_{i+1}) \Phi(e_i)(\lambda) &= F_i^{(\lambda_i)}F_{i+2}F_{i+1}^{(\lambda_{i+1}-1)}F_{i+2}^{(\lambda_{i+2})}(\lambda) ,
\end{align*}  
which are zero  unless $\lambda_{i+1}=2$ and $\lambda_{i+2}\in\{0,1\}$. 
If $\lambda_{i+1}=2$ these can be written 
\begin{align*}
  \Phi(e_i)\Phi(f_{i+1})(\lambda) &= F_i^{(\lambda_i)}F_{i+2}^{(\lambda_{i+2}+1)}F_{i+1}(\lambda) ,
  \intertext{and}
  \Phi(f_{i+1}) \Phi(e_i)(\lambda) &= F_i^{(\lambda_i)}F_{i+2}F_{i+1}F_{i+2}^{(\lambda_{i+2})}(\lambda) .
\end{align*}  
The case $\lambda_{i+2}=0$ is immmediate and the case $\lambda_{i+2}=1$ follows from the Serre relation  
\eqref{eq:R3serre}-\eqref{eq:dumbelXing}. 
\end{proof}

As explained in~\cite[Sections 1.5 and 6]{BrundanEllis-Monoidal} the Grothendieck group of a ($\bZ$-graded)
monoidal supercategory is a  $\bZ[q^{\pm 1},\pi]/ (\pi^2-1)$-algebra. 
Nontrivial parity shifts will occur when applying Tubbenhauer's trick. 
All the above can be used to prove the following. 
\begin{thm}
The assignment above defines an action of $U_q(\glk)$ on $\cR^\Lambda(k,d)$. 
With this action we have an isomorphism of $K_0(\cR^\Lambda(k,d))$ with the irreducible, finite-dimensional,  
$U_q(\glk)$-representation of highest weight $\Lambda$ at $\pi=1$.
\end{thm}



\vspace*{1cm}



\begin{thebibliography}{15}

\bibitem{BrundanEllis-Monoidal} 
  J.~Brundan and A.~Ellis, 
  Monoidal supercategories.
  \emph{Comm. Math. Phys.} 351 (2017), no. 3, 1045-1089.
  
\bibitem{Egilmez-Lauda}
  I.~Egilmez and A.~Lauda,
DG structures on odd categorified quantum $sl(2)$. 
Preprint 2018.
arXiv: 1808.04924 [math.QA]
  
\bibitem{EKL}
A.~Ellis, M.~Khovanov and A.~Lauda, 
The odd nilHecke algebra and its diagrammatics.   
\emph{Int. Math. Res. Not.} IMRN 2014, no. 4, 99-1062. 

\bibitem{EllisLauda}
A.~Ellis and A.~Lauda, 
An odd categorification of $U_q(\slt)$.   
\emph{Quantum Topol.} 7 (2016), no. 2, 329-433. 

\bibitem{EllisQi}
A.~Ellis and Y.~Qi, 
The differential graded odd nilHecke algebra.
\emph{Comm. Math. Phys.} 344 (2016), no. 1, 275-331. 

\bibitem{KKO}
  S-J.~Kang, M.~Kashiwara and S-J.~Oh,
  Supercategorification of quantum Kac-Moody algebras.
  \emph{Adv. Math.} 242 (2013), 116-162.

\bibitem{KKO2}
  S-J.~Kang, M.~Kashiwara and S-J.~Oh,
  Supercategorification of quantum Kac-Moody algebras II.
  \emph{Adv. Math.} 265 (2014), 169-240.
  
\bibitem{KL1}
M.~Khovanov and A.~Lauda,
A diagrammatic approach to categorification of quantum groups I.
\emph{Represent. Theory} 13 (2009), 309-347.

\bibitem{KL3}
M.~Khovanov and A.~Lauda,
A categorification of quantum $sl(n)$.
\emph{Quantum Topol.} 1 (2010), no. 1, 1-92. 

\bibitem{KLMS}
  M.~Khovanov, A.~Lauda, M.~Mackaay and M.~Sto\v si\'c,
  Extended graphical calculus for categorified quantum $\mathfrak{sl}(2)$. 
\emph{Mem. Amer. Math. Soc.} 219 (2012), no. 1029, vi+87 pp.

\bibitem{LaudaRussell}
A.~Lauda, and H.~Russell, 
Oddification of the cohomology of type $A$ Springer varieties.  
\emph{Int. Math. Res. Not.} IMRN 2014, no. 17, 4822-4854.

\bibitem{LQR}
A.~Lauda, H.~Queffelec and D.~Rose,
Khovanov homology is a skew Howe 2-representation of categorified quantum $\mathfrak{sl}_m$.
\emph{Algebr. Geom. Topol.} 15 (2015) 2517-2608.

\bibitem{Lu}
G.~Lusztig,
Introduction to quantum groups. 
\emph{Prog. in Math.} 110, Birkh\"{a}user, 1993.

\bibitem{NV-OHn}
G.~Naisse and P.~Vaz,
Odd Khovanov's arc algebra.
\emph{Fund. Math.} 241 (2018), no. 2, 143-178. 

\bibitem{ORS}
  P.~Ozsvath, J.~Rasmussen, and Z.~Szab\'o,
  Odd Khovanov homology.
\emph{Alg. Geom. Topol.} 13 (2013), 1465-1488.

\bibitem{putyra}
K.~Putyra,
\emph{A 2-category of chronological cobordisms and odd Khovanov homology}.
 Banach Center Publ., 103 (2014),  291-355. 

\bibitem{putyraThesis}
K.~Putyra,
On a triply-graded generalization of Khovanov homology.
PhD Thesis - Columbia University. 2014. 122 pp. 

  
\bibitem{R1}
R.~Rouquier, 2-Kac--Moody algebras. 
arXiv:0812.5023.

\bibitem{shumakovitch}
A.~Shumakovitch, 
Patterns in odd Khovanov homology. 
\emph{J. Knot Theory Ramifications} 20 (2011), no. 1, 203-222. 

\bibitem{shumakovitch-torsion}
Torsion of Khovanov homology. 
\emph{Fund. Math.} 225 (2014), no. 1, 343-364.  


\bibitem{stosic}
M.~Sto\v si\'c, 
On extended graphical calculus for categorified quantum $\sln$. 
\emph{J. Pure Appl. Algebra} 223 (2019), no. 2, 691-712. 

\bibitem{tubbenhauer}
  D.~Tubbenhauer.
  $\sln$-webs, categorification and Khovanov--Rozansky homologies.  
  arXiv: 1404.5752v2 [math.QA] (2014).


 \bibitem{webster}
   B.~Webster,
   Knot invariants and higher representation theory.
   \emph{Mem. Amer. Math. Soc.} 250 (2017), no. 1191, v+141 pp.

   
\end{thebibliography}
\end{document}